\documentclass[11pt]{article}

\usepackage{amssymb}

\usepackage{latexsym}
\usepackage{amsmath}
\usepackage{mathrsfs}
\usepackage{upgreek}
\usepackage{mathabx}
\usepackage{esint}
\usepackage{xcolor}
\usepackage{bbm}
\usepackage{hyperref}
\usepackage{empheq}
\usepackage{cleveref}
\usepackage{mathtools}
\usepackage{amsthm}
\usepackage[title]{appendix}
\usepackage{amsfonts}
\usepackage{relsize}
\usepackage{esvect}
\usepackage{enumerate}
\usepackage{cases}
\usepackage[left=22mm,
  right=22mm,
  top=20mm,
  bottom=22mm]{geometry}
\makeatletter

\newcommand{\Rmn}[1]{\expandafter\@slowromancap\romannumeral #1@}
\makeatother

\usepackage{graphicx}
\usepackage[left,modulo]{lineno}
\newtheorem{thm}{Theorem}[section]

  \newenvironment{hproof3}{%
  \proof}{\endproof}
  
  \newenvironment{hproof5}{%
  \proof}{\endproof}
 \newtheorem{Th}[thm]{Theorem}

 \newtheorem{Prop}[thm]{Proposition}
 \newtheorem{Co}[thm]{Corollary}
\newtheorem{Lm}[thm]{Lemma}

\newtheorem{Dfi}[thm]{Definition}
\newtheorem{Rm}[thm]{Remark}

\numberwithin{equation}{section}
\newcommand{\be}{\begin{equation}}
\newcommand{\ee}{\end{equation}}
\newcommand{\bg}{\begin{gather}}
\newcommand{\eg}{\end{gather}}
\newcommand{\ba}{\begin{align}}
\newcommand{\ea}{\end{align}}
\newcommand{\bad}{\begin{aligned}}
\newcommand{\ead}{\end{aligned}}
\newcommand{\R}{\mathbb{R}}
\newcommand{\N}{\mathbb{N}}

\newcommand{\mca}[1]{\mathcal{#1}}

\newcommand{\nf}{\infty}

\def\avint{\mathop{\mathchoice{\,\rlap{-}\!\!\int}
                              {\rlap{\raise.15em{\scriptstyle -}}\kern-.2em\int}
                              {\rlap{\raise.09em{\scriptscriptstyle -}}\!\int}
                              {\rlap{-}\!\int}}\nolimits}

\def\avint{\mathop{\,\rlap{-}\!\!\int}\nolimits}

\def\ovwe{\text{\larger[1.5]$\we$}}

\def\XXint#1#2#3{{\setbox0=\hbox{$#1{#2#3}{\int}$ }
\vcenter{\hbox{$#2#3$ }}\kern-.6\wd0}}

\usepackage{dsfont}
\def\mkt{\mkern2mu}
\def\mko{\mkern1mu}

\def\ga{\gamma}
\def\La{\Lambda}

\def\lf{\left}

\def\ot{\otimes}
\def\de{\delta}
\def\rg{\right}

\def\al{\alpha}

\def\wti{\widetilde}

\newcommand{\we}{\wedge}

\def\dil{\textup{dil}}
\def\rot{\textup{rot}}
\def\std{\textup{std}}
\def\sym{\textup{sym}}
\def\bwe{\textstyle\bigwedge}
\def\loc{\textup{loc}}

\def\vae{\varepsilon}

\def\vp{\varphi}

\def\Om{\Omega}
\def\II{\mathrm{I\!I}}
\def\bII{\vec{\II}}
\def\om{\omega}
\def\p{\partial}
\def\bn{\vec{n}}

\def\bH{\vec{H}}

\def\bv{\vec{v}}
\def\bvt{\vec \vartheta}
\def\vt{\vartheta}
\def\bg{\vec{g}}

\def\g{\nabla}

\def\lan{\langle}
\def\ran{\rangle}
\def\bL{\vec{L}}
\def\bR{\vec{R}}

\def\bw{\vec{w}}

\def\vet{\vec \eta}

\def\bc{\vec{c}}
\def\vp{\varphi}
\def\bP{\vec{\phi}}
\def\bP{\vec{\Phi}}
\def\dvol{d\textup{vol}}

\def\bul{\bullet}
\def\si{\sigma}
\def\Si{\Sigma}

\newcommand{\res}{\mathbin{\hbox{\vrule height 5pt width .4pt depth 0pt\vrule height .6pt width 4pt depth 0pt}}} 
\newcommand{\metricsub}[3]{_{
    \mkern-#3mu
    \smash{\lower #1\hbox{$\scriptscriptstyle #2$}}
}}
\newcommand{\resg}{\res_g}
\newcommand{\sbul}{\mathbin{
  \mathchoice{\scalebox{0.7}{$\bullet$}}{\scalebox{0.7}{$\bullet$}}%
             {\scalebox{0.7}{$\bullet$}}{\scalebox{0.7}{$\bullet$}}}
}
\newcommand{\bulg}{\sbul_g}
\newcommand{\dwe}{\mathbin{\dot{\wedge}}}
\newcommand{\dres}{\mathbin{\dot{\res}_g}}
\newcommand{\wres}{\mathbin{\ovs{\ovwe}{\res}_{\mkern -4mu g}}}
\newcommand{\rres}{\mathbin{\ovs{\res}{\res}_g}}

\newcommand{\ov}[1]{\overline{#1}}

\newcommand{\ovs}[2]{\overset{#1}{#2}}

\DeclareMathOperator{\supp}{supp}

\newcommand{\lap}{\Delta}
\newcommand{\Imm}{\mathrm{Imm}}
\usepackage{authblk}
\author[1]{Yann Bernard}
\author[2]{Tian Lan}
\author[3]{Dorian Martino}
\author[4]{Tristan Rivi\`ere}
\affil[1]{School of Mathematics, Monash University, 3800 Victoria, Australia}
\affil[2,3,4]{Department of Mathematics, ETH Zürich, 101 Rämistrasse, 8092 Zürich, Switzerland}

\title{The Regularity of Critical Points to Scale-Invariant Curvature Energies in Dimension 4}
\date{\today}

\begin{document}

\maketitle

\begin{abstract}
We consider a class of scale-invariant curvature energies defined on immersed $4$-dimensional manifolds and prove that weak immersions that are critical points of such energies are analytic in any given local harmonic chart. Because of the criticality of this variational problem, the regularity result is obtained through the identification of conservation laws by applying Noether theorem. The resulting identities generate a lower order elliptic system of PDEs to which methods from integrability by compensation and interpolation theory are applied.
\end{abstract} 


\section{Introduction}\label{sec:intro}

The study of critical points of scale-invariant Lagrangians, such as the Dirichlet energy of maps from a given surface into a given Riemannian manifold \cite{H2002}, the Willmore energy \cite{LaMaRi25,Ri16} or the Yang--Mills functional \cite{Ri20}, usually requires specific tools from integration by compensation and compensated compactness. 
In this context, the scale-invariance property is generally related to some critical Sobolev space $W^{s,p}(B^n)$ with $sp=n$ and $p>1$, which embeds into all $L^q$ spaces for $q<\infty$, but not into $L^{\infty}$. However, the geometric context provides a special structure in the Euler--Lagrange systems via the existence of a Coulomb gauge or the conservation laws induced by Noether theorem. In order to exploit this extra-structure, one needs to use refinement of Lebesgue spaces in order to be able to apply elliptic regularity, such as the Hardy space $\mathcal{H}^1$ for harmonic maps and Willmore surfaces (i.e. 2-dimensional problems), or Lorentz spaces for Yang--Mills connections (i.e. 4-dimensional problem). An additional difficulty arises from the tensorial analysis, meaning that we are studying non-linear systems and not equations, where specific tools such as the maximum principle can be decisive. In this article, we consider the analysis of scale-invariant Lagrangians defined on immersed 4-dimensional manifolds. \\

An example of such a Lagrangian appears in the computation of the renormalized volume of 5-dimensional minimal submanifolds of the $(m+1)$-dimensional hyperbolic space with boundary at infinity in $\R^m$, as independently proved by Graham--Reichert \cite{graham2020} and Zhang \cite{zhang} in the context of AdS/CFT correspondence. The computation of the renormalized volume of minimal surfaces has been introduced by Graham--Witten \cite{graham1999}, where they prove that such a computation leads to a conformally invariant functional on closed submanifolds of any even dimension $n\geq 2$ in $\R^m$ with $m\geq n+1$. In the case of 4-dimensional submanifolds $\Sigma^4$ of $\R^m$, Graham--Reichert \cite{graham2020} and Zhang \cite{zhang} obtained the following functional defined for any immersion $\bP \in \textup{Imm}(\Sigma^4,\R^m)$,
\begin{equation}\label{eq:GR}
    \mathcal{E}_{GR}(\bP)\coloneqq \int_{\Sigma^4} \Big(\big|\pi_{\bn_{\bP}}d \vec H_{\bP} \big|^2_{g_{\bP}} - \big| \bH_{\bP}\cdot \bII_{\bP} \big|^2_{g_{\bP}} + 7\, \big| \bH_{\bP} \big|^4 \Big)\, d\textup{vol}_{g_{\bP}}.
\end{equation}
In the above expression, given $\bP\in \textup{Imm}(\Sigma^4,\R^m)$, we denoted $g_{\bP}$ the first fundamental form of $\bP$, $\bII_{\bP}$ its second fundamental form, $\bH_{\bP} \coloneqq \frac{1}{4} \textup{tr}_{g_{\bP}}(\bII_{\bP})$ the mean curvature vector, and $\pi_{\bn_{\bP}}$ the normal projection. Such conformally invariant functionals whose Euler--Lagrange equation has a leading order term of the form $\Delta_{g_{\bP}}^2 \vec{H}_{\bP}$ were \textit{defined} as generalized Willmore functionals by Gover--Waldron \cite{gover2020}, see also \cite{gover2017,blitz2023generalized,graham2017}. Additional examples of generalized Willmore energies have recently been constructed in \cite{AG2024,M2024}. The goal of this work is to develop the regularity theory for critical points of such Lagrangians, where the invariance by choice of coordinates (due to the geometric meaning of the functional) constitutes a non-negligible difficulty. \\

As the principal part of $ \mathcal{E}_{GR}$ (see~\eqref{EGRmod}), the Dirichlet energy of the mean curvature also appears in computer-aided design and computer-aided manufacturing, where a typical problem is to find a robust way of designing fair surfaces (highly regular surfaces). To that extent, various approaches based on variational principles have been explored, see for instance \cite{hagen1993,greiner1994,schneider2000,welch1992,xu2006}. In \cite{du2005,welch1992,xu2006}, different possibilities based on the mean curvature were studied, such as the Willmore energy (the $L^2$ norm of the mean curvature) or the Canham--Helfrich energy. In \cite{moreton1992}, Moreton--Séquin proposed to consider functionals containing derivatives of the second fundamental form in order to impose an a priori higher regularity on the surface. This approach has later been developed by You--Comninos--Zhang \cite{you2004}, Xu--Zhang \cite{xu2006}, and You--Zhang \cite{zhang2004}, which led to the numerical study of the Dirichlet energy of the mean curvature (leading to a sixth-order Euler--Lagrange equation). One of the difficulties of the development of this approach is the absence of a theoretical framework to study the analytical properties of such a nonlinear equation on submanifolds.\\

A first step in this direction has been performed by Caffarelli--Stinga--Vivas \cite{caffarelli2024}, where they introduced the question of the regularity of critical points of the Dirichlet energy of the mean curvature (here $n\geq 1$ is arbitrary):
\begin{equation}\label{eq:Dir_n}
     \vec{\Phi}\in \Imm(B^n,\R^{n+1}) \mapsto \int_{B^n} |d H_{\bP}|^2_{g_{\bP}}\ d\textup{vol}_{g_{\bP}}.
\end{equation}
It can be shown (see \cite{xu2006}) that the Euler--Lagrange equation of $E$ for $n=2$ is of the following form
\begin{equation}\label{eq:EL_Dirichlet}
    \begin{cases}
        \lap_{g_{\bP}}^2 H_{\bP}+P(g_{\bP},\II_{\bP},\mko d H_{\bP},\mko \Delta_{g_{\bP}} H_{\bP})=0, \\[2mm]
        \big|P(g_{\bP},\II_{\bP},\mko d H_{\bP},\mko \Delta_{g_{\bP}} H_{\bP})\big| \leq C(g_{\bP})\mko \big( |\II_{\bP}|^2_{g_{\bP}}\mko |\Delta_{g_{\bP}} H_{\bP}| + |\II_{\bP}|_{g_{\bP}}\mko |dH_{\bP}|_{g_{\bP}}^2 \big).
    \end{cases}
 \end{equation}
 In order to simplify the setting, the authors in \cite{caffarelli2024} restricted themselves to the setting of immersions $\bP$ describing graphs of functions over a given bounded $C^2$ open set $\Omega\subset \R^n$. 
For a given $h\in C^{1,\al}(\ov\Om)$, they produced $u\in C^{2,\al}(\ov \Om)$ and $H\in C^{1,\al}(\ov \Om)$ such that $H$ is the scalar mean curvature of $\bP_u(x)=(x,u(x))$ and, for this fixed $u$, the function $H$ solves the Dirichlet minimization problem (denoting $g_u=\bP_u^* \,g_{\std}$)
\[
    \min_{H-h \,\in\, W^{1,2}_0(\Om)}  \frac 12\int_{\Om}|dH|_{g_{u}}^2 \,\dvol_{g_{u}}.
\]
In other words, instead of solving \eqref{eq:EL_Dirichlet}, the authors obtained a solution of $\Delta_{g_{u}} H=0$, a fourth-order equation instead of the sixth-order equation that was looked for in \cite{moreton1992,xu2006,graham2020}. Moreover, they established existence of minimizers within an admissible set in which $H$ is the mean curvature of $\bP_u$, the boundary values $H\big|_{\p\Om}$ and $u\big|_{\p \Om}$ are prescribed, and $\|H\|_{W^{1,\nf}(\Om)}+\|H\|_{W^{2,2}(\Om)}\le C_0$ for a constant $C_0>0$ depending only on $\Om$ and $h\big|_{\p\Om}$.  \\

In this paper, we aim at providing a general framework for the analysis of functionals such as \eqref{eq:GR} or \eqref{eq:Dir_n} without any a priori assumption on whether the immersion $\bP$ parametrizes a graph or not, or whether one has a priori $C^2$ bounds. To do so, we start from the following heuristic argument. If we consider that $\bH_{\bP} = 4^{-1}\, \Delta_{g_{\bP}} \bP$, and we forget (as in \cite{caffarelli2024}) the dependence between $\bP$ and $g_{\bP}$ for a moment, then the energy \eqref{eq:Dir_n} should provide a bound on the $W^{3,2}(B^n)$ norm of $\bP$. By Sobolev embeddings, if $n\leq 3$, then we freely obtain estimates of the form $g_{\bP}\in C^0$ and $\vec{n}_{\bP}\in C^0$ (the Gauss map), i.e. we are in a subcritical setting where one can expect to be able to find minimizers and to obtain regularity estimates by standard arguments of calculus of variations. This is not valid anymore in dimensions $n\geq 4$ and one needs to develop a refined analysis in this setting. Pursuing this heuristic, one can expect to still be able to recover some regularity properties in dimension $n=4$. Indeed, this is the critical setting where an immersion $\bP$ with a priori $W^{3,2}(B^4)$ estimates verifies that the induced metric $g_{\bP}$ has coefficients in $W^{2,2}(B^4)$, and thus in $\text{VMO}(B^4)$ (that is to say, almost $C^0$). Hence, we will restrict ourselves to the case $n=4$ in the remainder of the article.\\

As explained above, the natural variational framework in which the Lagrangian \eqref{eq:Dir_n} can be studied is not the one of $C^3$ immersions, but the one of {\it weak immersions} in the critical Sobolev Space $W^{3,2}(\Sigma^4,{\R}^m)$ ($m\ge 5$) introduced successively by the fourth author in 2-dimension in \cite{Riv14, Ri16} and more recently by the last two authors in arbitrary dimensions in \cite{MarRiv2}. We adopt the following definition.
\begin{Dfi}\label{defweakimm}
    Let $(\Sigma^4,h)$ be a 4-dimensional closed oriented Riemannian manifold. We define
    \begin{align*}
		\mathcal{I}_{1,2}(\Sigma^4,\R^m)\coloneqq \left\{
		\vec{\Phi}\in W^{3,2}(\Sigma^4,\R^m) \colon 
			\displaystyle \exists\, c_{\vec{\Phi}}\ge 1,\ c^{-1}_{\vec{\Phi}} h \leq g_{\vec{\Phi}} \leq c_{\vec{\Phi}}\, h
		\right\}.
	\end{align*} 
    Here $g_{\bP}=\bP^* g_{\std}$ denotes the metric induced by the weak immersion $\bP$. We define $\mathcal{I}_{1,2}(B^4,\R^m)$ analogously, replacing $h$ by the Euclidean metric on $B^4$.
    \end{Dfi}

By definition, a weak immersion in $\mathcal I_{1,2}(\Sigma^4,\R^m)$ might not be $C^1$ and therefore does not necessarily locally define a graph. This justifies the denomination ``weak''.\footnote{Observe nevertheless that the hypothesis $\vec{\Phi}\in W^{3,2}(\Sigma^4,\R^m)$ implies by Poincar\'e inequality that
$\nabla\vec{\Phi}$ is in the Vanishing Mean Oscillation Space (VMO) which is the closure of $C^0$ functions for the Bounded Mean Oscillation norm and by being a ``weak immersion'' one is ``almost'' $C^1$.} Using $g_{\bP}$-harmonic coordinates (that is, coordinates $(z^i)_{1\le i\le 4}$ with $\lap_{g_{\bP}} z^i=0$), the last two authors proved in~\cite{MarRiv2} that any $\bP\in \mathcal I_{1,2}(\Sigma^4,\R^m)$ induces a $C^1$ differential structure on $\Sigma^4$, that is to say an atlas of coordinates on $\Sigma^4$ composed of harmonic charts in which the coefficients of the metric $g_{\bP}$ are $C^0$ and the transition charts are $C^1$. \\

Inspired by the setting of generalized Willmore energies, we study geometric energies of the following type, where $F$ is a real polynomial on $(g^{ij})$ and the ambient coordinates of $(\bII_{ij})$:
\begin{align}\label{firdefE}
    \bP\in \mathcal{I}_{1,2}(\Sigma^4,\R^m) \mapsto \int_{\Si} \Big(| d\bH|_g^2+F(g,\bII)\Big)\, \dvol_g.
\end{align}
We are interested in the case when the functional in \eqref{firdefE} is invariant under dilations and rotations in the ambient space, and invariant under reparametrization. In order for $F(g,\bII)\,\dvol_g$ to be pointwise invariant under dilations and reparametrization, the polynomial $F$ must be homogeneous of degree $4$ in the components of $g^{-1}\bII$, see \cite{mondino2018,Alexakis12}. More precisely, $F(g,\bII)$ is a linear combination of \textit{geometric complete contractions} of the form
\begin{equation}\label{comcon}
    \text{contr}(\bII_{i_1j_1}\ot \bII_{i_2j_2}\ot\bII_{i_3j_3}\ot \bII_{i_4j_4} ).
\end{equation}
In addition, if $F(g,\bII)\,\dvol_g$ is pointwise invariant under rigid motions, then by the first fundamental theorem of invariant theory (see for instance~\cite[Thm.~2.9.A]{weyl97} and~\cite[Appendix~A]{Bailey94}), the contractions in~\eqref{comcon} are generated by inner products in $\R^m$ among $\big(\bII_{i_kj_k}\big)$. By direct computations, we obtain that $F$ must be a linear combination of the following types:\footnote{Instead of polynomials, one may consider for instance the function $f(g,\bII)=|\bH|\cdot|\bII|_g^3$, in which case $f(g,\bII)\,\dvol_g$ is also invariant under dilations and rotations in the ambient space. However, since $f$ is not smooth at $\bH=0$, the Euler--Lagrange operator associated to $\int f(g,\bII)\,\dvol_g$ may not be smooth even for smooth immersions.  }
\begin{equation}\label{typesPs}
    \begin{dcases}
        P_1(g,\bII)= \big|\bH\big|^4,\qquad& P_2(g,\bII)=\big|\bH\big|^2\ \big|\bII\big|_g^2,\\[2mm]
        P_3(g,\bII)= \big| \bH\cdot \bII \big|_g^2,\qquad& P_4(g,\bII)= \big|\bII\big|_g^4,\\[2mm]
        P_5(g,\bII)= \Big(\bII^i_j\cdot \bII^k_\ell\Big)\ \Big(\bII^j_i\cdot \bII^\ell_k \Big),\qquad& P_6(g,\bII)=\Big(\bH\cdot \bII^i_j\Big)\ \Big(\bII^j_k\cdot \bII^k_i \Big),\\[2mm]
        P_7(g,\bII)= \Big(\bII^i_j\cdot \bII^k_\ell \Big) \ \Big(\bII^j_k\cdot \bII^\ell_i\Big),\qquad& P_8(g,\bII)= \Big(\bII^i_j\cdot \bII^j_k\Big)\ \Big(\bII^k_\ell\cdot \bII^\ell_i\Big).
    \end{dcases}
\end{equation}
In this setting, we can rewrite the functional \eqref{firdefE} in the following form. Given $\vec c=(c_1,\ldots,c_8)\in \R^8$, we define
\begin{align}\label{defE}
   \forall \bP\in \mathcal{I}_{1,2}(\Sigma^4,\R^m),\qquad E_{\vec c}(\bP)=\int_{\Si^4} \bigg(| d\bH|_g^2+\sum_{s=1}^8 c_s P_s(g,\bII)\bigg) \, \dvol_g.
\end{align}
Let $\pi_T$ be the orthogonal projection onto the tangent bundle of $\bP(\Si^4)\subset \R^m$. Since $\pi_T\mko \p_i\bH= -\big(\bH\cdot \bII^j_i\big)\mko \p_j \bP$, we have
\begin{align*}
   |\pi_{\bn} \mko d\bH|_g^2= |d\bH|^2_g-|\pi_T\mko d\bH|_g^2=|d\bH|_g^2-P_3(g,\bII).
\end{align*}
It follows that
\begin{equation}\label{EGRmod}
    \mathcal{E}_{GR}(\bP)=\int_{\Si^4} \Big(| d\bH|_g^2-2\mko P_3(g,\bII) +7\mko P_1(g,\bII)\Big)\, \dvol_g.
\end{equation}
Hence the functional $\mathcal{E}_{GR}$ can be represented as in~\eqref{defE}. One can also prove that the Dirichlet energy of $\vec\II$ is proportional to the Dirichlet energy of $\vec H$ plus a linear combination of the terms $P_1,\ldots,P_8$, see for instance the proof of Lemma B.9 in \cite{M2024}. In addition, we note that when $m=5$ (in codimension $1$), we have $P_2=P_3$, $P_4=P_5$, and $P_7=P_8$. Consequently, in codimension 1 there are only $5$ linearly independent geometric energies invariant under dilations and rotations which are polynomials in the second fundamental form. \\

A weak immersion $\bP\in \mathcal I_{1,2}(\Sigma^4,\R^m)$ is said to be a critical point of $E_{\vec c}$ if for any $\vec w\in C_c^\infty(\Si^4,\R^m)$, it holds that
\[
   \frac d{dt} \left. E_{\vec c}\left( \vec\Phi+t\mkt \vec w \right)\right|_{t=0}=0.
\]
The main result of the present article is as follows. 
\begin{Th}
\label{th-main}
Let $\Sigma^4$ be a closed orientable smooth manifold of dimension $4$ and $\vec{c}\in\R^8$. Then a critical point $\vec{\Phi}\in \mathcal I_{1,2}(\Sigma^4,\R^m)$ of $E_{\vec c}$ is real-analytic in $g_{\bP}$-harmonic coordinates.
\end{Th}

The strategy for the proof of Theorem~\ref{th-main} takes its roots in the two dimensional counterpart, the proof of the regularity of weak critical points to the Willmore functional.
In dimension two, the conservation laws were the main tools for performing the variational analysis of the Willmore energy as it has been initiated in \cite{Riv08}. The first author later discovered in \cite{Ber} that these conservation laws follow the Noether theorem, using the invariance of the Willmore energy by the Möbius group. Indeed, the Euler--Lagrange equation is a priori not compatible with the Willmore energy (the equation contains a term roughly of order $H_{\bP}^3$ but the Willmore equation controls only the $L^2$ norm of the mean curvature, we refer to \cite{LaMaRi25,Riv08} for a discussion of this problem).
The invariance by translation allows first to write the Euler--Lagrange equation in divergence form, which in return, makes the equation compatible with the natural variational assumptions that the immersion is in $W^{1,\infty}\cap W^{2,2}$. Then the invariance by dilation and rotations permits to rewrite the Willmore equation which is a priori a fourth order degenerate elliptic PDE (as first computed by S. Poisson in \cite{poisson}) into a second order nondegenerate elliptic system with critical nonlinearities to which methods from the integrability by compensation theory can be applied.
In the original proof the role of the local isothermic coordinate in which the metric is continuous had been used.
Recently, the second author in \cite{Lan} provides an alternative proof without using these very special coordinates, which exist exclusively in dimension 2. In generic coordinates, the elliptic system contains coefficients in $L^\infty\cap W^{1,2}$. \\

Coming now to the four-dimensional case, the application of Noether theorem is reminiscent of the two-dimensional Willmore problem though the equations are very different and substantially more complex. The absence of isothermal coordinates moreover requires a preparatory analysis on elliptic systems with Sobolev coefficients to which a full preliminary section is devoted.
Let $\bP\in \mathcal I_{1,2}(B^4,\R^m)$. Throughout this paper, we shall often omit the subscript ${\vec{\Phi}}$ for $g_{\bP}$, $H_{\bP}$, etc. when there is no ambiguity. We will prove in Section \ref{sec:pfmainThm}, that the Euler--Lagrange equation of $E$ is of the form (see the convention in Section \ref{geono}):
\begin{numcases}{}
d*_g \vec V=0,\label{intrd*V=0}\\
        \vec V=\frac 12\mko d\lap_g\bH+\vec l_0+\sum_{s=1}^8c_s\vec l_s,  \label{defV1st}\\
        \vec l_0\in W^{-1,\frac 43}+L^1(B^4,\R^m\ot \R^4)\label{intrl0},\\[1.2ex]
        \vec l_s\in L^1(B^4,\R^m\ot \R^4),\quad 1\le s\le 8.\label{intrls}
\end{numcases} 
 By weak Poincar\'e lemma (see for instance \cite[Cor.~3.4]{Costa10}), there exists $\bL\in W^{-1,(2,\infty)}\big(B^4,\R^m\ot \bwe ^2 \R^4\big)$ such that 
 \begin{equation}
     d *_g \bL=*_g \,\vec V. \label{d*gL=*gV}
    \end{equation}
    One of the main achievements of this work, is the development of the analysis required to study solvability problems associated with the Hodge--Dirac operator $d+d^{*_g}$ for differential forms in negative Sobolev spaces, together with a background metric $g$ that is merely $L^{\infty}\cap W^{1,n}(B^n)$ and thus, \textit{a priori not continuous}. As a result of Section \ref{sec:hod-dec}, we shall prove that $\vec L$ can be chosen such that
    \begin{equation}
        d\bL=0. \label{dL0=0int}
    \end{equation}

Invariance of $E$ under dilations and rotations in $\R^m$ yields the corresponding Noether currents $\bR$ and $S$ analogously to the $2$-dimensional case in~\cite{Riv08}, with the key difference that in this four-dimensional setting, the maps $\bL$, $\vec{R}$, and $S$ are 2-forms (as opposed to 0-forms in the Willmore case). The codifferentials $d^{*_g} \vec{R}$ and $d^{*_g} S$ are determined by $\bL$. In this setting, the regularity and properties of $\bL$, $\vec{R}$, and $S$ depend critically on the choices of their differentials $d\bL$, $d\vec{R}$, and $dS$, though their codifferentials remain fixed. Another major difficulty is to find structure equations that allow us to deduce estimates on $\Delta_g \bH$ from the regularity estimates on $\vec{R}$ and $S$, which requires a careful analysis of differential forms. The study of regularity properties for generalized Willmore energies in higher dimensions will be the topic of a future work.

\subsubsection*{Organization of the paper.} 
  In \Cref{sec-preliminaries}, we set some notations and record some estimates and identities that will be used later. In Section \ref{sec:hod-dec}, we solve the Dirichlet problem for $\lap_g$, which enables us to find $\bL$ satisfying~\eqref{d*gL=*gV}--\eqref{dL0=0int}, and select $d\vec{R}$, $dS$ as closed forms within specified regularity classes. In Section \ref{sec:strid}, we prove the structural identities and estimates necessary to manipulate the Noether currents coming from the invariance of $E_{\vec{c}}$ by isometries and dilations. Finally, we derive the conservation laws and prove Theorem \ref{th-main} in Section~\ref{sec:pfmainThm}.

\subsubsection*{Acknowledgements.}
This project is financed by the Swiss National Science Foundation, project SNF 200020\textunderscore219429.



\section{Preliminaries}\label{sec-preliminaries}
\subsection{Notation} \label{geono}
\begin{enumerate}[(i)]
\setlength{\itemsep}{0pt}   
\item Write $\p_{x^i}= \frac \p{\p x^i}$, and           abbreviate to $\p_i$ when there is no ambiguity. 
\item In the Euclidean space $\R^n$, we define $B_r(x)\coloneq \{y\in \R^n\colon |y-x|<r\}$, and write $B^n\coloneq B_1(0)\subset \R^n$. The phrase ``for every ball $B_r\subset U$'' refers to every ball of radius $r$ contained in $U$.
\item We denote by $\R^{m\times n}$ the space of $m\times n$ real matrices, and by $\R^{n\times n}_{\text{sym}}$ the space of $n\times n$ symmetric real matrices.
    \item We use $C(\alpha,\beta,\dots)$ to denote a positive constant depending on $\alpha,\,\beta,\dots$ only. Similarly, when a term $A$ depends only on $\alpha,\,\beta,\dots$, we write $A=A(\alpha,\beta,\dots)$.
    \item For open sets $U$ and $V$ of a given manifold, we write $V\Subset U$ if $\overline V$ is compact and $\overline V\subset U$.
    \item Let $\Sigma$ be a closed smooth manifold, $E$ be a vector bundle over $\Sigma$, we denote by $L^{p}(\Si,E)$ the space of $L^{p}$ sections of $E$. This convention also applies to the other function/distribution spaces.
\item We denote by $\mathcal L^n$ the Lebesgue measure on $\R^n$. For $U\subset \R^n$ with $\mathcal L^n(U)<\infty$, $f\in L^1(U)$, we write $\fint_U f\coloneq(\mathcal L^n(U))^{-1}\int _U f$.
\item For an open set $U\subset \R^n$, we denote by $\mathcal D'(U,\R^m)$ the space of $\R^m$-valued distributions on $U$ with the weak$^*$ topology, and we write $\mathcal D'(U)=\mathcal D'(U,\R)$.
 \item 
    Let $(V, g)$ be an $n$-dimensional inner product space with a positively oriented orthonormal basis $(e_1,\dots,e_n)$. The \textit{Hodge star operator} $*_g$ (see for instance \cite[Sec.~3.3]{Jost17}) is defined as the unique linear operator from $\bwe ^k V$ to $\bwe ^{n-k}V$ satisfying 
    \[
    \alpha\we *_g\, \beta=\langle \alpha,\beta  \rangle_g\, e_1\we\cdots \we e_n, \qquad \text{for all }\, \al,\beta\in \bwe^k V.\]
   When $V=\R^m$ and $g=g_{\text{std}}$ is the standard inner product on $\R^m$, we write $\star$ instead of $*_g$.
   \item Let $U\subset \R^n$ be an open set and let $g$ be a Riemannian metric on $U$. On the space of differential $\ell$-forms on $U$, we define the
\textit{codifferential} and \textit{Laplace} operators (cf. \cite[Defs.~3.3.1--3.3.2]{Jost17}) by
\begin{equation}
\label{defd*glap}
 d^{*_g}\coloneq(-1)^\ell*_g^{-1}d\,*_g=(-1)^{n(\ell+1)+1}*_gd\,*_g\qquad \text{and} \qquad \Delta_g\coloneq-(d\mko d^{*_g}+d^{*_g}d).
\end{equation} 
 If $\al\in C^\nf\big(U,\bwe^\ell T^* U\big)$ and $\beta \in C^\nf\big(U,\bwe^{\ell+1} T^* U\big)$ with at least one of $\al,\beta$ compactly supported in $U$, then we have
\begin{equation}\label{d^*adjoi}
    \int_U \lan d\al,\beta \ran_g \,\dvol_g = \int_U \lan \al,d^{*_g}\beta\ran_g\,\dvol_g.
\end{equation}
When $\ell=0$, $\lap_g$ coincides with the Laplace--Beltrami operator on functions.
   \item Let $U\subset \R^n$ be open and bounded. We call $U$ a \textit{Lipschitz domain} if for each $y\in \p U$, there exist $r>0$ and a Lipschitz function $f\colon \R^{n-1}\rightarrow\R$ such that, upon rotating and relabeling the coordinate axes if necessary and writing $x=(x',x_n)$, we have
\[
    U\cap B_r(y)=\{x\in B_r(y):x_n>f(x')\}.
\]
Equivalently, we say $\p U$ is Lipschitz or in $C^{0,1}$.
 \item Throughout this paper, we use the Einstein summation convention, and set $\delta_i^j=\de_{ij}=\mathbf 1_{i=j}$. We leave out the symbols $\otimes$ for sections of $\bwe \R^m\ot \bwe T^*U$ when there is no ambiguity. For instance, we write $\p_i\bP \,dx^j=\p_i\bP\ot dx^j$.
\item \label{not-met}Let $U\subset\R^{4}$ be open and $\bP\colon U\to\R^{m}$ be an immersion. Set $g=g_{\bP}=\bP^*g_{\text{std}}$, where $g_{\std}$ is the Euclidean metric on $\R^m$. Denote 
\[
   g_{ij}=\p_i\bP\cdot \p_j\bP,\quad (g^{ij})=(g_{ij})^{-1},\quad  \det g=\det(g_{ij}),\quad d\textup{vol}_g=(\det g)^{\frac 12}dx^1\we\cdots \we dx^4.
\]
The pullback metric induces a pairing on $T^* U$ with $\lan dx^i,dx^j \ran_g=g^{ij}$; we denote by $|\cdot|_{g}$
the associated pointwise norm. We also write $|\cdot|_{\R^{m}}$ for the
Euclidean norm in $\R^{m}$ (abbreviated to $|\cdot|$ when unambiguous).

Let $\pi_{\bn_{\bP}}$ denote the orthogonal projection onto the normal bundle of $\bP(U)\subset \R^m$. We define the \textit{Gauss map} and \textit{second fundamental form} \begin{gather*}
\bn_{\bP}\coloneq \star \,\frac{\p_1 \bP\we \p_2\bP\we \p_3\bP  \we \p_4 \bP}{|\p_1 \bP\we\p_2\bP  \we  \p_3 \bP \we\p_4 \bP|},\\
\bII_{\bP}(X,Y)\coloneq\pi_{\bn_{\bP}} X(d\bP(Y)),\qquad \text{for all } X,Y\in T_p U\text{ and all } p\in U.
\end{gather*}
We shall often omit the subscript ${\vec{\Phi}}$ when there is no ambiguity. We denote 
\[\bII_{ij}\coloneq \bII\Big(\frac \p{\p x^i},\frac{\p}{\p x^j}\Big)\quad\text{and}\quad \bII_i^j\coloneq g^{jk}\bII_{ik}.
\]
The \textit{mean curvature vector} is defined by
\[
\bH\coloneq \frac 14 \mkt \mbox{tr}_{g}\,\bII=\frac 14\mkt g^{ij}\,\bII_{ij}.
\]
 \item \label{not-overset} We define a bilinear map $\dwe\colon\big( \R^m\ot\bwe^{k_1} T^*_p U\big)\times\big(\R^m\ot\bwe^{k_2} T^*_pU\big) \rightarrow  \bwe^{k_1+k_2} T^*_pU$ by 
 \[
              (\vec u_1 \mkt dx_I)\dwe(\vec u_2\mkt dx_J)=( \vec u_1\cdot  \vec u_2) \mkt dx_I\wedge dx_J,
\]
where $\vec u_1,\vec u_2\in\R^m$, $I,J $ are multi-indices. If one of $k_1,k_2$ is $0$, we use $\cdot$ instead of $\dwe$ for convenience. Similarly, for $\vec v,\vec v_1,\vec v_2\in \bwe\R^m$ and multi-indices $I,J$, we define bilinear maps $\we$ and $\overset{\ovwe}{\wedge}$ by \begin{align*}
    &dx_I \we (\vec v\, dx_J)=\vec v\, dx_I \we dx_J,\qquad
         \vec v_1 \we (\vec v_2\mkt dx_J)=( \vec v_1 \we  \vec v_2)\mkt  dx_J,\\[0.3ex]
         & (\vec v_1 \mkt dx_I)\overset{\ovwe}\wedge(\vec v_2\mkt dx_J)=( \vec v_1 \we  \vec v_2) \mkt dx_I\wedge dx_J.
         \end{align*}
\item \label{con-formsob}Unless explicitly stated otherwise, we identify $T_x^*\R^n$ with $\R^n$ for $x\in \R^n$, and write $\nabla$, $|\cdot|$ for the flat Euclidean derivative and
pointwise tensor norms. We reserve $\nabla^g$, $\lan\cdot,\cdot \ran_g$, and $|\cdot|_g$ for the metric $g$. 
 For an open set $U\subset \R^n$, we write $\al\in W^{k,p}\big(U,\bwe^\ell \R^n\big)$  if $\alpha=\sum_{|I|=\ell}\alpha_I\mko dx^I$ with each $\al_I\in W^{k,p}(U)$, where $I$ ranges over all strictly increasing multi-indices of length $\ell$. We set
\begin{equation}\label{convsonab}
    |\nabla\alpha|^2\coloneq\sum_{i=1}^n\sum_I |\partial_i\mko \alpha_I|^2,\qquad 
    \|\al\|_{W^{k,p}(U)}\coloneq\sum_{I} \|\al_I\|_{W^{k,p}(U)}.
\end{equation}
Similarly, we write $X\in W^{k,p}(U, TU)$ if $X=X_j\mko \frac {\p}{\p x^j}$ with each $X_j\in W^{k,p}(U)$, and set 
\begin{equation}\label{sobvecfie}
    |\nabla X|^2\coloneq\sum_{i=1}^n\sum_{j=1}^n |\partial_i\mko X_j|^2,\qquad 
    \|X\|_{W^{k,p}(U)}\coloneq\sum_{j=1}^n \|X_j\|_{W^{k,p}(U)}.
\end{equation}
The same convention applies to any Banach function/distribution space (e.g. $L^{p}$, $W^{k,(p,q)}$).
\end{enumerate}
\subsection{A useful lemma}\label{sec:uselem}
We define the \textit{interior product} and \textit{first order contraction} between multivectors (see also \cite[Sec.~1.5]{Federer96} and \cite[Sec.~I]{Riv08}).
\begin{Dfi}\label{defresbul}
    Let $(V,g)$ be a finite-dimensional inner product space. For $\alpha\in \bwe^pV$ and $\beta\in \bwe^q V$ with $q\le p$, we define the \textit{interior product} $\alpha\,\resg \beta\in \bwe^{p-q}V$ satisfying 
    \begin{align}\label{eq:defres}
    \langle \alpha\,\resg \beta,  \gamma\rangle_g=\langle \alpha, \beta\wedge \gamma\rangle_g,\qquad \text{for all }\gamma\in \bwe^{p-q}V. 
    \end{align}
For $q>p$, set $\al\, \resg \beta\coloneq0$. We also define the \textit{first order contraction} $\bulg \mko \colon \bwe^{p} V\times \bwe^{q} V\to \bwe^{p+q-2} V$ as follows. For $\al\in \bwe^p V$ and $\beta\in V$, set $\al \bulg \beta\coloneq\al\mko\resg \beta$; and for $\beta\in \bwe^{q_1} V$, $\ga\in \bwe^{q_2} V$, it holds that
\begin{align}\label{eq:defbul}
    \al\bulg(\beta \we \ga)=(\al\bulg \beta) \we \ga+(-1)^{q_1q_2}  (\al \bulg \ga)\we \beta.
\end{align}
\end{Dfi}
    In particular, if $\alpha,\beta \in \bwe^pV$, then we have $\alpha\,\resg\,\beta=\langle \alpha,\beta \rangle_g$. In addition, for $\al\in \bwe^p V$, $\beta\in\bwe V$, $v\in V$, we have the following identities:
   \begin{numcases}{}   *
   _g\,(\alpha\wedge \beta )=(*_g\,\alpha)\,\resg \beta \label{*_gcommwed},\\[0.4ex]
   (\al\we \beta)\,\resg v=(\al\,\resg v) \we \beta+(-1)^p\, \alpha\we (\beta\,\resg v),\label{prodrule}
   \end{numcases}\smallskip
  A consequence of \eqref{*_gcommwed} is that, for $\al\in \bwe^p V$ and $\beta\in \bwe^q V$ we have
  \begin{align}\label{*_gcomres}
      *_g\, (\al\, \resg \beta)=\begin{dcases}
          (*_g\, \al) \we \beta,& \text{if }\dim(V) \text{ is odd},\\[0.4ex]
          (-1)^q (*_g\, \al) \we \beta, \quad& \text{if }\dim(V) \text{ is even}.
      \end{dcases}
  \end{align} 
  Combining~\eqref{eq:defbul} and \eqref{prodrule}, for $u_1,u_2,v_1,v_2\in V$, we also obtain
\begin{align}\label{bul2vecs}
\begin{aligned}
    (u_1\we u_2) \bulg (v_1\we v_2)
    &=\lan u_1,v_1\ran_g \,u_2\we v_2+\lan u_2,v_2\ran_g \,u_1\we v_1\\
    &\quad-\big(\lan u_1,v_2\ran_g \mkt u_2\we v_1+\lan u_2,v_1\ran_g \,u_1\we v_2\big).
    \end{aligned}
\end{align}

When $V=\R^m$ with the Euclidean metric, we write $\,\res,\sbul$ instead of $\,\resg,\bulg$, and denote by $\cdot$ the inner product on $\bwe \R^m$ induced by the Euclidean metric. Now we define the \textit{contracted wedge product} (cf. \cite[Sec.~2.1]{Houri10}), the $k$-fold version of $\sbul$.
\begin{Dfi}
    Let $(\mathbf e_1,\dots,\mathbf e_m)$ be an orthonormal basis of $\R^m$. For $\al,\beta\in \bwe \R^m$ and $k\in \N$, we define the $k$-fold contracted wedge inductively by 
    \begin{align*}
       \al \sbul^{(0)} \beta\coloneq \al\we \beta,\qquad  \al \sbul^{(k)} \beta\coloneq\sum_{i=1}^m (\al\,\res \mathbf e_i)\sbul^{(k-1)} (\beta\,\res \mathbf e_i). 
    \end{align*}
   This definition is independent of the choice of the orthonormal basis $(\mathbf e_i)_{i=1}^m$. In particular, we have $\sbul^{(1)}=\sbul$\,.
\end{Dfi}
Now we prove a higher-dimensional generalization of \cite[Lem.~1.8]{Ri16}.
\begin{Lm}\label{lm:Hdnstrequ}
    Let $n\ge 2$, and $\bP\colon B^n\rightarrow \R^{m}$ be a smooth immersion. Let $*_g$ be the Hodge star operator with respect to $g=g_{\bP}$, and $\bn$, $\bH$ be as in~\eqref{not-met}. Then we have
    \begin{equation}\label{eq:Hdnstrequ}
        d\bn+n\,d\bP\we (\bn\,\res \bH)= \frac{(-1)^n}{(n-2)!}*_g\star\mko \Big( \big(d\bn\sbul^{(m-n-1)} \bn\big) \ovs{\ovwe}{\we}d\bP^{\ovs\ovwe\we(n-2)}\Big).
    \end{equation}
    Here we write $d\bP^{\ovs\ovwe\we(n-2)}$ for the mixed $\ovs\ovwe \we$ of $d\bP$ with itself $n-2$ times, i.e.  
    \[
    d\bP^{\ovs\ovwe\we(0)}\coloneq 1,\qquad d\bP^{\ovs\ovwe\we (k)}\coloneq d\bP^{\ovs\ovwe\we (k-1)}\ovs\ovwe\we d\bP,\quad \text{for all }k\in \N^+.
    \]
\end{Lm}
\begin{proof}
    Since the normal bundle to $\bP(B^n)\subset \R^m$ is trivial, there exists a smooth orthonormal frame $(\bn_\al)_{\al=1}^{m-n}$ with $\bn=\bn_1\we\cdots \we \bn_{m-n}$. Let $\pi_{T}=\text{id}-\pi_{\bn}$ denote the orthogonal projection onto the tangent bundle of $\bP(B^n)\subset \R^m$. We first prove that for each $\al=1,\dots,m-n$, we have
    \begin{equation}\label{apphypercase}
        \pi_{T}\mko d\bn_\al+n\mko (\bH\cdot \bn_\al) \mko d\bP= \frac{(-1)^{m-n-1}}{(n-2)!}*_g\star\mko \Big( d\bn_\al\ovs{\ovwe}{\we}d\bP^{\ovs\ovwe\we(n-2)}\we \bn\Big). 
    \end{equation}
   Let $p\in B^n$, $(e_1,\dots,e_n)$ be a $g$-positive orthonormal basis of $T_pB^n$, and $(e^i)$ be the dual coframe to $(e_i)_{i=1}^n$. Write $\vec e_i=d\bP(e_i)$. Then we have 
  \[
        \frac{1}{(n-2)!}\, d\bP^{\ovs\ovwe\we (n-2)}=\sum_{1\le i_1<\cdots <i_{n-2}\le n}  (\vec e_{i_1}\we\cdots \we \vec e_{i_{n-2}})\,  e^{i_1}\we \cdots \we e^{i_{n-2}}.
    \]
    By direct computation, for $1\le i\le n$, we obtain
    \begin{equation}\label{eiiwedphi*}
       \frac{(-1)^{m-n-1}}{(n-2)!}*_g\star\mko \Big( (\vec e_i\ot e^i)\ovs{\ovwe}{\we}d\bP^{\ovs\ovwe\we(n-2)}\we \bn\Big)=-\sum_{j\neq i} \vec e_j\ot e^j.
    \end{equation}
    For $j\neq i$, we have 
    \begin{equation}\label{eijwedphi*}
          \frac{(-1)^{m-n-1}}{(n-2)!}*_g\star\mko \Big( (\vec e_j\ot e^i )\ovs{\ovwe}{\we}d\bP^{\ovs\ovwe\we(n-2)}\we \bn\Big)=\vec e_i\ot e^j.
    \end{equation}
    Combining \eqref{eiiwedphi*}--\eqref{eijwedphi*} and using $\pi_T\mko d\bn_{\al}=\sum_{1\leq i,j\leq n} (e_i(\bn_\al)\cdot \vec e_j)\,  \vec e_j\ot e^i$ yields
    \begin{align*}
      \frac{(-1)^{m-n-1}}{(n-2)!}*_g\star\mko \Big( d\bn_\al\ovs{\ovwe}{\we}d\bP^{\ovs\ovwe\we(n-2)}\we \bn\Big) 
      &=  \frac{(-1)^{m-n-1}}{(n-2)!}*_g\star\mko \Big( \pi_T\mko d\bn_\al\ovs{\ovwe}{\we}d\bP^{\ovs\ovwe\we(n-2)}\we \bn\Big)\\
      &=\sum_{i\neq j}\Big(  (e_i(\bn_\al)\cdot \vec e_j)\,  \vec e_i\ot e^j-(e_i(\bn_\al)\cdot \vec e_i)\,  \vec e_j\ot e^j\Big).
   \end{align*}
    On the other hand, since $n\mko (\bH\cdot \bn_\al)=-\sum_{i=1}^{n} e_i(\bn_\al)\cdot \vec e_i $ and $e_i(\bn_\al)\cdot \vec e_j=e_j(\bn_\al)\cdot \vec e_i$,  we obtain that
    \begin{align*}
          \pi_{T}\mko d\bn_\al+n\mko (\bH\cdot \bn_\al) \mko d\bP&= \sum_{i,j=1}^{n}  \Big((e_i(\bn_\al)\cdot \vec e_j)\,  \vec e_j\ot e^i- (e_i(\bn_\al)\cdot \vec e_i)\,  \vec e_j\ot e^j\Big)\\
         &=\sum_{i\neq j} \Big((e_i(\bn_\al)\cdot \vec e_j)\,  \vec e_j\ot e^i- (e_i(\bn_\al)\cdot \vec e_i)\,  \vec e_j\ot e^j\Big)\\
         &=\sum_{i\neq j} \Big((e_i(\bn_\al)\cdot \vec e_j)\,  \vec e_i\ot e^j- (e_i(\bn_\al)\cdot \vec e_i)\,  \vec e_j\ot e^j\Big).
    \end{align*}
    The identity~\eqref{apphypercase} is thus proved. \\
    Since $d\bn_\al\cdot \bn_\al=0$, we have $(\pi_{\bn}\mko d\bn_\al) \we (\bn\,\res \bn_\al)=0$. Wedging both sides of~\eqref{apphypercase} with $\bn\,\res \bn_\al$ then gives
   \begin{equation}\label{dnawenresna}
   \begin{aligned}
       \big(d\bn_\al+n\mko (\bH\cdot \bn_\al) \mko d\bP\big)\we (\bn\,\res \bn_\al)
       &=\big(\pi_T\mko d\bn_\al+n\mko (\bH\cdot \bn_\al) \mko d\bP\big)\we (\bn\,\res \bn_\al)\\
       &=\frac{(-1)^{m-n-1}}{(n-2)!}*_g\star\mko \Big( d\bn_\al\ovs{\ovwe}{\we}d\bP^{\ovs\ovwe\we(n-2)}\we \bn\Big)\we (\bn\,\res \bn_\al).
   \end{aligned}
   \end{equation}
   Concerning the right-hand side of~\eqref{dnawenresna}, the identity~\eqref{*_gcomres} for $\star$ and $\res$ on $\bwe \R^m$ implies that
   \begin{align}\label{strdnawenres}
   \begin{aligned}
       &\star\mko \Big( d\bn_\al\ovs{\ovwe}{\we}d\bP^{\ovs\ovwe\we(n-2)}\we \bn\Big)\we (\bn\,\res \bn_\al)\\
       &=(-1)^{(m+1)(m-n-1)}\star \lf( \Big( d\bn_\al\ovs{\ovwe}{\we}d\bP^{\ovs\ovwe\we(n-2)}\we \bn\Big)\res (\bn\,\res \bn_\al)\rg)\\
       &=(-1)^{(m+1)(m-n-1)}\star \lf( \Big( (\pi_T \mko d\bn_\al)\ovs{\ovwe}{\we}d\bP^{\ovs\ovwe\we(n-2)}\we \bn_\al\we (\bn\,\res \bn_\alpha)\Big)\res (\bn\,\res \bn_\al)\rg).       
       \end{aligned}
   \end{align}
   Since $(\pi_T\mko d\bn_\al) \ovs{\ovwe}{\we}d\bP^{\ovs\ovwe\we(n-2)}\we \bn_\al$ is a section of $\bwe^nT\bP \ot \bwe^{n-1}T^*B^n$, where $T\bP=\bP_*(TB^n)$ is the tangent bundle of $\bP$, switching the order of wedge in~\eqref{strdnawenres} gives that 
   \begin{align}\label{strdnawenres2}
       \begin{aligned}
          &\star\mko \Big( d\bn_\al\ovs{\ovwe}{\we}d\bP^{\ovs\ovwe\we(n-2)}\we \bn\Big)\we (\bn\,\res \bn_\al)\\
          &=(-1)^{(m+1)(m-n-1)+n(m-n-1)}\star \lf( \Big((\bn\,\res \bn_\alpha)\we (\pi_T \mko d\bn_\al)\ovs{\ovwe}{\we}d\bP^{\ovs\ovwe\we(n-2)}\we \bn_\al\Big)\res (\bn\,\res \bn_\al)\rg)\\
          &=(-1)^{m+n+1}\star \Big((\pi_T\mko d\bn_\al) \ovs{\ovwe}{\we}d\bP^{\ovs\ovwe\we(n-2)}\we \bn_\al \Big)\\
       &=(-1)^{m+1}\star \Big((\pi_T\mko d\bn_\al) \we \bn_\al \ovs{\ovwe}{\we}d\bP^{\ovs\ovwe\we(n-2)} \Big).
       \end{aligned}
   \end{align}
    Combining~\eqref{dnawenresna} and~\eqref{strdnawenres2}, we obtain 
    \begin{equation}\label{dnawenres2}
         \big(d\bn_\al+n\mko (\bH\cdot \bn_\al) \mko d\bP\big)\we (\bn\,\res \bn_\al)=\frac{(-1)^n}{(n-2)!}*_g \star \Big((\pi_T\mko d\bn_\al) \we \bn_\al\ovs{\ovwe}{\we}d\bP^{\ovs\ovwe\we(n-2)}\Big).
    \end{equation}
    By the definition of the contracted wedge operator, we have 
    \begin{equation}\label{piTdnawena}
        (\pi_T\mko d\bn_\al) \we \bn_\al=\big(d\bn_\al \we (\bn\,\res \bn_\al)\mko\big)\sbul^{(m-n-1)} \bn.
    \end{equation}
    Since $d\bn=\sum_{\al=1}^{m-n} d\bn_\al \we (\bn\,\res \bn_\al)$, substituting~\eqref{piTdnawena} into~\eqref{dnawenres2} and summing over $\al=1,\dots,m-n$ completes the proof of~\eqref{eq:Hdnstrequ}.
\end{proof}

\subsection{Sobolev--Lorentz spaces}\label{sec:soblor}
Assume $X_1,X_2$ are Banach spaces that are continuously embedded in a Hausdorff topological vector space $Z$. Equipped with the norms defined for instance in~\cite[Ch.~5]{Bennett88}, $X_1+X_2$ and $X_1\cap X_2$ are Banach spaces. We write $L^p+L^q(U)\coloneqq L^p(U)+L^q(U)$. The same convention applies to Sobolev--Lorentz spaces.

\begin{Dfi}[Lorentz spaces]\label{def-Lor}
Let \(U\subset\R^{n}\) be a measurable set. Given a measurable function \(f\colon U\to\R\), we define the distribution function and the decreasing rearrangement of \(f\) as
\[
d_{f}(\lambda)\coloneqq\mathcal{L}^{n}\bigl\{x\in U:|f(x)|>\lambda\bigr\},
\qquad
f^{*}(t)\coloneqq\inf\bigl\{\lambda\ge 0 : d_{f}(\lambda)\le t\bigr\}.
\]
For \(1\le p<\infty\) and \(1\le q\le\infty\), we define the Lorentz quasinorm
\[
|f|_{L^{p,q}(U)}
\coloneqq\bigl\|t^{\frac1p}f^{*}(t)\bigr\|_{L^{q}(\R_{+},\,dt/t)}
    =p^{\frac1q}\bigl\|\lambda\,d_{f}(\lambda)^{\frac1p}\bigr\|_{L^{q}(\R_{+},\,d\lambda/\lambda)}.
\]
The Lorentz space \(L^{p,q}(U)\) consists of all measurable \(f\) with
\(|f|_{L^{p,q}(U)}<\infty\). 
When \(p>1\), the Lorentz quasinorm is equivalent to a norm (see for instance \cite[Ch.~4, Thm.~4.6]{Bennett88}), which we denote by
\(\|\cdot\|_{L^{p,q}(U)}\).
\end{Dfi}
We record the following form of H\"older's inequality for Lorentz spaces.
\begin{Lm}[{\cite[Thm.~4.5]{Hunt}}]\label{lm:lorHold}
   Assume $f_1\in L^{p_1,q_1}(U)$ and $f_2\in L^{p_2,q_2}(U)$, with $p,p_1,p_2\in [1,\nf)$, $q,q_1,q_2\in [1,\nf]$, and $1/p=1/p_1+1/p_2$, $1/q\le 1/q_1+1/q_2$. Then $f_1f_2\in L^{p,q}(U)$, with
   \begin{align*}
       |f_1f_2|_{L^{p,q}(U)}\le C(p_1,p_2,q_1,q_2) |f_1|_{L^{p_1,q_1}(U)}|f_2|_{L^{p_2,q_2}(U)}.
   \end{align*}
\end{Lm}
\begin{Dfi}
    \label{dfi-So-Lor}
    Let $k\in \N^+$ and $U\subset \R^n$ be an open set. For $1< p< \infty$, $1\le q\le \infty$, we set 
    \[
    W^{k,(p,q)}(U)\coloneqq\Big\{f\in L^{p,q}(U)\colon \p^\al f\in L^{p,q}(U) \text{ for each } 0\le |\al|\le k\Big\}.
    \]
    We also define the negative-order Sobolev--Lorentz space and its norm by
\begin{align*}
   &W^{-k,(p,q)}(U)\coloneqq\bigg\{f\in\mathcal D'(U)\colon f=\sum_{|\alpha|\le k} \p^\alpha f_{\alpha} \mbox{ for some }\{f_{\al}\}\subset L^{p,q}(U)\bigg\},\\
       &\|f\|_{W^{-k,(p,q)}(U)}\coloneqq\inf \bigg\{\sum_{|\alpha|\le k}\|f_\alpha\|_{L^{p,q}(U)}\colon f=\sum_{|\alpha|\le k} \p^\alpha f_{\alpha}\bigg\}.
\end{align*}
    When $U$ is bounded, we denote by $W_0^{k,p}(U)$ the closure of $C_c^\infty(U)$ in $W^{k,p}(U)$, equipped with the norm \begin{align*}
        \|f\|_{W^{k,p}_0(U)}\coloneqq\|\g^k f\|_{L^p(U)}.
    \end{align*}
     If $1<p<\infty$, $1< q\le \nf$, and \(1/p+1/p'=1\), \(1/q+1/q'=1\), then the same argument as in \cite[Sec.~1.1.15]{Mazya2011} implies $W^{-k,(p,q)}(U)=\bigl(W^{k,(p',q')}_{0}(U)\bigr)^{\!*}$.
\end{Dfi}
Applying Riesz potential estimates, we obtain the following embedding results for Sobolev-Lorentz spaces, see for instance \cite[Eqs.~(1.3)--(1.5)]{Mingi11}, \cite[Ch.~4, Thm.~4.18]{Bennett88}, and \cite[Eq.~(1.2.4) \& Thm.~3.1.4]{Adams96}.
\begin{Lm}\label{lm:LpembW-1}
    Let $U\subset \R^n$ be an open set. Suppose $f\colon U\to \R$ is measurable.
\begin{enumerate}[(i)]
  \item If $f\in L^1(U)$, then $f\in W^{-1,(\frac n{n-1},\nf)}(U)$, with 
  \[
      \|f\|_{W^{-1,(\frac{n}{n-1},\nf)}(U)}\le C(n) \|f\|_{L^1(U)}.
  \]
  \item If $f\in L^{p,q}(U)$ for some $1< p<n$ and $1\le q\le \nf$, then $f\in W^{-1,(\frac{np}{n-p},q)}(U)$, with 
  \[
      \|f\|_{W^{-1,(\frac{np}{n-p},q)}(U)}\le C(n,p) \|f\|_{L^{p,q}(U)}.
  \]
\end{enumerate}
\end{Lm}
We will use the following interpolation results for Sobolev--Lorentz spaces; see~\cite[Sec.~1.3]{Triebel95} for the real interpolation method. Their proofs follow from the retraction--coretraction principle~\cite[Thm.~1.2.4]{Triebel95} and the Stein extension theorem \cite[Ch.~VI, Thm.~5]{Stein}, together with the interpolation result for $W^{k,p}(\R^n)$ in~\cite[Thm.~2.4.2/1(c)]{Triebel95}.
\begin{Lm}[{\cite{Triebel95}}]\label{lm-interpo}
    Let $U\subset \R^n$ be a bounded Lipschitz domain. Let $k\in \N^+$, $1<p_0<p<p_1<\nf$, $1\le q\le \nf$, $0<\theta<1$ and $1/p=(1-\theta)/p_0+ \theta/ p_1$. Then we have:
    \begin{align*}
       &\big( W^{k,p_0}(U),W^{k,p_1}(U)\big)_{\theta,q}=W^{k,(p,q)}(U),\\
       &\big( W^{k,p_0}_0(U),W^{k,p_1}_0(U)\big)_{\theta,q}=W^{k,(p,q)}_0(U),\\
       &\big( W^{-k,p_0}(U),W^{-k,p_1}(U)\big)_{\theta,q}=W^{-k,(p,q)}(U).
    \end{align*}
\end{Lm}

\subsection{Elliptic estimates}
\begin{Lm}\label{lm:ellcacciolp}
 Let $n\ge 2$. Suppose $\{a^{ij}\}_{i,j=1}^n\subset L^{\infty}\cap W^{1,n}(B^n)$ satisfies \begin{equation}\label{elli}
    \La^{-1}|\xi|^2\le a^{ij}(x)\,\xi_i\mko \xi_j \le \La |\xi|^2,\qquad \text{for a.e. } x\in B^n \text{ and all } \,\xi\in \R^n,
\end{equation}
where $\La>0$ is a constant. Let $\omega\colon (0,\infty)\rightarrow [0,\infty)$ be a function satisfying $\lim_{\rho\rightarrow 0} \omega (\rho)=0$. We assume for all $1\le i,j\le n$ and any ball $B_\rho\subset \R^n$ of radius $\rho$, it holds that
\[
   \|\g a^{ij}\|_{L^n(B_\rho\cap B^n)}\le \om(\rho).
\]
Let $1<p<\nf$, $1\le q,q_0\le \nf$, and $p_0>n/(n-1)$. Suppose $\vec f=(f^1,\dots, f^n)\in L^{p,q}(B^n,\R^n)$, and $u\in L^{p_0,q_0}(B^n)$ satisfies $ \p_j(a^{ij}\mko\p_i u)=\textup{div}\,\vec f$ in $\mca D'(B^n)$, that is, 
\[
    -\int_{B^n} u\,\p_i(a^{ij}\mko\p_j \vp)=\int_{B^n} f^i\mko \p_i \mko\vp,\qquad \text{for all }\, \vp\in C^\infty_c(B^n).
\]
 Then $u\in W^{1,(p,q)}_{\loc}(B^n)$. Moreover, for any $\al\in  (0,\frac np)$ and all $r\in (0,\frac 12]$, we have the following estimates:
 \begin{align}
  &\|\g u\|_{L^{p,q}(B_r(0))} \le C(\La,\om,\al,p,p_0,n) \big(\|\vec f\|_{L^{p,q}(B^n)}+r^{\al} \|u\|_{L^{p_0,q_0}(B^n)}\big),\label{eq:ellcacciop1}\\
   &\|\g u\|_{L^{p,q}(B_r(0))} \le C(\La,\om,\al,p,n) \big(\|\vec f\|_{L^{p,q}(B^n)}+r^{\al} \|\g u\|_{L^{p,q}(B_{3/4}(0))}\big).\label{eq:ellcacciopg}
  \end{align}
If in addition $p<n$, then for all $r\in (0,\frac 12]$, we obtain
\begin{equation}\label{eq:ellcacciop2}
     \|u\|_{L^{\frac{np}{n-p},q}(B_r(0))} \le C(\La,\om,p,p_0,n) \big(\|\vec f\|_{L^{p,q}(B^n)}+r^{\frac {n-p}p} \|u\|_{L^{p_0,q_0}(B^n)}\big).
\end{equation}
\end{Lm}
\begin{proof}
    By \cite[Thm.~1.5]{Byun05} and Lemma~\ref{lm-interpo}, we find a unique $u_0\in W^{1,(p,q)}_0(B^n)$ solving
\begin{equation}\label{lapu=divf}
 \begin{dcases}  
 \p_j(a^{ij} \mko\p_i\mko u_0)=\textup{div}\,\vec f & \quad\mbox{in } B^n,\\
 u_0=0  &\quad \mbox{on } \p B^n.
\end{dcases}
\end{equation}
Moreover, we have
\begin{equation}\label{estlapu0=divf}
    \|u_0\|_{W^{1,(p,q)}_0(B^n)}\le C(\La,\om,p,n)\|\vec f\|_{L^{p,q}(B^n)}. 
\end{equation}
Fix $p_1>n/(n-1)$ such that $1/p_1> 1/p-1/n$. By the Sobolev embedding, we obtain
\begin{equation}\label{u0Lp1bdf}
    \|u_0\|_{L^{p_1}(B^n)}\le C(p,n)  \|u_0\|_{W^{1,(p,q)}_0(B^n)}\le C(\La,\om,p,n)\|\vec f\|_{L^{p,q}(B^n)}. 
\end{equation}
Set $p_2\coloneqq\min(p_1,p_0)$ and $u_1\coloneqq u-u_0$. Then we have $\p_j(a^{ij}\mko\p_i \mko u_1)=0$ in $\mca D'(B^n)$, and
\begin{align*}
    \|u_1\|_{L^{p_2,q_0}(B^n)}&\le \|u_0\|_{L^{p_2,q_0}(B^n)}+\|u\|_{L^{p_2,q_0}(B^n)}\\
    &\le C(\La,\om,p,n) \big(\|\vec f\|_{L^{p,q}(B^n)}+\|u\|_{L^{p_0,q_0}(B^n)}\big).
\end{align*} 
Since $p_2>n/(n-1)$ and $a^{ij}\in W^{1,n}(B^n)$, by~\cite[Thm.~4.1]{laMan20} we obtain $u_1\in W^{1,s}_{\loc}(B^n)$ for any $s\in (1,\nf)$, with the estimate 
\begin{align*}
    \|u_1\|_{W^{1,s}(B_{3/4}(0))} &\le C(\La,\om,p,p_0,s,n) \|u_1\|_{L^{p_2,q_0}(B^n)}\\
    &\le C(\La,\om,p,p_0,s,n) \big(\|\vec f\|_{L^{p,q}(B^n)}+\|u\|_{L^{p_0,q_0}(B^n)}\big).
\end{align*}
Fix $\al\in (0,\frac np)$ and set 
$ \ga\coloneqq\max\big(\frac 12,\al+1-\frac np\big)\in (0,1)$.
Morrey's inequality \cite[Sec.~5.6.2]{evans} then implies $u_1\in C^{0,\ga}_{\loc}(B^n)$ with 
\begin{equation}\label{moru1bdfu}\begin{aligned}
  &\sup_{x\in B_{3/4}(0), x\neq 0}|x|^{-\ga}\mko |u_1(x)-u_1(0)|\\
 &\le C(\ga, n)\|u_1\|_{W^{1,\frac n{1-\ga}}(B_{3/4}(0))}\\
   &\le C(\La,\om,p,p_0,\al,n) \big(\|\vec f\|_{L^{p,q}(B^n)}+\|u\|_{L^{p_0,q_0}(B^n)}\big).
   \end{aligned}
\end{equation}
 Since the function $u_1-u_1(0)$ is a weak solution to the equation $\p_j\big(a^{ij}\mko\p_i(u_1-u_1(0))\big)=0$,
the elliptic estimate~\cite[Thm.~3.1]{Byun05} together with a dilation argument implies that, for all $r\in (0,\frac 12]$, 
\begin{align}\label{estgi_1Lpq}
\begin{aligned}
  &\|\g u_1\|_{L^{p,q}(B_{r}(0))}\\&=  \|\g(u_1-u_1(0))\|_{L^{p,q}(B_{r}(0))}\\
  &\le C(\La,\om,p,n) \,r^{\frac np-1}\|u_1-u_1(0)\|_{L^\nf (B_{3r/2}(0))}\\
  &\le C(\La,\om,\al,p,p_0,n)\, r^{\frac np-1+\ga} \big(\|\vec f\|_{L^{p,q}(B^n)}+\|u\|_{L^{p_0,q_0}(B^n)}\big)\\
 & \le  C(\La,\om,\al,p,p_0,n)\, r^\al \big(\|\vec f\|_{L^{p,q}(B^n)}+\|u\|_{L^{p_0,q_0}(B^n)}\big).
 \end{aligned}
\end{align}
Then, combining~\eqref{estlapu0=divf} with~\eqref{estgi_1Lpq} yields $u=u_0+u_1\in W^{1,(p,q)}_{\loc}(B^n)$, and for all $r\in (0,\frac 12]$, we have 
\begin{align*}
    \|\g u\|_{L^{p,q}(B_r(0))}&\le  \|\g u_0\|_{L^{p,q}(B^n)}+ \|\g u_1\|_{L^{p,q}(B_{r}(0))}\\
    &\le  C(\La,\om,\al,p,p_0,n) \big(\|\vec f\|_{L^{p,q}(B^n)}+r^{\al} \|u\|_{L^{p_0,q_0}(B^n)}\big).
\end{align*}
This completes the proof of \eqref{eq:ellcacciop1}. To prove~\eqref{eq:ellcacciopg}, we can without loss of generality assume $\int_{B_{3/4}(0)}u=0$. Then~\eqref{eq:ellcacciopg} follows from combining the estimate~\eqref{eq:ellcacciop1} in $B_{3/4}(0)$ (taking $q_0=q$ and $1/p-1/n\le 1/p_0<1-1/n$) with Poincar\'e inequality.

Finally, if in addition $p<n$, then the same argument as in \eqref{u0Lp1bdf} gives 
\begin{equation}\label{u0Lnpn-pqbd}
   \|u_0\|_{L^{\frac{np}{n-p},q}(B^n)}\le C(\La,\om,p,n)\|\vec f\|_{L^{p,q}(B^n)}.
\end{equation}
Applying Morrey's inequality as in~\eqref{moru1bdfu}, we obtain 
\[
    \|u_1\|_{C^0(\overline{B_{3/4}(0)})}\le C(\La,\om,p,p_0,n) \big(\|\vec f\|_{L^{p,q}(B^n)}+\|u\|_{L^{p_0,q_0}(B^n)}\big).
\]
Hence for $r\in (0,\frac 12]$, we have
\begin{align}
     \|u_1\|_{L^{\frac{np}{n-p},q}(B_r(0))}
     & \le r^{\frac{n-p}{p}}  \|u_1\|_{C^0(\overline{B_{3/4}(0)})} \nonumber \\
     &\le C(\La,\om,p,p_0,n)\mko r^{\frac{n-p}{p}} \big(\|\vec f\|_{L^{p,q}(B^n)}+\|u\|_{L^{p_0,q_0}(B^n)}\big).\label{u1Lnpn-pqbd}
\end{align}
Combining~\eqref{u0Lnpn-pqbd}--\eqref{u1Lnpn-pqbd} concludes the proof of~\eqref{eq:ellcacciop2}.
\end{proof}
\subsection{Extension of Sobolev metrics}
\begin{Lm}\label{lm-ext-metr}
Let $k\in\N^+$, $1\le p\le \infty$, $\La\ge 1$ be a constant. Let $U\subset\R^{n}$ be a bounded Lipschitz domain.  
Then there exists a continuous linear operator
$T\colon
W^{k,p}(U,\R^{n\times n}_{\textup{sym}})
\to
W^{k,p}_{\textup{loc}}(\R^{n},\R^{n\times n}_{\textup{sym}})$ such that for any $g=(g_{ij})\in W^{k,p}(U,\R^{n\times n}_{\textup{sym}})$, the following hold:

\begin{enumerate}[(i)]
\item    \label{Tisexten}
        $Tg=g$ a.e. in $U$.
\item \label{expreuniell}Suppose in addition that 
\begin{align}\label{asssobmet}
        \Lambda^{-1}\mkt  |\xi|^{2}
        \le
        g_{ij}(x)\,\xi^i\xi^j
        \le
        \Lambda\mkt  |\xi|^{2}, \qquad
        \text{for a.e. }x\in U\text{ and all }\,\xi\in\R^{n}.
\end{align}
        Then for a.e. $x\in\R^n$ and all $\xi\in\R^{n}$, the extension $Tg=((Tg)_{ij})$ satisfies
\[
        \frac 12\mkt  \Lambda^{-1}\mkt  |\xi|^{2}
        \le
        (Tg)_{ij}(x)\mkt  \xi^i\xi^j
        \le
        2\mkt  \Lambda\mkt  |\xi|^{2}.
\]
\item \label{wkpbdext}$\g (Tg)\in W^{k-1,p}(\R^n,\R^{n\times n \times n})$. Moreover, suppose $1\le p<\infty$ and there exists a bounded function $\om\colon (0,\infty)\to[0,\infty)$ such that  \begin{align*}
    \sum_{\ell=0}^k \|\g^\ell g\|_{L^p(B_r)} \le \om(r),\qquad \text{for every ball }\,B_r\subset \R^n.
\end{align*}
Then there exists a function $\widetilde \omega\colon (0,\infty)\to [0,\infty)$ depending only on $U$, $\La$, $p$, $k$, and $\omega$, with $\lim_{r\to 0} \widetilde \omega(r)=0$, such that
\[
     \|\g^k (Tg)\|_{L^p(B_r)}\le \wti \om(r),\qquad \text{for every ball }\,B_r\subset \R^n.
\]
\end{enumerate}

\end{Lm}
\begin{hproof3} Similar to \cite[Ch.~VI, Lem.~1]{Stein}, we construct $\psi \in L^\infty( [1,\infty))$ such that 
\begin{equation}\label{moment0}
   \int_1^\infty \psi(t) \,t^\ell \,dt= \begin{dcases} 1 \quad &\ell=0,\\
   0 \quad & \ell=1,\dots,k-1.       
    \end{dcases}
\end{equation}
In addition, for a small constant $\vae=\vae(\La)>0$, we choose $\psi$ satisfying \begin{gather}\label{addirespsi}
    \begin{dcases}
(i)\ \mathrm{supp}\,\psi \text{ is compact,}\\
(ii)\  \int_{1}^{\infty}\psi^{-}(t)\,dt \le \varepsilon.
\end{dcases}
\end{gather}
Following the proof of Stein's extension theorem~\cite[Ch.~VI, Thm.~5]{Stein} and using the weight $\psi$ constructed above, we cover $\p U$ by finitely many cylinders $\{U_{s}\}_{s=1}^N$ and obtain bounded linear extension operators $T_{U_s}\colon W^{k,p}(U)\to W^{k,p}(U_s)$ satisfying the following property: If $g=(g_{ij})$ satisfies~\eqref{asssobmet}, then for a.e. $x\in U_s$ and all $\xi\in \R^n$, it holds that
\begin{equation}\label{conTUs}
    \frac 23  \mkt  \La^{-1} \mkt  |\xi|^2\le T_{U_s}(g_{ij})(x)\mkt  \xi^{i}\xi^{j}\le \frac 43\mkt  \La\mkt  |\xi|^2.
\end{equation} 
Let $U_0\coloneqq U$. We choose a partition of unity $\{\eta_{s}\}_{s=0}^{N}\subset C_{c}^{\infty}(\R^{n})$ with
\[
\sum_{s=0}^{N}\eta_{s}\equiv 1\ \text{ on }\overline{U},\qquad \eta_s\ge 0\ \text{ in }\R^n,\qquad \mathrm{supp}\,\eta_{s}\subset U_{s}\quad(0\le s\le N).
\]
Define \begin{equation}\label{defome1}
    U'\coloneqq\Big\{x\in \R^n\colon \frac 34<\sum_{s=0}^N \eta_s(x)<\frac 32 \Big\}.
    \end{equation}
   Then we have $\overline{U}\subset U'$. Hence we can choose a cut-off function $\vartheta\in C_{c}^{\infty}(\R^n)$ satisfying 
\[
    0\le \vartheta\le 1 \text{ in }\R^n,\qquad   \vartheta\equiv1 \text{ on } \overline{U},\qquad \supp \vartheta \subset U'.
\]
   Now for any $g=(g_{ij})\in W^{k,p}(U,\R^{n\times n}_{\text{sym}})$, we define $\widetilde T g\in W^{k,p}(\R^{n},\R^{n\times n}_{\text{sym}})$ and $T g\in W^{k,p}_{\mathrm{loc}}(\R^{n},\R^{n\times n}_{\text{sym}})$ by
   \begin{align*}
&(\widetilde T g)_{ij}(x)=
\eta_0(x) g_{ij}(x)+\sum_{s=1}^N 
\eta_{s}(x)\mkt  T_{U_s} (g_{ij})(x),\quad &x\in \R^n,\,1\le i,j\le n,\\
&(Tg)_{ij}(x)=\vartheta(x) (\widetilde T g)_{ij}(x)+\big(1-\vartheta (x)\big)\delta_{ij}\fint_{U} g_{11} , \quad &x\in\R^n,\,1\le i,j\le n.
  \end{align*}
Then both $\wti T$ and $T$ are linear extension operators in the sense of~\eqref{Tisexten}. Moreover, combining~\eqref{conTUs}--\eqref{defome1} with the definiton of $T$ implies that $T$ also satisfies~\eqref{expreuniell}. Finally, the condition~\eqref{wkpbdext} follows from the proof of~\cite[Ch.~VI, Thm.~5']{Stein}.
\end{hproof3}
\section{Dirichlet problem for the Hodge Laplacian with $W^{1,n}$ metrics on bounded $C^1$ domains }
\label{sec:hod-dec}
In this section, we study the Dirichlet problem for the Hodge-Laplacian on domains in $\R^n$ ($n\ge 3$), equipped with a uniformly elliptic metric $g\in L^{\infty}\cap W^{1,n}$. 
The case $n=4$ will be used to solve \eqref{d*gL=*gV}--\eqref{dL0=0int} and the equations associated with the conservation laws in Section~\ref{sec:conlaws}. 
Our aim is to obtain right-inverse estimates for the Hodge--Dirac operator $d+d^{*_g}$ via elliptic estimates on differential forms. We prove a Poincar\'e-type inequality in Proposition \ref{prop-dform} for differential forms vanishing on the boundary of a Lipschitz domain, then derive elliptic estimates in Theorem \ref{th:dirdelsol} on $C^1$ domains, and finally prove two versions of the right-inverse estimates for $d+d^{*_g}$ in negative Sobolev--Lorentz spaces in Corollaries~\ref{co-Hod-decw-1p} and \ref{Co:Hod-declp}. We expect that a Hodge decomposition can be established in the framework of this section by combining existing methods.\\

Let $U\subset \R^n$ be a bounded $C^1$ domain, and let $g=(g_{ij})_{1\le i,j\le n}\in L^\infty\cap W^{1,n}(U,\R^{n\times n}_{\text{sym}})$. Suppose that there exists a constant $\La\ge 1$ such that for a.e. $x\in U$ and any $\xi=(\xi^1,\dots, \xi^n)$, it holds that \begin{align*}
    \La^{-1}|\xi|^2 \le g_{ij}(x)\xi^i\xi^j\le \La|\xi|^2.
\end{align*}
Then $g$ defines a metric on $U$. We adopt the convention~\eqref{con-formsob} for Sobolev spaces of differential forms, and define the operators $d^{*_g}$ and $\lap_g$ as in~\eqref{defd*glap}. Then for $p\in (\frac{n}{n-1},n)$, $\lap_g$ is a bounded linear operator from $W^{1,p}\big(U,\bwe^\ell \R^n\big)$ to $W^{-1,p}\big(U,\bwe^\ell \R^n\big)$, see the proof of Lemma \ref{lm-apriosys} together with \eqref{lap-rep}--\eqref{gbcdbd}. \\

Let $\ga=d\beta$ with $\beta\in W^{-1,p}\big(U,\bwe^\ell\R^{n}\big)$,  To solve \eqref{d*gL=*gV}--\eqref{dL0=0int}, we aim to find $\si \in W^{-1,p}\big(U,\bwe^{\ell}\R^n\big)$ solving the system
\begin{equation}\label{exaintdd*}
         \begin{dcases}
            d\mko \si=\ga ,\\
             d^{*_g}\si=0.
         \end{dcases}\qquad \text{in }U.
\end{equation}
Assume that $\beta$ admits a Hodge decomposition, namely, there exist $\tau_1\in L^p\big(U, \bwe^{\ell-1}\R^n\big)$, $\tau_2\in L^p\big(U,\bwe^{\ell+1}\R^n\big)$, and $\kappa\in W^{-1,p}\big(U, \bwe^\ell\R^n\big)$ with $d\kappa=0$ and $d*_g\kappa=0$ such that 
\begin{align}\label{eq:hogdec}
    \beta=d\tau_1+d^{*_g}\tau_2+\kappa.
\end{align}
Then $\si=d^{*_g}\tau_2$ solves \eqref{exaintdd*}. \\

On smooth compact Riemannian manifolds with boundary, the $L^2$ Hodge decomposition goes back to Friedrich~\cite{Friedrich55} and Morrey~\cite{Morrey56}. In 1966, Morrey~\cite{Morrey08} established the following $L^p$ decomposition into exact forms with vanishing tangential component, coexact forms with vanishing normal component, and harmonic forms, for $1<p<\nf$: 
\begin{equation}\label{clahoddec}
L^p\big(M,\bwe^\ell T^* M \big)=dW^{1,p}_T\big(M,\bwe^{\ell-1} T^* M \big)\oplus d^{*_g} W^{1,p}_N \big(M,\bwe^{\ell+1} T^* M\big)\oplus \mathcal{H}^{p}\big(M,\bwe^\ell T^* M \big).
\end{equation}
In 1995, Günter Schwarz~\cite{Schwarz95} further showed that if $\om\in W^{s,p}\big(M,\bwe^\ell T^*M\big)$ with $s\in \N$ and $1<p<\nf$, then in the decomposition $\om=d\tau_{1,\om}+d^{*_g}\tau_{2,\om}+\kappa_{\om}$ as above, we can choose $\tau_{1,\om}$ and $\tau_{2,\om}$ such that 
\[
    \|\tau_{1,\om}\|_{W^{s+1,p}(M)}+ \|\tau_{2,\om}\|_{W^{s+1,p}(M)}\le C(M,g)\mko \|\om\|_{W^{s,p}(M)}.
\]
This gives a complete solvability criterion for~\eqref{exaintdd*} with prescribed tangential boundary value (see~\cite[Ch.~3]{Schwarz95}). The $L^p$ Gaffney inequalities for smooth compact Riemannian manifolds (with and without boundary) were established in~\cite{Scott95, Iwaniec99}. Later, sharp Sobolev--Besov Hodge decompositions on Lipschitz domains in two and three dimensions were proved in~\cite{mitrea02-b,mitrea04}. In 2017, the Besov and Triebel--Lizorkin Hodge decompositions were obtained in~\cite{Ballest17}. For further study on $d+d^{*_g}$ with the flat Euclidean metric on Lipschitz domains in $\R^n$, see~\cite{Jakab09,Mcintosh18}. There are some other generalizations for the Hodge decomposition~\eqref{clahoddec} on smooth Riemannian manifolds, see~\cite{Muller15,  Amar17, Amar20, Murro24} for non-compact complete Riemannian manifolds, and~\cite{Kupferman25} for elliptic pre-complexes.  \\

 We note that if $\al$ satisfies $\lap_g \mko\al=\beta$, then setting $\tau_1=d^{*_g} \al$, $\tau_2=d\al$, and $\kappa=0$ provides a solution for \eqref{eq:hogdec}, and thus \eqref{exaintdd*} is solved. For this reason, we consider the Dirichlet problem: \begin{align}\label{dirpro}
    \begin{dcases}
        \Delta_g \mko\alpha = \beta \;\qquad\quad \text{in }U,\\
        \alpha\in W_0^{1,p}\big(U,\bwe^\ell \R^n \big).
    \end{dcases}
\end{align} 
When $g\in W^{2,r}$ with $r>n$ (hence $g_{ij}\in C^{1,\alpha}$ by Sobolev embedding), solvability and \textit{a priori} estimates for boundary value problems associated with $\Delta_g$ were proved in \cite{mitrea02-a,mitrea01}. In~\cite[Thm.~5.1]{mitrea02-a}, D.~Mitrea--M.~Mitrea proved that for any Lipschitz domain $U$, there exists $\vae=\vae (U)>0$ such that the following holds: If $2-\vae<p<2+\vae$ and $\beta \in L^p\big(U,\bwe^\ell T^* U\big)$ satisfies the standard compatibility conditions, then there exists a solution $\al$ of $\lap_g\mko \al=\beta$ with both $\al\big|_{\p U}$ and $d\al\big|_{\p U}$ tangential, satisfying
\[
\|dd^{*_g}\al\|_{L^p(U)}+\|d^{*_g}d\al\|_{L^p(U)}\le C \|\beta\|_{L^p(U)}.
\]
Using this result, they proved an analogue of the Hodge decomposition~\eqref{clahoddec} for $2-\vae<p<2+\vae$, see~\cite[Sec.~6]{mitrea02-a}.\\



In contrast to the above works, which assume at least $C^{\alpha}$ or Lipschitz metrics (and often smooth geometry), the present work proves the solvability of the Hodge Laplacian Dirichlet problem on differential forms under a much weaker metric regularity: namely a merely $L^{\infty}\cap W^{1,n}$ Riemannian metric (with $n=\dim M$), not even continuous a priori. \\

Writing $\alpha$, $\beta$ in the standard basis $\{dx^I\}$ of $\bwe T^*\R^n$, the equation~\eqref{dirpro} becomes an elliptic system of the form studied in the following lemma, as we verify after the proof of Lemma \ref{lm-apriosys}. For notations on Sobolev and Lorentz spaces, see Section \ref{sec:soblor}.
\begin{Lm}\label{lm-apriosys}
    Let $n,m\in\N$, $n\ge 3$, $n/(n-1)<p\le q<n$.
    Let $U\subset \R^n$ be a bounded open set with $C^1$ boundary, and let $a^{\mkt  ij}_{kl} \in L^\infty\cap W^{1,n}(U)$, $b^{\mkt  i}_{kl} , c^{\mkt  i}_{kl} \in L^n(U)$, $d^{\mkt  k}_l \in L^{n/2}(U)$ for all $1\le i,j\le n$, $1\le k,l\le m$.
   Assume there exists a constant $\La\ge 1$ such that for a.e. $x\in U$ and all $\xi=(\xi_i^k)\in \R^{m\times n}$, there holds
\begin{equation}\label{uniellcon}
     \La^{-1}|\xi|^2\le a^{\mkt  ij}_{kl} (x)\mkt \xi^k_i\mkt \xi^l_j \le \La \mkt  |\xi|^2.
   \end{equation} Let $\omega\colon (0,\infty)\rightarrow [0,\infty)$ be a function satisfying $\lim_{r\rightarrow 0} \omega (r)=0$. Assume that for any $1\le i,j\le n$, $1\le k,l\le m$, and any ball $B_r\subset\R^n$ of radius $r$, the coefficients satisfy
   \begin{equation}
    \big\|\mkt  |\g a^{\mkt  ij}_{kl} |+|b^{\mkt  i}_{kl} |+|c^{\mkt  i}_{kl} |\mkt  \big\|_{L^n(B_r\cap U)}+\big \|d^{\mkt  k}_l\big\|_{L^{\frac n2}(B_r\cap U)}\le \omega(r).\label{dabcwbound}
    \end{equation}
We define the operator $L\colon W^{1,p}(U,\R^m)\rightarrow W^{-1,p}(U,\R^m)$ by setting, for $\vec u=(u^1,\dots, u^m)\in W^{1,p}(U,\R^m)$,
\begin{equation}\label{def-oper}
    (L\vec u)^k\coloneqq \p_i(a^{\mkt  ij}_{kl} \mkt  \p_ju^l)+\p_i(b^{\mkt  i}_{kl}  \mkt  u^l)+c^{\mkt  i}_{kl}  \mkt  \p_i u^l+d^{\mkt  k}_l u^l,\qquad  1\le k\le m.
\end{equation}
Let $\vec f\in W^{-1,q}(U,\R^m)\subset W^{-1,p}(U,\R^m)$. Suppose $\vec u\in W^{1,p}_0(U,\R^m)$ solves the elliptic system $L\vec u=\vec f$.
    Then we have $\vec u\in W^{1,q}_0(U,\R^m)$ along with the \textit{a priori} estimate \begin{equation}\label{apriori-sys}
        \|\g \vec u\|_{L^q(U)}\le C(\La, U, \omega,q,m)\mkt  \big(\|\vec u\|_{L^1(U)}+\|\vec f\|_{W^{-1,q}(U)}\big).
    \end{equation}
\end{Lm}
\begin{proof}
Let 
\begin{equation}\label{defss0}
    \frac{n}{n-1}<s\le s_0\coloneqq \min \Big(q,\frac{np}{n-p}\Big).
\end{equation}
    We define the operator $L_0\colon W_0^{1,s}(U,\R^m)\to W^{-1,s}(U,\R^m)$ by 
    $$
    (L_0\mkt  \vec w)^k\coloneqq \p_i(a^{\mkt  ij}_{kl} \mkt  \p_jw^l),\qquad 1\le k\le m.
    $$ 
 By Poincar\'e inequality, the VMO modulus of $a_{kl}^{\mkt ij}$ satisfies
\[
  \sup_{B_r\subset U}\fint_{B_r}\Big|a_{kl}^{\mkt ij} -\fint_{B_r}a_{kl}^{\mkt ij} \Big| \le C(n) \sup_{B_r\subset \R^n}\big\|\g a^{\mkt  ij}_{kl}  \big\|_{L^n(B_r\cap U)}\le C(n)\mkt \omega(r)\xrightarrow[r\to 0]{}0.
\]
Then by $W^{1,p}$ regularity theory for divergence-form elliptic systems with VMO coefficients (see for instance \cite[Thm.~1.7]{Byun08}), there exists a positive constant $C_s$ depending only on $\La,\, U,\, s,\, \omega$, such that for any $\vec w\in W^{1,s}_0(U,\R^m)$, there holds
\begin{equation}\label{w<cl0w}
\|\vec w\|_{W^{1,s}_0(U)}\le C_s\|L_0\mkt  \vec w \|_{W^{-1,s}(U)}.
\end{equation}
Hence $L_0$ is an isomorphism between the Banach spaces $W^{1,s}_0(U,\R^m)$ and $W^{-1,s}(U,\R^m)$.

Fix $x_0\in\overline U$, and $r>0$ small enough (depending only on $\La,\, U,\, \omega,\, s,\, m$). We define the Banach space
    $\mathscr B_{s,r}\coloneqq W^{1,s}\cap L^{\frac{ns}{n-s}}(B_r(x_0)\cap U,\R^m)$, equipped with the norm 
    \[
        \|\vec w\|_{\mathscr B_{s,r}}\coloneqq\|\g \vec w\|_{L^s(B_r(x_0)\cap U)}+\|\vec w\|_{L^{\frac{ns}{n-s}}(B_r(x_0)\cap U)}.
   \]
 Let $\zeta$ be a $C^\infty$ function supported in $B_r(x_0)$ depending on $r,x_0$ only such that $\zeta\equiv 1$ on $B_{r/2}(x_0)$. Let $\vec v\coloneqq\zeta \vec u$, then we have $\vec v\in \mathscr B_{p,r}$. Now it suffices to show that $\vec v\in \mathscr B_{q,r}$ with the corresponding norm bounded by the right-hand side of \eqref{apriori-sys}.
 
Applying $L_0^{-1}$ to both sides of the equation $L_0 \vec v=(L_0-L)\vec v+L\vec v$, we obtain \[
    \vec v=L_0^{-1}(L_0-L)\vec v+L_0^{-1}L\vec v.
\]
We show that for $r$ small enough, it holds that
\begin{equation}\label{V=Tv+h}
    \|L_0^{-1}(L_0-L)\|_{\mathscr B_{s,r}\to \mathscr B_{s,r} }\le \frac 12.
\end{equation}
Let $\vec w\in \mathscr B_{s,r}$. A priori $\bw$ is only defined on $B_r(x_0)\cap U$. Extending $b^{\mkt  i}_{kl} w^l$, $c^{\mkt  i}_{kl} \mkt  \p_i w^l$, and $d^{\mkt  k}_lw^l$ by $0$ on $U\setminus B_r(x_0)$, we define $(L_0-L)\vec w\in W^{-1,s}(U,\R^m)$ as follows.  By H\"older's inequality, then we estimate \begin{align}\begin{aligned}\label{estbwl}
    \|\p_i(b^{\mkt  i}_{kl} w^l)\|_{W^{-1,s}(U)}&\le \|b^{\mkt  i}_{kl} w^l\|_{L^s(U\cap B_r(x_0))}\\
    &\le \|b^{\mkt  i}_{kl} \|_{L^n(U\cap B_r(x_0))}   \|w^l\|_{L^{\frac{ns}{n-s}}(U \cap B_r(x_0))}
   \\
   &\le \omega(r) \|\vec w\|_{\mathscr B_{s,r}}.
\end{aligned}
\end{align}
Using the embedding $L^{\frac{ns}{n+s}}(U)\hookrightarrow W^{-1,s}(U)$, we also obtain \begin{align}
\begin{aligned}\label{estcpiwl}
        \|c^{\mkt  i}_{kl} \mkt  \p_i w^l\|_{W^{-1,s}(U)}&\le C(n,s) \|c^{\mkt  i}_{kl} \mkt  \p_i w^l\|_{L^{\frac{ns}{n+s}}(U)}\\
        &\le C(n,s) \|c^{\mkt  i}_{kl} \|_{L^n(U\cap B_r(x_0))}\mkt  \|\p_i w^l\|_{L^{s}(U\cap B_r(x_0))}\\
        &\le C(n,s) \mkt  \omega(r) \|\vec w\|_{\mathscr B_{s,r}}.
    \end{aligned}
\end{align}
Similarly, we have \begin{align*}
    \|d^{\mkt  k}_l w^l\|_{W^{-1,s}(U)}&\le C(n,s)\|d^{\mkt  k}_l w^l\|_{L^{\frac{ns}{n+s}}(U)}\\
    &\le C(n,s)\|d^{\mkt  k}_l \|_{L^\frac n2(U\cap B_r(x_0))}\mkt  \| w^l\|_{L^{\frac{ns}{n-s}}(U\cap B_r(x_0))}\\
        &\le C(n,s) \mkt  \omega(r) \|\vec w\|_{\mathscr B_{s,r}}.
\end{align*}
Consequently, we have \begin{align}\begin{aligned}\label{estL0-L}
    \|(L_0-L)\vec w\|_{W^{-1,s}(U)}&\le \sum_{k}\|\p_i(b^{\mkt  i}_{kl}  \mkt  w^l) +c^{\mkt  i}_{kl} \mkt  \p_i w^l+d^{\mkt  k}_l  w^l\|_{W^{-1,s}(U)}\\
   & \le C(n,s,m)\mkt  \omega(r)\|\vec w\|_{\mathscr B_{s,r}}.
    \end{aligned}
\end{align}
In particular, when $\vec w=\vec v$ and $s=p$, this definition coincides precisely with the standard definition of $(L_0-L)\vec v$ since $\vec v\in W_0^{1,p}(U\cap B_r(x_0),\R^m)$.
Using the estimates \eqref{w<cl0w} and \eqref{estL0-L}, we obtain a constant $C_s'$ independent of $r$ such that  \begin{equation}\label{l0l0-lw}\begin{aligned}
    \|L_0^{-1}(L_0-L)\vec w\|_{W^{1,s}_0(U)}&\le C_s\mkt  \|(L_0-L)\vec w\|_{W^{-1,s}(U)}\\
    &\le C_s'\, \omega(r)\|\vec w\|_{\mathscr B_{s,r}}.
    \end{aligned}
\end{equation}
Now we define a linear operator $T$ on $\mathscr B_{s,r}$ by \begin{align*}
    T\vec w\coloneqq(L_0^{-1}(L_0-L)\vec w)\big |_{B_r(x_0)\cap U}.
\end{align*} 
By Sobolev embedding and \eqref{l0l0-lw}, there exists a constant $C_s''$ independent of $r$ such that \begin{align}\label{Tnorm<w}
    \|T\|_{\mathscr B_{s,r}\to\mathscr B_{s,r}}\le C_s''\mkt  \omega(r).
\end{align} Then we choose $r_s>0$ depending on $\La,\, U,\, \omega,\, m,\, s$ only such that \begin{align}\label{csw<1/2}
    C_s''\mkt \omega(r)\le \frac 12,\qquad \text{for all }\, r\le r_s.
\end{align} Let $r_0=\min(r_{p},r_{s_0})$, $\vec h\coloneqq(L_0^{-1}L\vec v)|_{U\cap B_{r_0}(x_0)}$. By direct computation, for $1\le k\le m$, we have \begin{align}\label{gxu<Lxu}\begin{aligned}
    \big(L\vec v\big)^k&=\big(L(\zeta \vec u)\big)^k\\
    &=\zeta \big(L\vec u\big)^k+\p_i\zeta\big(a^{\mkt  ij}_{kl} \mkt  \p_ju^l +b^{\mkt  i}_{kl} \mkt  u^l +c^{\mkt  i}_{kl} \mkt  u^l\big)+\p_i\big(a^{\mkt  ij}_{kl}  \mkt  \p_j\zeta\, u^l\big).
    \end{aligned}
\end{align}
For the first term on the right-hand side, there exists a positive constant $C_1$ depending only on $\La$, $U$, $\omega$, $m$, $p$, $q$ such that \begin{equation}\label{xtLu}
    \|\zeta\mkt  L\vec u\|_{W^{-1,q}(U)}=\|\zeta \mkt  \vec f\|_{W^{-1,q}(U)}\le C_1 \mkt \|\vec f\|_{W^{-1,q}(U)}.
\end{equation}
Let $s$ be as in \eqref{defss0}. Then since $s\le np/(n-p)$ and $\vec u\in W^{1,p}(U)$, by Sobolev embedding we have $\vec u\in L^s(U)$. 

In the following, $C$ will denote a positive constant depending only on $\La$, $U$, $\omega$, $m$, $q$, $s$. Applying the identity $\p_i\zeta\, a^{\mkt  ij}_{kl} \mkt  \p_ju^l =\p_j(\p_i\zeta\,a^{\mkt  ij}_{kl} u^l)-u^l\,\p_j(\p_i\zeta\,a^{\mkt  ij}_{kl} )$,
we estimate as in \eqref{estbwl}--\eqref{estcpiwl}: 
\begin{align}
    \begin{aligned}\label{akpju}
        &\|\p_i\zeta\, a^{\mkt  ij}_{kl} \mkt  \p_ju^l\|_{W^{-1,s}(U)}\\&\le C(n,s)\big( \|\p_i\zeta\,a^{\mkt  ij}_{kl} \mkt  u^l\|_{L^s(U)}+\mkt \|u^l\,\p_j(\p_i\zeta\,a^{\mkt  ij}_{kl} )\|_{L^{\frac{ns}{n+s}}(U)}\big)\\
        &\le C(n,s)\mkt \|u^l\|_{L^s(U)}\big( \|\g \zeta\,a^{\mkt  ij}_{kl} \|_{L^\infty(U)}+\|\g(\p_i\zeta\,a^{\mkt  ij}_{kl} )\|_{L^n(U)}\big)\\
        &\le C\mkt  \|\vec u\|_{L^s(U)}.
    \end{aligned}
\end{align}
The remaining terms are estimated similarly:
\begin{align}\begin{aligned}
\label{pxbc}
    \|\p_i\zeta\mkt  (b^{\mkt  i}_{kl} +c^{\mkt  i}_{kl} )\mkt  u^l\|_{W^{-1,s}(U)}&\le C(n,s)  \|\p_i\zeta\mkt  (b^{\mkt  i}_{kl} +c^{\mkt  i}_{kl} )\mkt  u^l\|_{L^{\frac{ns}{n+s}}(U)}\\
    &\le C(n,s)\mkt  \|\p_i\zeta(b^{\mkt  i}_{kl} +c^{\mkt  i}_{kl} )\|_{L^n(U)}\mkt  \|u^l\|_{L^s(U)}\\
    &\le C \mkt  \|\vec u\|_{L^s(U)}.
    \end{aligned}
\end{align}
We also have
\begin{equation}
\label{piakpjuk}
         \|\p_i(a^{\mkt  ij}_{kl} \mkt  \p_j\zeta\,  u^l)\|_{W^{-1,s}(U)}\le \|a^{\mkt  ij}_{kl} \mkt  \p_j\zeta\, u^l\|_{L^s(U)}
         \le C \mkt  \|\vec u\|_{L^s(U)}.
\end{equation}
Combining \eqref{gxu<Lxu}--\eqref{piakpjuk} yields
   $$\|L\vec v\|_{W^{-1,s}(U)}\le C\big(\|\vec f\|_{W^{-1,q}(U)}+\|\vec u\|_{L^s(U)}\big).$$
Then by applying \eqref{w<cl0w} as in the proof of \eqref{Tnorm<w}, we obtain 
\begin{align}\begin{aligned}\label{h<cfu}
    \|\vec h\|_{\mathscr B_{s,r_0}}&\le C(n,s) \mko\|L_0^{-1}L\vec v\|_{W^{1,s}_0(U)}\\
    &\le C\mkt \|L\vec v\|_{W^{-1,s}(U)}\le C\big(\|\vec f\|_{W^{-1,q}(U)}+\|\vec u\|_{L^s(U)}\big).
    \end{aligned}
\end{align}
In particular, we have $\vec h\in \mathscr B_{s_0,r_0}$. By \eqref{Tnorm<w}, \eqref{csw<1/2}, and definition of $r_0$, we have \begin{gather}\begin{aligned}
\label{Tcont}
    \|T\|_{\mathscr B_{p,r_0}\to\mathscr B_{p,r_0} }\le \frac 12\qquad\text{and}\qquad 
    \|T\|_{\mathscr B_{s_0,r_0}\to\mathscr B_{s_0,r_0} }\le \frac 12.
\end{aligned}    
\end{gather} 
Hence $T$ is a contraction on both $\mathscr B_{p,r_0}$ and $\mathscr B_{s_0,r_0}$. Now by restricting both sides of \eqref{V=Tv+h} on $U\cap B_{r_0}(x_0)$, we have $ \vec v=T\vec v+\vec h$. 
The contraction mapping principle \cite[Thm.~5.1]{Gilbarg01} then implies that there exists a unique solution to the equation $\vec w=T\vec w+\vec h$ in either $\mathscr B_{s_0,r_0}$ or $\mathscr B_{p,r_0}$. Since we have $\mathscr B_{s_0,r_0}\subset \mathscr B_{p,r_0}$ and $\vec v\in \mathscr B_{p,r_0}$, the solution in $\mathscr B_{s_0,r_0}$ coincides with $\vec v$. Moreover, by setting $s=s_0$ in \eqref{h<cfu} and applying \eqref{Tcont}, we have the estimate \begin{align*}
   \|\g \vec u\|_{L^{s_0}(U\cap B_{r_0/2}(x_0))}&\le \|\vec v\|_{\mathscr B_{s_0,r_0}}\\
   &\le 2\|\vec v-T\vec v\|_{\mathscr B_{s_0,r_0}}\\
   &=2\|\vec h\|_{\mathscr B_{s_0,r_0}}\le C_2(\|\vec f\|_{W^{-1,q}(U)}+\|\vec u\|_{L^{s_0}(U)}).
\end{align*}
Covering $\overline U$ by finitely many balls $B_{r_0/2}(x_0)$, we obtain \begin{align*}
    \|\g \vec u\|_{L^{s_0}(U)}\le C_3(\|\vec f\|_{W^{-1,q}(U)}+\|\vec u\|_{L^{s_0}(U)}),
\end{align*}
where $C_2,\,C_3$ depend only on $\La,\, U,\, \omega,\, m,\, s_0,\, q$. If $s_0=q$, then we immediately obtain the estimate \begin{align}\label{gulqfu}
    \|\g \vec u\|_{L^{q}(U)}\le C_4(\|\vec f\|_{W^{-1,q}(U)}+\|\vec u\|_{L^{q}(U)}).
\end{align}
Otherwise, we have $\vec u\in W^{1,s_0}_0(U)$ with $\frac 1{s_0}=\frac 1p-\frac 1n$. In this case, we can replace $p$ by $np/(n-p)$ and repeat the preceding argument. After finitely many iterations, we again arrive at the desired estimate \eqref{gulqfu}, where the constant $C_4$ depends only on $\La,\, U, \,\omega,\, m,\, q$. The estimate~\eqref{apriori-sys} then follows from standard interpolation inequality (see for instance \cite[Thm.~II]{Nirenberg59}):\[
\|\vec u\|_{L^q(U)}\le \vae\mkt  \|\g\vec u\|_{L^q(U)}+C(n,q)\mkt  \vae^{-\frac{n(q-1)}{q}} \|\vec u\|_{L^1(U)}.\qedhere
\]
\end{proof}
Let $U\subset \R^n$ be a bounded domain, $g=(g_{ij})\in W^{1,n}(U,\R^{n\times n}_{\text{sym}})$ be a metric satisfying the uniform ellipticity condition 
\begin{align}\label{ellimet-hod}
    \La^{-1}|\xi|^2 \le g_{ij}(x)\xi^i\xi^j\le \La|\xi|^2,\qquad \text{for a.e. }x\in U\,\text{ and all }\,\xi\in \R^n,
\end{align}for some constant $\La\ge 1$. We denote $\det g$, $g^{ij}$, $d\textup{vol}_g$ as in Notation \eqref{not-met}. 

For $x\in U$, we define $|\g g(x)|^2\coloneqq \sum_{i,j,s} |\p_s\mkt   g_{ij}(x)|^2$. For an $\ell$-form $\alpha$ on $U$, we write $\alpha=\sum_I\alpha_I \, dx^I$,
where $I$ ranges over all strictly increasing multi-indices of length $\ell$. Then by~\cite[Eq.~4.11]{mitrea01}, there exist coefficient tensors $b_I^{\mkt  iJ},\,c_I^{\mkt  iJ},\,d_I^J$ depending only on the metric \(g_{ij}\) such that for any $\ell$-form $\al$, we have
\begin{equation}
\begin{aligned}\label{lap-rep}
(\Delta_{g}\alpha)_{I}&=\sum_{i,j=1}^n \partial_{i}\bigl(g^{ij}\partial_{j}\alpha_{I}\bigr)
+\sum_{i=1}^n\sum_{|J|=\ell} \partial_{i}\bigl(b_{I}^{\mkt  iJ}\alpha_{J}\bigr)\\
&\quad+\sum_{i=1}^n\sum_{|J|=\ell} c_{I}^{\mkt  iJ}\partial_{i}\alpha_{J}
    +\sum_{|J|=\ell} d_{I}^{\mkt  J}\alpha_{J}.
    \end{aligned}
\end{equation}
Moreover, there exists $C=C(\La)>0$ such that   
\begin{align*}
|b_{I}^{\mkt  iJ}|+|c_{I}^{\mkt  iJ}|\le C\mkt  |\nabla g|
\quad\text{and}\quad
|d_{I}^{\mkt  J}|\le C\mkt  |\nabla g|^{2} \quad \text{on }U.
\end{align*}
Consequently, for every ball \(B_r\subset\R^{n}\), we have
\begin{align}\label{gbcdbd}\begin{aligned}
    &\big\|\mkt  |\g g^{ij}|+|b_I^{\mkt  iJ}|+|c_I^{\mkt  iJ}|\mkt  \big\|_{L^n(B_r\cap U)}+\big \|d_I^J\big\|^{1/2}_{L^{n/2}(B_r\cap U)}\\
    &\le C(\La)\mkt  \big\|\g g\big\|_{L^n(B_r\cap U)}.
    \end{aligned}
\end{align}

Therefore, Lemma \ref{lm-apriosys} applies to the Laplace–Beltrami operator acting on differential forms whenever the underlying metric 
$g$ belongs to $W^{1,n}$ and satisfies the uniform ellipticity condition \eqref{ellimet-hod}. Before turning to the applications, we prove some fundamental inequalities for products of distributions.

 Let $n\ge 3$, $p\in (\frac{n}{n-1},n)$, and $U\subset \R^n$ be a bounded Lipschitz domain. Then by the embedding results in Lemma~\ref{lm:LpembW-1}, for all $1\le j\le n$, $a\in L^\nf \cap W^{1,n}(U)$, and $f\in L^{p}(U)$, we have 
\begin{align*}
    \|a \mkt \p_{j} f  \|_{W^{-1,p}(U)}
    &\le \|\p_{j}(af)\|_{W^{-1,p}(U)}+C(n,p)\|f\mkt\p_{j} a \|_{L^\frac{np}{n+p}(U)}\\
    &\le \|af\|_{L^{p}(U)}+C(n,p)\|\p_{j} a\|_{L^n(U)} \|f\|_{L^{p}(U)}\\
    &\le C(n,p) \|a\|_{L^\nf \cap W^{1,n}(U)} \mko\|f\|_{L^{p}(U)}.
\end{align*}
It follows that for all $a\in L^\nf \cap W^{1,n}(U)$ and $T\in W^{-1,p}(U)$,
\begin{align}\label{prornd}
    \|aT\|_{W^{-1,p}(U)}\le C(n,p)\mko\|a\|_{L^\nf \cap W^{1,n}(U)} \mko \|T\|_{W^{-1,p}(U)}.
\end{align}  
By the Sobolev embedding, we also obtain that for all $a\in L^\nf \cap W^{1,n}(U)$ and $f\in W^{1,p}(U)$,
\begin{align}\label{prolorgend}
     \|af\|_{W^{1,p}(U)}\le C(n,p,U)\mko\|a\|_{L^\nf \cap W^{1,n}(U)} \mko \|f\|_{W^{1,p}(U)}.
\end{align}

In the remainder of this section, we fix $\La\ge 1$ and a function $\omega\colon (0,\infty)\to [0,\infty)$ with $\lim_{r\rightarrow 0} \omega (r)=0$. We consider metrics $g=(g_{ij})\in L^\infty\cap W^{1,n}(U,\R^{n\times n}_{\sym})$ satisfying: 
\begin{align}\label{w1nassg}
    \begin{dcases}
     \Lambda^{-1}\mkt  |\xi|^{2}\le g_{ij}(x)\mkern2mu \xi^{i}\xi^{j}\le \Lambda\mkt  |\xi|^{2},
        \quad&\text{for a.e. } x\in U\text{ and all }\mkt \xi\in\R^{n}, \\
         \|\g g\|_{L^n(B_r\cap U)}\le \om(r), &\text{for every ball }\mkt B_r\subset \R^n.
    \end{dcases}
\end{align}
Combining~\eqref{prornd}--\eqref{w1nassg}, we then obtain
\begin{subnumcases}{}
    \| *_{g}\|_{W^{-1,p}\lf(U,\scriptstyle\bigwedge^\ell \R^n\rg)\to W^{-1,p}\lf(U,\scriptstyle\bigwedge^{n-\ell} \R^n\rg)}\le C(\La,U,p,\om),\label{*estlornd}\\
    \| *_{g}\|_{W^{1,p}\lf(U,\scriptstyle\bigwedge^\ell\R^n\rg)\to W^{1,p}\lf(U,\scriptstyle\bigwedge^{n-\ell}\R^n\rg)}\le C(\La,U,p,\om).\label{*estW1pnd}
\end{subnumcases} 
Consequently, since $\lap_g=-(dd^{*_g}+d^{*_g}d)$, we have
\begin{align}
      \| \lap_{g}\|_{W^{1,p}(U,\bwe \R^n)\to W^{-1,p}(U,\bwe \R^n)}\le C(\La,U,p,\om).\label{*estlapnd}
\end{align}

Under the assumptions \eqref{w1nassg}, Lemmas~\ref{lm-apriosys} implies a Poincar\'e-type inequality for differential forms in $W^{1,p}_0$, where $p\in (\frac n{n-1},n)$.
\begin{Prop}\label{prop-dform}
    Let $n\ge 3$, $p\in (\frac{n}{n-1},n)$, and $U\subset \R^n$ be a bounded Lipschitz domain. 
    Suppose $g=(g_{ij})$ satisfies~\eqref{w1nassg}.
Then for any differential $\ell$-form $\alpha\in W^{1,p}_0\big(U,\bwe^\ell \R^n\big)$ with $0\le \ell\le n$, we have
    \begin{equation}\label{apriori-form}
        \|\alpha\|_{W^{1,p}_0(U)}\le C(\La,U,p,\omega)\big(\|d\mko \alpha\|_{L^{p}(U)}+\|d^{*_g}\alpha\|_{L^p(U)}\big).
    \end{equation}
\end{Prop}

\begin{proof}
    If $\ell=0$, then $\eqref{apriori-form}$ follows from Poincar\'e inequality. We now assume $1\le \ell\le n$.\\
    By Lemma~\ref{lm-ext-metr}, the metric $g$ admits an extension $\tilde g \coloneqq T g\in W^{1,n}_{\textup{loc}}(\R^n,\R^{n\times n}_{\textup{sym}})$ such that
\begin{align*}
        \frac 12 \mkt  \Lambda^{-1}|\xi|^{2}\le \tilde g_{ij}(x)\mkt  \xi^{i}\xi^{j}\le 2\mkt  \Lambda\mkt  |\xi|^{2},
     \qquad\text{for a.e. } x\in \R^n\text{ and all }\mkt \xi\in\R^{n}.
    \end{align*}
Moreover, there exists a function $\widetilde \omega\colon (0,\infty)\to [0,\infty)$ depending only on $U$, $\omega$, and $\La$, with $\lim_{r\to 0} \widetilde \omega(r)=0$, such that
\[
     \|\g \tilde g\|_{L^n(B_r)}\le \wti \om(r),\qquad \text{for every ball }\mkt B_r\subset \R^n.
\]
Let $B$ be a fixed open ball containing $U$; all constants depending on $B$ may thus be regarded as depending on $U$. Then $W_0^{1,p}\big(U, \bwe^\ell\R^n\big)$ embeds naturally into $W_0^{1,p}\big(B, \bwe^\ell\R^n\big)$, and for any $\al\in W_0^{1,p}\big(U, \bwe^\ell\R^n\big)$, we have 
\[
    \|d\mko \al\|_{L^p(U)}=\|d\mko \al\|_{L^p(B)},\qquad \|d^{*_g}\al\|_{L^p(U)}= \|d^{*_{\tilde g}}\al\|_{L^p(B)}.
\]
Hence it suffices to prove the estimate on $B$. Suppose that \eqref{apriori-form} fails. Then by the preceding discussion, there exist a sequence of differential forms $\{\alpha_k\}_{k=1}^\infty\subset W^{1,p}_0\big(B,\bwe^\ell\R^n\big)$ and a sequence of metrics $\{g_k\}_{k=1}^\infty\subset W^{1,n}(B,\R^{n\times n}_{\textup{sym}})$ such that:
\begin{enumerate}[(i)]
       \item For a.e. $x\in B$, any $k\in \N$, and all $\xi\in\R^{n}$, it holds that 
       \begin{align}\label{gaffneyuniellseq}
            \frac 12 \mkt  \Lambda^{-1}\mkt  |\xi|^{2}\le g_{k,ij}(x)\mkern2mu \xi^{i}\xi^{j}\le  2\mkt  \Lambda\mkt  |\xi|^{2}.
       \end{align}  
     \item  For each $k\in \N$, it holds that
     \begin{align}\label{gaffLnbdseq}
            \|\g g_k\|_{L^n(B_r\cap B)}\le \wti \om(r),\qquad \text{for every ball }\mkt B_r\subset \R^n.
     \end{align}
    \item The sequence $\{\al_k\}$ satisfies
    \begin{align}\label{contra-hyp}
             \|\g \alpha_k\|_{L^{p}(B)}=1,\qquad\|d\mko \alpha_k\|_{L^p(B)}+\|d^{*_{g_k}}\alpha_k\|_{L^p(B)}\xrightarrow[k\to\infty]{} 0.
    \end{align}
     \end{enumerate}
Since the representation \eqref{lap-rep} and the bound \eqref{gbcdbd} for the coefficients hold uniformly for each metric $g_k$, the operators $\Delta_{{g_k}}=-(dd^{*_{g_k}}+d^{*_{g_k}}d)$ satisfy the hypotheses of Lemma \ref{lm-apriosys}. Hence the lemma yields a constant $C_0=C_0(\La,U,p,\omega)>0$ such that for any $k$, there holds
\begin{align}\label{estpriseq}
        \|\g \al_k\|_{L^p(B)}\le C_0\big(\|\al_k\|_{L^1(B)}+\|\Delta_{g_k} \al_k\|_{W^{-1,p}(B)}\big).
    \end{align}
By~\eqref{contra-hyp} and~\eqref{*estlornd}, we obtain 
\begin{align}\begin{aligned}\label{lapgkalbd}
       \|\Delta_{g_k}\alpha_k\|_{W^{-1,p}(B)}&\le \|dd^{*_{g_k}}\alpha_k\|_{W^{-1,p}(B)}+\|*_{g_k}d*_{g_k}d\mko \alpha_k\|_{W^{-1,p}(B)}\\
       &\le  C(\La,U,p,\omega) \big(\|d^{*_{g_k}}\alpha_k\|_{L^p(B)}+\|d\mko \alpha_k\|_{L^p(B)}\big)
       \\ 
       &\xrightarrow[k\rightarrow\infty]{} 0.
         \end{aligned}
   \end{align}
Inserting \eqref{lapgkalbd} into \eqref{estpriseq} and using \eqref{contra-hyp}, we deduce
   \begin{align}\label{lbdalk}
       \liminf_{k\rightarrow \infty} \|\al_k\|_{L^1(B)}\ge C_0^{-1}.
   \end{align}
Since $\{\al_k\}$ and $\{g_{k}\}$ are bounded in $W_0^{1,p}\big(B,\bwe^\ell \R^n\big)$ and $W^{1,n}(B,\R^{n\times n}_{\text{sym}})$ respectively, there exist subsequences (still denoted by $\{\alpha_{k}\}$, $\{g_{k}\}$) and $\al\in W_0^{1,p}\big(B,\bwe^\ell \R^n\big)$, $g=(g_{ij})\in W^{1,n}(B,\R^{n\times n}_{\text{sym}})$ such that 
\begin{equation}\label{weakconag}
    \al_{k}\rightharpoonup \al\,\text{ in }W_0^{1,p}\big(B, \bwe^\ell \R^n\big),\qquad g_{k}\rightharpoonup g \,\text{ in }W^{1,n}(B, \R^{n\times n}_{\text{sym}}).
\end{equation} 
The Rellich–Kondrachov compact embedding then implies 
\begin{align}\label{strconag}
    \al_{k}\to \al\,\text{ in }L^{p}\big(B, \bwe^\ell \R^n\big),\qquad g_{k}\to g \,\text{ in }L^n(B, \R^{n\times n}_{\text{sym}}).
\end{align} 
After possibly passing to a subsequence, we may also assume $g_k \to g$ a.e. in $B$. Since $\nabla g_k \rightharpoonup \nabla g$ in $L^n(B)$, by \eqref{gaffneyuniellseq}--\eqref{gaffLnbdseq} and the weak lower semicontinuity of the $L^n$ norm, we obtain
\vspace{0.5ex}
\begin{align*}
\begin{dcases}
     \frac 12 \mkt  \Lambda^{-1}\mkt  |\xi|^{2}\le g_{ij}(x)\mkern2mu \xi^{i}\xi^{j}\le  2\mkt  \Lambda\mkt  |\xi|^{2},
        \qquad&\text{for a.e. } x\in B\text{ and all }\mkt \xi\in\R^{n}, \\
         \|\g g\|_{L^n(B_r\cap B)}\le \wti \om(r), &\text{for every ball }\mkt B_r\subset \R^n.
         \end{dcases}
\end{align*} 
Moreover, the convergences \eqref{weakconag}--\eqref{strconag} together with H\"older's inequality imply that \begin{align*}
    d\mko \al_{k}\rightharpoonup d\mko \al\,\text{ in }L^{p}\big(B, \bwe^{\ell+1} \R^n\big),\qquad d^{*_{g_k}}\al_{k}\rightharpoonup d^{*_{g}}\al \,\text{ in }L^{p}\big(B, \bwe^{\ell-1} \R^n\big).
\end{align*}
Hence, by the weak lower semicontinuity of the $L^p$ norm and the assumption \eqref{contra-hyp}, we have \begin{align*}
    &\|d\mko \alpha\|_{L^p(B)}\le \liminf_{k\rightarrow\infty} \|d\mko \al_k\|_{L^p(B)}=0,\\
    &\|d^{*_{g}}\al\|_{L^p(B)}\le \liminf_{k\rightarrow\infty} \|d^{*_{g_k}}\al_{k}\|_{L^p(B)}=0.
\end{align*}
Thus $d\mko \al=0$ and $d^{*_g}\al=0$, hence $\lap_g\mko \al=0$. Lemma~\ref{lm-apriosys} then implies $\al\in W^{1,2}_0\big(B,\bwe^\ell \R^n\big)$. By \cite[Cor.~3.4]{Costa10}, there exists $\wti \al\in W^{2,2}\big(B,\bwe^{\ell-1} \R^n\big)$ such that $d\wti \al=\al$. Using $d*_g\mko \al=0$ in $B$, we obtain
\begin{align}
\begin{aligned}
   \label{inbypa1n}
    \int_B \lan \al,\al \ran_g \,d\textup{vol}_g&=  \int_B \lan d \wti\al,\al \ran_g \, d\textup{vol}_g\\
    &=\int_B d\wti \al \we *_g\mko \al
    =\int_B d(\wti \al \we *_g\mko \al).
    \end{aligned}
\end{align}
By Stokes' theorem (valid for $W^{1,1}$ forms), the last integral vanishes since $\al\in W^{1,2}_0\big(B,\bwe^\ell \R^n \big)$. Consequently $\al=0$ a.e. in $B$. However, the lower bound~\eqref{lbdalk} and the convergence~\eqref{strconag} imply that $\|\al\|_{L^1(B)}\ge C_0^{-1}>0$, which is a contradiction.
\end{proof}
\begin{Th}\label{th:dirdelsol}
     Let $n\ge 3$, $p\in(\frac{n}{n-1},n)$, and $U\subset \R^n$ be a bounded $C^1$ domain. Suppose $g=(g_{ij})\in L^\infty\cap W^{1,n}(U,\R^{n\times n}_{\sym})$ satisfies \eqref{w1nassg}. Then for any $\beta\in W^{-1,p}\big(U,\bwe^\ell \R^n\big)$ with $0\le \ell\le n$, there exists a unique $\al$ solving the following Dirichlet problem:
     \begin{align}\label{eq-dirpro}
    \begin{dcases}
        \Delta_g \mko \alpha = \beta \qquad\quad \text{in }U,\\
        \alpha\in W_0^{1,p}\big(U,\bwe^\ell \R^n\big).
    \end{dcases}
\end{align} 
We also have the estimate
\begin{equation}\label{dibdform}
    \|\al\|_{W^{1,p}_0(U)} \le C(\La,U,p,\om) \|\beta\|_{W^{-1,p}(U)}.
\end{equation}
\end{Th}
\begin{proof}
For $\al_1,\al_2\in C^\nf_c\big(U,\bwe^\ell \R^n\big)$, set 
\begin{align*}
    & (\al_1,\al_2)_g\coloneqq\int_U \lan \al_1 ,\al_2 \ran_g\,\dvol_g,\\
    &\mathcal B_g(\alpha_1,\al_2)\coloneqq\int_U\big(\,\langle d\mko \al_1,d\mko \al_2\rangle_g+\langle d^{*_{g}}\al_1,d^{*_g}\al_2\rangle_g\big)\,\dvol_g.
\end{align*}
Using integration by parts as in~\eqref{d^*adjoi}, we have 
\begin{align}\label{delaa=Baa}
\begin{aligned}
    ( \lap_g\al_1,\al_2)_g&=\int_U \big\lan -(dd^{*_g}+d^{*_g}d)\al_1,\al_2\big\ran_g\,\dvol_g\\
    &=-\int_U \big(\,\langle d\mko \al_1,d\mko \al_2\rangle_g+\langle d^{*_{g}}\al_1,d^{*_g}\al_2\rangle_g\big)\,\dvol_g\\
    &=-\mathcal B_g(\alpha_1,\al_2).
\end{aligned}
\end{align}
Let $\al_1=\sum_I\al_{1,I}\mkt dx^I,\al_2=\sum_J \al_{2,J}\mkt dx^J\in C^\nf_c\big(U,\bwe^\ell \R^n\big)$. Denote $p'=\frac p{p-1}$. By~\eqref{prolorgend}, we have 
     \begin{align}\label{al1ap2inn}
     \begin{aligned}
          (\al_1,\al_2)_g&=\int_U\sum_{I,J} \al_{1,I}\mkt\al_{2,J}\,\lan dx^I,dx^J \ran_g\,\dvol_g\\
          &\le \sum_{I,J}\big\|\al_{1,I}\big\|_{W^{-1,p}(U)} \,\big\|\al_{2,J} \mkt \lan dx^I,dx^J \ran_g\mkt (\det g)^{\frac 12}\big\|_{W^{1,p'}_0(U)} \\
          &\le C(\La,U,p,\om) \|\al_1 \|_{W^{-1,p}(U)}\mko \|\al_2\|_{W^{1,p'}_0(U)}.
          \end{aligned}
     \end{align}
It follows that $(\cdot,\cdot)_g$ extends uniquely to a continuous bilinear map from $W^{-1,p}\big(U,\bwe^\ell \R^n\big)\times W^{1,p'}_0\big(U,\bwe^\ell \R^n\big)$ to $\R$. By~\eqref{*estW1pnd} and H\"older's inequality, we also obtain that $\mathcal B_g$ extends uniquely to a continuous bilinear map from $W_0^{1,p}\big(U,\bwe^\ell \R^n\big)\times W^{1,p'}_0\big(U,\bwe^\ell \R^n\big)$ to $\R$. We now establish the existence and uniqueness for the Dirichlet problem~\eqref{eq-dirpro}.

\textbf{(Case I)} $p=2$. By Proposition~\ref{prop-dform}, the bilinear map $\mathcal B_g$ is continuous and coercive on the Hilbert space $W^{1,2}_0\big(U,\bwe^\ell \R^n\big)$: for all $\vp\in W^{1,2}_0\big(U,\bwe^\ell \R^n\big)$, we have
\begin{align}\label{comcaBg}\begin{aligned}
   \|\vp\|_{W^{1,2}_0(U)}^{2}&\le C(\La,U,\om) \big(\|d\mko \vp\|_{L^2(U)}^{2}+\|d^{*_{g}}\vp\|_{L^2(U)}^2\big)\\
   &\le C(\La,U,\om)\mko\mathcal B_g(\vp,\vp).
   \end{aligned}
\end{align}
Given $\beta\in W^{-1,2}\big(U,\bwe^\ell \R^n\big)$, the inequality~\eqref{al1ap2inn} implies that the linear functional $\vp\mapsto -(\beta,\varphi)_g$ is bounded on $W^{1,2}_0\big(U,\bwe^\ell \R^n\big)$.
Then by the Lax–Milgram theorem (see e.g. \cite[Sec.~6.2.1]{evans}), there exists a unique $\alpha\in W^{1,2}_0\big(U,\bwe^\ell \R^n\big)$ such that
\begin{align}\label{mcaalvp=be}
\mathcal B_g(\alpha,\varphi)=-(\beta,\varphi)_g,\qquad\text{for all } \vp\in W^{1,2}_0\big(U,\bwe^\ell \R^n\big).
\end{align}
By~\eqref{delaa=Baa}, this is exactly $\Delta_g\mko \alpha=\beta$ in $W^{-1,2}\big(U,\bwe^\ell\R^n\big)$. Moreover, combining~\eqref{al1ap2inn}--\eqref{mcaalvp=be}, we obtain 
\[
    \|\al\|_{W^{1,2}_0(U)}^2\le C(\La,U,\om) \mko\mathcal B_g(\alpha,\al)\le C(\La,U,\om) \|\beta\|_{W^{-1,2}(U)} \|\al\|_{W^{1,2}_0(U)}.
\]
The estimate~\eqref{dibdform} is established for $p=2$.

\textbf{(Case II)} $p\in(2,n)$. Let $\beta\in W^{-1,p}\big(U,\bwe^\ell \R^n\big)\subset W^{-1,2}\big(U,\bwe^\ell \R^n\big)$. By Case I, there exists a unique $\al\in W^{1,2}_0\big(U,\bwe^\ell \R^n\big)$ solving $\lap_g\mko \al=\beta$ in $U$. Moreover, we have 
\begin{align}\label{alW12betaW-1p}
     \|\al\|_{W^{1,2}_0(U)}\le C(\La,U,\om) \|\beta\|_{W^{-1,2}(U)}\le C(\La,U,\om) \|\beta\|_{W^{-1,p}(U)}.
\end{align}
Applying Lemma~\ref{lm-apriosys} to the operator $\lap_g$ and using~\eqref{alW12betaW-1p}, we obtain $\al\in W^{1,p}_0\big(U,\bwe^\ell \R^n\big)$ with
\begin{align}\label{dibform>2}
\begin{aligned}
    \|\al\|_{W^{1,p}_0(U)}&\le C(\La,U,p,\om)\big(\|\al\|_{L^1(U)}+\|\Delta_{g} \mko \al\|_{W^{-1,p}(U)}\big)\\
    &\le C(\La,U,p,\om) \big(\|\al\|_{W^{1,2}_0(U)}+\|\beta\|_{W^{-1,p}(U)} \big)\\
    &\le C(\La,U,p,\om)\|\beta\|_{W^{-1,p}(U)}.
    \end{aligned}
\end{align}

\textbf{(Case III)} $p\in (\frac n{n-1},2)$. Let $\beta\in C^\nf\big(\overline{U},\bwe^\ell \R^n\big)$, and take $\al$ to be the unique solution in $W^{1,2}_0\big(U,\bwe^\ell \R^n\big)$ of $\lap_g\mko \al=\beta$ in $U$. Once~\eqref{dibdform} is proved, existence in the general case follows by density.

We argue by duality. Let $\tau\in C^{\nf}\big(\overline{U}, \bwe^{\ell+1}\R^n\big)$. Since $d^{*_g}\tau$ is an $\ell$-form in $W^{-1,p'}(U)$, Case II yields a unique $\vp\in W^{1,p'}_0\big(U,\bwe^\ell \R^n\big)$ solving $\lap_g\mko \vp=d^{*_g}\tau$ in $U$. Moreover, by~\eqref{*estlornd} and~\eqref{dibform>2}, we have
\begin{align}
\label{vpestW1p'}\begin{aligned}
    \|\vp\|_{W^{1,p'}_0(U)}&\le C(\La,U,p,\om) \|d^{*_g}\tau\|_{W^{-1,p'}(U)}\\
    &\le C(\La,U,p,\om) \|\tau\|_{L^{p'}(U)}.
    \end{aligned}
\end{align}
Using integration by parts as in~\eqref{d^*adjoi}, we obtain
\begin{align}\label{daltau=bevp}\begin{aligned}
    &(d\al,\tau)_g\\&=(\al, d^{*_g}\tau)_g =(\al , \lap_g \vp)_g
    =-\mathcal B_g(\alpha,\vp)=(\lap_g \al , \vp)_g= (\beta,\vp)_g.
    \end{aligned}
\end{align}
Estimating $(\beta,\vp)_g$ via~\eqref{al1ap2inn} and~\eqref{vpestW1p'} then gives
\begin{align*}
     |(d\al,\tau)_g|
     &\le |(\beta,\vp)_g|\\
    &\le C(\La,U,p,\om) \|\beta\|_{W^{-1,p}(U)} \|\vp\|_{W^{1,p'}_0(U)}\\
     &\le C(\La,U,p,\om)\|\beta\|_{W^{-1,p}(U)}  \|\tau\|_{L^{p'}(U)}.
\end{align*}
Since $\tau$ is arbitrarily chosen in $C^{\nf}\big(\overline{U}, \bwe^{\ell+1}\R^n\big)$, it follows that
\begin{align}\label{dalL^pest}
    \|d\al\|_{L^p(U)}\le  C(\La,U,p,\om)\|\beta\|_{W^{-1,p}(U)}.
\end{align}
A similar argument yields \begin{equation}\label{d*alL^pest}
    \|d^{*_g}\al\|_{L^p(U)}\le  C(\La,U,p,\om)\|\beta\|_{W^{-1,p}(U)}.
\end{equation} 
The estimate~\eqref{dibdform} then follows from combining~\eqref{dalL^pest}--\eqref{d*alL^pest} with Proposition~\ref{prop-dform}. For uniqueness, suppose $\al\in W^{1,p}_{0}\big(U,\bwe^\ell \R^n\big)$ satisfy $\lap_g \mko \al=0$ in $\mca D'\big(U,\bwe^\ell \R^n\big)$. Then for all $\tau\in C^{\nf}\big(\overline{U}, \bwe^{\ell+1}\R^n\big)$, the equation~\eqref{daltau=bevp} still holds and implies $(d\al,\tau)_g=0$. Hence $d\al=0$. Similarly, we have $d^{*_g}\al=0$. Proposition~\ref{prop-dform} then implies $\al=0$. The proof is complete.
\end{proof}
\begin{Prop}\label{prop-W2pdir}
    Let $n,p,U,g$ be as in Theorem~\ref{th:dirdelsol}. Assume $\beta\in L^p\big(U, \bwe^\ell \R^n\big)$ (hence in $W^{-1,p}\big(U,\bwe^\ell \R^n\big)$), and let $\al$ be the unique solution of~\eqref{eq-dirpro}. Then for any subdomain $U'\Subset U$, we have $d\mko \al, d^{*_g}\al\in W^{1,p}(U')$, with the estimate
\begin{align}\label{estdald*alW1p}
    \|d\mko \al\|_{W^{1,p}(U')}+\|d^{*_g} \al\|_{W^{1,p}(U')}\le C(\La,U,U',p,\om)\mko \|\beta\|_{L^p(U)}.
\end{align}
\end{Prop}
\begin{proof}
 We begin with a smooth approximation of the metric. 
 There exist a function $\widetilde \omega\colon (0,\infty)\to [0,\infty)$ depending only on $U$, $\omega$, $\La$ with $\lim_{r\to 0} \widetilde \omega(r)=0$, and a sequence $\{g_k\}_{k=1}^\nf \subset C^\nf(\overline{U},\R^{n\times n}_{\sym})$ such that
  \begin{align}\label{w1nassgk}
    \begin{dcases}
    g_k\to g \qquad &\text{in }W^{1,n}(U,\R^{n\times n}_{\sym}),\\
     \Lambda^{-1}\mkt  |\xi|^{2}\le g_{k,ij}(x)\mkern2mu \xi^{i}\xi^{j}\le \Lambda\mkt  |\xi|^{2},
        \qquad&\text{for all } x\in U\text{ and } \xi\in\R^{n}, \\
         \|\g g_k\|_{L^n(B_r\cap U)}\le \wti\om(r), &\text{for every ball }\mkt B_r\subset \R^n.
    \end{dcases}
\end{align}
Next, we choose $\{\beta_k\}_{k=1}^\nf\subset C^\nf \big(\overline{U},\bwe^\ell \R^n\big)$ with $\beta_k\to \beta$ in $L^p(U)$, and let $\al_k\in W_0^{1,p}\big(U,\bwe^\ell \R^n\big)$ solve \begin{align}\label{eq-dirprok}
        \Delta_{g_k}  \alpha_k = \beta_k \qquad \text{in }U.
\end{align} 
By Theorem~\ref{th:dirdelsol} we have \begin{equation}
    \label{dibdformk}
    \|\al_k\|_{W^{1,p}_0(U)} \le C(\La,U,p,\om) \|\beta_k\|_{W^{-1,p}(U)}.
\end{equation}
By classical elliptic regularity (see e.g. \cite[Sec.~6.3.1]{evans}), $\al_k\in C^\nf(U)$. We fix $\zeta\in C^\nf_c (U)$ with $\zeta\equiv 1$ on $U'$. Arguing as in~\eqref{gxu<Lxu} and~\eqref{lap-rep}, we obtain
\begin{align*}
    |\Delta_{g_k} (\zeta \mko\alpha_k)-\zeta \mko\lap_{g_k} \al_k |\le C(\La,U,U')\big(|\g \al_k|+|\g g_k|\mko |\al_k|\big)\qquad \text{in }U.
\end{align*}
Applying Hölder's inequality, Sobolev embeddings, and \eqref{dibdformk} yields
\begin{align}\label{lapzealkLp}\begin{aligned}
    & \|\Delta_{g_k} (\zeta \mko\alpha_k)\|_{L^p(U)}\\
    &\le C(\La,U,U',p) \big(\|\beta_k\|_{L^p(U)}+\|\g \al_k\|_{L^p(U)}+\|\g g_k\|_{L^n(U)}\|\al_k\|_{L^{\frac{np}{n-p}}}\big)\\
    &\le C(\La,U,U',p,\om)\big(\|\beta_k\|_{L^p(U)}+\|\beta_k\|_{W^{-1,p}(U)}\big)\\
    &\le C(\La,U,U',p,\om) \|\beta_k\|_{L^p(U)}.
    \end{aligned}
\end{align}
Now from $d^2=0$ and the definition of $\lap_{g_k}$ we deduce \begin{align}\label{lapdcomm}\begin{aligned}
    \lap_{g_k}d\mko(\zeta \mko\alpha_k)&=-d\mko d^{*_{g_k}}\mko d(\zeta \mko\alpha_k) \\
    &=d(-d^{*_{g_k}} d-d\mko d^{*_{g_k}})(\zeta \mko\alpha_k) =d\mko \lap_{g_k}(\zeta \mko\alpha_k).
    \end{aligned}
\end{align}
Since $d(\zeta\al_k) \in W^{1,p}_0\big(U, \bwe^{\ell+1}\R^n\big)$, applying Theorem~\ref{th:dirdelsol} with~\eqref{lapzealkLp}--\eqref{lapdcomm} gives
\begin{align}\label{dzealkW1pbd}\begin{aligned}
\|d(\zeta\mko \al_k)\|_{W^{1,p}_0(U)}&\le C(\La,U,p,\om) \|d\mko \lap_{g_k}(\zeta \mko\alpha_k)\|_{W^{-1,p}(U)}\\
&\le C(\La,U,p,\om)  \| \lap_{g_k}(\zeta \mko\alpha_k)\|_{L^p(U)}\\
&\le C(\La,U,U',p,\om) \|\beta_k\|_{L^p(U)}.
\end{aligned}
\end{align}
Using~\eqref{*estlornd} together with the commutation $d^{*_{g_k}}\lap_{g_k}=\lap_{g_k}d^{*_{g_k}}$, the same argument yields
\begin{align}\label{d*zealkW1pbd}\begin{aligned}
\|d^{*_{g_k}}(\zeta\mko \al_k)\|_{W^{1,p}_0(U)}
&  \le C(\La,U,p,\om) \|d^{*_{g_k}} \lap_{g_k}(\zeta \mko\alpha_k)\|_{W^{-1,p}(U)}\\
&\le C(\La,U,p,\om)  \| \lap_{g_k}(\zeta \mko\alpha_k)\|_{L^p(U)}\\
&\le C(\La,U,U',p,\om) \|\beta_k\|_{L^p(U)}.
\end{aligned}
\end{align}
By \eqref{dibdformk}, the sequence $\{\al_k\}$ is bounded in $W^{1,p}_0\big(U,\bwe^\ell \R^n\big)$.
Hence there exist a subsequence (still denoted by $\{\al_k \}$) and $\wti \al\in W^{1,p}_0\big(U,\bwe^\ell \R^n\big)$ such that $\al_k\rightharpoonup \wti\al$ in $W^{1,p}_0\big(U,\bwe^\ell \R^n\big)$.
Using \eqref{w1nassgk}, \eqref{eq-dirprok}, and $\beta_k\to \beta$ in $L^p\big(U,\bwe^\ell \R^n\big)$, we pass to the limit and obtain $\lap_g \mko\wti \al =\beta$ in $\mca D'\big(U,\bwe^\ell \R^n\big)$. By uniqueness of solution for~\eqref{eq-dirpro}, $\al=\wti \al$. 
Combining~\eqref{dzealkW1pbd}--\eqref{d*zealkW1pbd} with the distributional convergences $d(\zeta\mko\al_k)\to d(\zeta\mko\al)$ and $d^{*_{g_k}}(\zeta\mko\al_k)\to d^{*_{g}}(\zeta\mko\al)$, we finally conclude $d\al,d^{*_g}\al\in W^{1,p}_{\loc}(U)$ with the estimate~\eqref{estdald*alW1p}.
\end{proof}
\begin{Dfi}\label{dfi-star-like}
    Let $U\subset \R^n$ be an open set. We call $U$ \textit{starlike} if there exists a non-empty ball $B$ such that for all $x\in U$, the convex hull of $\{x\}\cup B$ is contained in $U$.
\end{Dfi}
\begin{Co}\label{co-Hod-decw-1p}
Let $n\ge 3$, $p\in(\frac{n}{n-1},n)$, $q\in [1,\nf]$, and let $U\subset \R^n$ be a bounded starlike domain with $C^1$ boundary. Suppose $g=(g_{ij})\in L^\infty\cap W^{1,n}(U,\R^{n\times n}_{\sym})$ satisfies \eqref{w1nassg}. Let $\ga\in W^{-2,(p,q)}\big(U, \bwe^{\ell+1}{\R^n}\big)$ with $0\le \ell \le n$. Assume $d \ga=0$ in $\mca D'\big(U,\bwe^{\ell+2}\R^n\big)$. Then there exists $\si\in W^{-1,(p,q)}\big(U, \bwe^{\ell}\R^n\big)$ such that\footnote{Here we write $d*_g\si=0$ instead of $d^{*_g}\si=0$ to avoid applying $*_g$ to $W^{-2,(p,q)}$-forms, since $*_g$ has no canonical continuous extension to $W^{-2,(p,q)}$ for uniformly elliptic $g\in W^{1,n}$. See Remark~\ref{rm*W-22nf}.} 
\[
            d\mko \si=\ga\qquad \text{and}\qquad              d*_g\si=0,
         \qquad \text{in }U,
\]
with the estimate 
\[
    \|\si\|_{W^{-1,(p,q)}(U)}\le C(\La,U,p,\om)\mko \|\ga\|_{W^{-2,(p,q)}(U)}.
\]
\end{Co}
\begin{proof}
   By the weak Poincar\'e lemma on starlike domains \cite[Cor.~3.4]{Costa10} and the Sobolev--Lorentz interpolation (see Lemma~\ref{lm-interpo}), there exists $\beta\in W^{-1,(p,q)}\big(U, \bwe^{\ell}{\R^n}\big)$ satisfying $d\beta=\ga$, with the estimate
    \begin{align*}
         \|\beta\|_{W^{-1,(p,q)}(U)}\le C(U,p)\mko \|\ga\|_{W^{-2,(p,q)}(U)}.
    \end{align*}
    Then applying Theorem~\ref{th:dirdelsol} and Lemma~\ref{lm-interpo}, we obtain a unique $\al\in W_0^{1,(p,q)}\big(U,\bwe^\ell \R^n\big)$ solving $\Delta_g \mko \alpha = \beta$ in $U$, with   
\begin{equation}\label{dibdformdec}\begin{aligned}
    \|\al\|_{W^{1,(p,q)}_0(U)} &\le C(\La,U,p,\om) \mko\|\beta\|_{W^{-1,(p,q)}(U)} \\
    &\le C(\La,U,p,\om)\mko\|\ga\|_{W^{-2,(p,q)}(U)}.
    \end{aligned}
\end{equation}
Now we set $\si\coloneqq-d^{*_g} d \mko\al$. Using $d*_g d^{*_g}=0$ and $d^2=0$, we obtain $d*_g \si=0$ and
\begin{align*}
    d\si=-d\mko d^{*_g} d \mko\al=d \mko\lap_g\mko \al=d\beta=\ga.
\end{align*}
Finally, combining~\eqref{*estlornd} and~\eqref{dibdformdec} with Lemma~\ref{lm-interpo} yields
\begin{align*}
      \|\si\|_{W^{-1,(p,q)}(U)}
      &\le C(\La,U,p,\om)\|d*_g d\al \|_{W^{-1,(p,q)}(U)}\\
      &\le C(\La,U,p,\om)\|d\al \|_{L^{p,q}(U)}\\
      &\le C(\La,U,p,\om)\mko \|\ga\|_{W^{-2,(p,q)}(U)}.
\end{align*}
This completes the proof.
\end{proof}
The same technique, combined with Proposition \ref{prop-W2pdir}, implies the following Corollary.
\begin{Co}\label{Co:Hod-declp}
   Let $n$, $p$, $q$, $U$, $g$ be as in Corollary \ref{co-Hod-decw-1p}. Let $\ga_1\in W^{-1,(p,q)}\big(U, \bwe^{\ell+1}{\R^n}\big)$, $\ga_2\in W^{-1,(p,q)}\big(U, \bwe^{\ell-1}{\R^n}\big)$ with $0\le \ell \le n$. Assume $d\mko \ga_1=0$ and $d*_g\ga_2=0$ in $\mca D'(U)$. Then there exists $\si\in L_{\loc}^{p,q}\big(U, \bwe^{\ell}{\R^n}\big)$ such that 
\[
         d\mko \si=\ga_1\qquad \text{and}\qquad 
             d^{*_g}\si=\ga_2,
         \qquad \text{in }U.
\]
Moreover, for any subdomain $U'\Subset U$, we have the estimate
\begin{align}\label{estsiLpq}
    \|\si\|_{L^{p,q}(U')}\le C(\La,U,U', p,\om)\big(\|\ga_1\|_{W^{-1,(p,q)}(U)}+\|\ga_2\|_{W^{-1,(p,q)}(U)}\big).
\end{align}
\end{Co}
\begin{proof}
By \cite[Cor.~3.4]{Costa10} and Lemma~\ref{lm-interpo}, we find \(\beta_1\in L^{p,q}\big(U,\bwe^\ell\R^n\big)\) with $d\mko \beta_1=\ga_1 $ and
\begin{align*}
\|\beta_1\|_{L^{p,q}(U)}\le C(U,p)\,\|\ga_1\|_{W^{-1,(p,q)}(U)}.
\end{align*}
Applying the previous argument to \(*_g\mko\ga_2\) and using~\eqref{*estlornd}, we obtain \(\beta_2\in L^{p,q}\big(U,\bwe^\ell\R^n\big)\) such that $(-1)^\ell d*_g\beta_2=*_g \,\ga_2$ (i.e. $d^{*_g}\beta_2=\ga_2$), and
\begin{align*}
    \|\beta_2\|_{L^{p,q}(U)}&\le C(\La)\mko \|*_g \beta_2\|_{L^{p,q}(U)}\\
    &\le C(\La,U,p)\mko\|*_g\ga_2\|_{W^{-1,(p,q)}(U)}\\
    &\le C(\La,U,p,\om)\mko\|\ga_2\|_{W^{-1,(p,q)}(U)}.
\end{align*}
By Theorem~\ref{th:dirdelsol}, there exists a unique \(\al_i\in W^{1,(p,q)}_0\big(U,\bwe^\ell\R^n\big)\) solving $\Delta_g\mko \al_i=\beta_i $ in $U$ ($i=1,2$).
Applying Proposition~\ref{prop-W2pdir} and Lemma~\ref{lm-interpo} to \(\al_1\) and \(\al_2\) yields, for every \(U'\Subset U\),
\begin{align}\label{eq:locW1pq-d-d*}\begin{aligned}
& \|d\mko \al_1\|_{W^{1,(p,q)}(U')}+\|d^{*_g}\al_2\|_{W^{1,(p,q)}(U')} \\
&\le C(\La,U,U',p,\om)\big(\|\beta_1\|_{L^{p,q}(U)}+\|\beta_2\|_{L^{p,q}(U)}\big)\\
&\le C(\La,U,U',p,\om)\big(\|\ga_1\|_{W^{-1,(p,q)}(U)}+\|\ga_2\|_{W^{-1,(p,q)}(U)}\big).
\end{aligned}
\end{align}
Now we define $\sigma\coloneqq-\mko d^{*_g} d\mko \al_1 - d\mko d^{*_g}\al_2$.
Using \(d^2=0\) and \((d^{*_g})^2=0\), we have in \(W^{-1,p}_{\loc}\big(U, \bwe^{\ell+1}\R^n\big)\),
\begin{align*}
d\mko \sigma
=-\,d\mko d^{*_g} d\mko \al_1
=d (\Delta_g\mko \al_1)
=d\mko \beta_1
=\ga_1.
\end{align*}
Similarly, we obtain in \(W^{-1,p}_{\loc}\big(U, \bwe^{\ell-1}\R^n\big)\) that
\begin{align*}
d^{*_g} \sigma
=-\mko d^{*_g} d\mko d^{*_g}\al_2
=d^{*_g}(\Delta_g\mko \al_2)
=d^{*_g}\beta_2
=\ga_2.
\end{align*}
Finally, by~\eqref{*estW1pnd} and \eqref{eq:locW1pq-d-d*}, we estimate
\begin{align*}
\|\sigma\|_{L^{p,q}(U')}
&\le \|d^{*_g} d\mko \al_1\|_{L^{p,q}(U')}
+\|d\mko d^{*_g}\al_2\|_{L^{p,q}(U')} \\
&\le C(\La,U,U',p,\om)\big(\|d\mko \al_1\|_{W^{1,(p,q)}(U')}
+\|d^{*_g}\al_2\|_{W^{1,(p,q)}(U')}\big) \\
&\le C(\La,U,U',p,\om)\big(\|\ga_1\|_{W^{-1,(p,q)}(U)}+\|\ga_2\|_{W^{-1,(p,q)}(U)}\big).\qedhere
\end{align*}
\end{proof}
\begin{Rm}
    By the same argument, each result in this section remains valid if we replace the $C^1$-domain assumption by Lipschitz continuity of the domain with small local Lipschitz constant. See for instance \cite[Def.~1.2]{Byun05} and \cite[Def.~1.4]{Byun08}.
\end{Rm}
\begin{Rm}\label{rm:smW1ncoe}
    For any sequence $\{g_k\}_{k\in \N}\subset W^{1,n}(U,\R^{n\times n}_{\sym})$ satisfying the strong convergence $\lim_{k\to \nf} \|\g g_k\|_{L^{n}(U)} =0$, we have 
\[
    \sup_{\substack{k\in \N\\ B_r\subset \R^n}} \|\g g_k\|_{L^{n}(B_r\cap U)}\xrightarrow[r\to 0]{}0.
\]
Thus there exists $\vae=\vae(\La,U,p)>0$ such that, if the second assumption in~\eqref{w1nassg} is replaced by $\|\g g\|_{L^n(U)}\le \vae$, then the constant in \eqref{apriori-form} may be chosen as $C(\La,U,p)$ (independent of $\om$). The same independence from $\om$ holds for Corollaries~\ref{co-Hod-decw-1p}--\ref{Co:Hod-declp} under the assumption $\|\g g\|_{L^n(U)}\le \vae$.
\end{Rm}
\section{Some structural identities}\label{sec:strid}
 Let $\vec{\Phi}\in \mathcal{I}_{1,2}(B^4, \R^m)$ and $g=\bP^*g_{\std}\in W^{2,2}\cap L^\nf (B^4,\R^{4\times 4}_{\sym})$. We fix $\La\ge 1$ such that, for a.e. $x\in B^4$ and all $v\in T_x B^4$, it holds that
\begin{align}\label{eq:immcon4d}
    \Lambda^{-1} \mko |v|^2_{\R^4} \le | d\vec{\Phi}_x(v)|^2_{\R^m}  \le \Lambda  \mko |v|^2_{\R^4}.
\end{align} 
Throughout this section, we adopt the convention~\eqref{con-formsob}. 

By Gram-Schmidt, we find $e_1,\dots,e_4\in L^\infty\cap W^{1,4}(B^4,TB^4)$ such that for all $1\le i,j\le 4$, the following hold a.e. in $B^4$:
\begin{equation}\label{gra-sch-4d}
\langle e_i,e_j\rangle_g=\delta_{ij},\qquad
|e_i|\le C(\La),\qquad
 |\g e_i|\le C(\La)|\g g|.
\end{equation}
Similar to~\eqref{prornd}, there exists a universal constant $C>0$ such that for all $a\in L^\nf \cap W^{1,4}(B^4)$ and $T\in L^1+W^{-1,4/3}(B^4)$, it holds that
\begin{align}\label{prorL1W-143}
    \|aT\|_{L^1+W^{-1,\frac 43}(B^4)}\le C\mko\|a\|_{L^\nf \cap W^{1,4}(B^4)} \mko \|T\|_{L^1+W^{-1,\frac 43}(B^4)}.
\end{align} 
It follows that
\begin{align}\label{*gL1+W-143}\begin{aligned}
    &\| *_{g}\|_{L^1+W^{-1,\frac 43} \left( B^4,\bigwedge^\ell \R^4 \right)\to L^1+W^{-1,\frac 43} \left( B^4,\bigwedge^{4-\ell} \R^4 \right)}\le C(\La)\big(1+\|\g g\|_{L^4(B^4)}\big).
    \end{aligned}
\end{align}
Let $\{e^{i}\}_{i=1}^4\subset L^\infty\cap W^{1,4}(B^4,T^*B^4)$ be the dual coframe to $\{e_i\}_{i=1}^4$. 
For $\vec L\in W^{-1,4/3}\big(B^4, \bwe \R^m\ot \bwe^\ell T^* B^4\big)$ with $0\le \ell \le 4$, set 
\begin{equation}
\label{forcomnot}
    \vec L_{i_1\dots i_\ell}\coloneqq\vec L(e_{i_1},\dots,e_{i_\ell}).
\end{equation}
Then by \eqref{gra-sch-4d} and repeated use of \eqref{prorL1W-143}, it holds that 
\begin{align*}
    \| \vec L_{i_1\dots i_\ell}\|_{L^1+W^{-1,\frac 43}(B^4)}
    &\le \|\vec L\|_{L^1+W^{-1,\frac 43}(B^4)} \prod_{k=1}^\ell\| e_{i_k}\|_{L^\nf \cap W^{1,4}(B^4)}\\
    & \le C(\La)  \|\vec L\|_{L^1+W^{-1,\frac 43}(B^4)} \big(1+\|\g g\|_{L^4}^\ell \big).
\end{align*}
In addition, we have 
\[
     \vec L=\frac{1}{\ell! }\sum_{1\le i_1,\dots,i_\ell\le 4}\mko \vec L_{i_1\dots i_\ell}\ot (e^{i_1}\we \cdots \we e^{i_\ell}).
\]
For scalar-valued differential forms $A, \,\al,\dots$, we write $A_{i_1\dots i_\ell},\,\al_{i_1\dots i_{\ell}},\dots$ analogously.\footnote{The notation $\alpha_{i_1\cdots i_\ell}$ used here is to be distinguished from $\al_I$ (with $I$ a strictly increasing multi-index) in Section \ref{sec:hod-dec}, where we expand in the coordinate coframe $(dx^i)$. 
In what follows, $|\g \al|$ and Sobolev norms are still taken with respect to the $(dx^i)$-coframe.}
Set $\vec e_i\coloneqq d\bP(e_i)$, and we define
\begin{equation}\label{eq:defeta}
    \vec \eta\coloneqq\frac 12\mkt d\bP\overset{\ovwe} \we d\bP.
\end{equation} 
Then under the above notation, we have $(d\bP)_i=\vec e_i$, and $\vec\eta_{ij}=\vec e_i\we \vec e_j$.

 In what follows, the bilinear operators $\,\resg$ and $\bulg$ on $\bwe T^* B^4$ (as introduced in Section~\ref{sec:uselem}) are defined fibrewise with respect to the metric $g_x(dx^i,dx^j)=g^{ij}(x)$ on $T_x^*B^4\simeq \R^4$ for a.e. $x\in B^4$.
Then we follow Notation \eqref{not-overset}: for a.e. $x\in B^4$ the maps $\dres$, $\wres$, $\ovs{\sbul}{\res}_g$, $\ovs{\sbul}{\sbul}_g$, etc. are bilinear on $(\bwe\R^m\otimes\bwe T^{*}_xB^{4})\times (\bwe\R^m\otimes\bwe T^{*}_xB^{4})$; the upper operators act on the $\bwe \R^m$-factors, and the base operators $\,\resg$, $\bulg$ act on the $\bwe T^*_x B^4$-factors.

For completeness, we record the structural identities of \cite[Secs.~II.4--II.5]{Bernard25} in
our notation, providing a self-contained derivation.
\begin{Lm}[{\cite[Prop.~II.2]{Bernard25}}]
\label{ber-prop:II.2}
For $\vec L\in L^1+W^{-1,4/3}\big(B^4, \R^m \ot \bwe^2 \R^4\big)$, we define
\begin{equation*}
\begin{dcases}
        A \coloneqq \vec L\, \dres d\bP \in \mca D'(B^4,\R^4), \\
    B \coloneqq 2\mko \vec L \dwe d\bP \in \mca D'\big(B^4,\bwe^3 \R^4\big),\\
    \vec C \coloneqq \vec L \wres d\vec\Phi \in \mca D'\big(B^4, \bwe^2\R^m\ot \R^4\big),\\
    \vec D \coloneqq 2\mko\vec L \ovs{\ovwe}\we d\vec\Phi \in \mca D'\big(B^4, \bwe^2\R^m\ot \bwe^3\R^4\big).
    \end{dcases}
\end{equation*}
Then the following identity holds in $\mca D'(B^4)$:
\begin{align}
    -3\mko \vec C =\vec\eta \,\ovs{\sbul}{\res}_g \vec C+ \vec D\, \ovs{\sbul}{\res}_g\vec\eta+\vec\eta\,\resg A- B\mko \resg\vet.\label{eq:berIII.1.a}
    \end{align}
\end{Lm}
\begin{proof} 
Since $g, \bP, \vet\in L^\nf \cap W^{1,4}$, by \eqref{prorL1W-143} we have $A,B,\vec C,\vec D\in L^1+W^{-1,4/3}$, and the same inclusion holds for $\vec\eta \,\ovs{\sbul}{\res}_g \vec C,\, \vec D\,\dres\vec\eta, \,\vet\,\resg A$, etc. Hence by approximation, it suffices to prove~\eqref{eq:berIII.1.a} 
pointwise when $\vec L$ is smooth. In this case $\vec L_{ij}\in L^\nf_{\loc}(B^4)$ by \eqref{forcomnot}, since $e_i\in L^\nf(B^4,TB^4)$.
We first compute\footnote{Throughout this section, we write $$
    \sum_j=\sum_{j=1}^4\quad\text{and}\quad \sum_{j,k}=\sum_{j,k=1}^4.$$}
\begin{align}
&A_i=\sum_{j} \vec L_{ji}\cdot \vec e_j, \qquad  B_{ijk}=2\mko \big(\vec L_{ij}\cdot \vec e_k-\vec L_{ik}\cdot \vec e_j+\vec L_{jk}\cdot \vec e_i\big),\label{repAB}\\
&\vec C_i=\sum_{j} \vec L_{ji}\we \vec e_j,\qquad
\vec D_{ijk}=2\mko \big(\vec L_{ij}\we \vec e_k-\vec L_{ik}\we \vec e_j+\vec L_{jk}\we \vec e_i\big).\label{repCD}
\end{align}
The definition~\eqref{eq:defeta} of $\vet$ and the expression~\eqref{repCD} of $\vec D_{ijk}$ yield that for each $i\in \{1,\dots,4\}$, \begin{align*}
    \big(\vec D \,\ovs{\sbul}{\res}_g\vec\eta\big)_i&=\frac 12 \sum_{j,k} \vec D_{ijk} \sbul \vet_{jk}    \\
    &=\sum_{j,k} \big(\vec L_{ij}\we \vec e_k-\vec L_{ik}\we \vec e_j+\vec L_{jk}\we \vec e_i\big)\sbul (\vec e_j \we \vec e_k)  \\
    &=\sum_{j,k} \Big(2\mko \big(\vec L_{ij}\we \vec e_k\big)\sbul (\vec e_j \we \vec e_k)+\big(\vec L_{jk}\we \vec e_i\big)\sbul (\vec e_j \we \vec e_k)\Big).
\end{align*}
By~\eqref{bul2vecs}, this becomes
\begin{align}\label{expaDbuleta}\begin{aligned}
    & \big(\vec D \,\ovs{\sbul}{\res}_g\vec\eta\big)_i =2\sum_{j,k} \Big( \vec L_{ij}\we \vec e_j-\de_{k}^j\mkt \vec L_{ij}\we \vec e_k+ \big(\vec L_{ij}\cdot \vec e_k\big)\mko \vec e_j \we \vec e_k\Big ) \\
    & +\sum_{j,k} \Big(\big(\vec L_{jk}\cdot \vec e_j\big)\mko \vec e_i \we \vec e_k-\big(\vec L_{jk}\cdot \vec e_k\big)\mko \vec e_i \we \vec e_j- \de_{i}^j\mko \vec L_{jk}\we \vec e_k+ \de_{i}^k\mko \vec L_{jk}\we \vec e_j\Big). 
    \end{aligned}
\end{align}
Interchanging $j$ and $k$, and using~\eqref{repCD}, we obtain
\begin{align}\label{intchajk}
\begin{dcases}
    \sum_{j,k}\big(\vec L_{jk}\cdot \vec e_j\big)\mko \vec e_i \we \vec e_k=\sum_{j,k}-\big(\vec L_{jk}\cdot \vec e_k\big)\mko \vec e_i \we \vec e_j,\\
    \sum_{j,k}\de_{i}^k\mko \vec L_{jk}\we \vec e_j=- \sum_{j,k}\de_{i}^j\mko \vec L_{jk}\we \vec e_k=\sum_{j}\vec L_{ji}\we \vec e_j=\vec C_i.
    \end{dcases}
\end{align}
We also have 
\begin{equation}\label{sumLijweej}
    \sum_{j,k} \big( \vec L_{ij}\we \vec e_j-\de_{k}^j\mkt \vec L_{ij}\we \vec e_k\big)=4 \sum_{j}  \vec L_{ij}\we \vec e_j-\sum_{j}  \vec L_{ij}\we \vec e_j=-3\mko \vec C_i.
\end{equation}
Substituting~\eqref{intchajk}--\eqref{sumLijweej} into~\eqref{expaDbuleta} then yields
\begin{align}\label{Dbulreseta}
     \big(\vec D \,\ovs{\sbul}{\res}_g\vec\eta\big)_i=-4\mko \vec C_i+2\sum_{j,k} \Big( \big(\vec L_{ij}\cdot \vec e_k\big)\mko \vet_{jk} + \big(\vec L_{jk}\cdot \vec e_j\big)\mko \vet_{ik}\Big).
\end{align}
In addition, the expression~\eqref{repAB} of $A_k$ gives 
\begin{align}\label{etaresA}
    (\vet\, \resg A)_i=\sum_k \vet_{ki} \mkt A_k=\sum_{j,k} \big(\vec L_{jk}\cdot \vec e_j\big)\mko \vet_{ki}.
\end{align}
Combining~\eqref{Dbulreseta}--\eqref{etaresA}, then we have
\begin{align}\label{comDbulreseta}
     \big(\vec D \,\ovs{\sbul}{\res}_g\vec\eta\big)_i=-4\mko \vec C_i-2\mko  (\vet\, \resg A)_i+2\sum_{j,k}  \big(\vec L_{ij}\cdot \vec e_k\big)\mko \vet_{jk}.
\end{align}
By~\eqref{repCD} and~\eqref{etaresA} we also compute 
\begin{align}\label{cometabulresC}\begin{aligned}
    &(\vec\eta\, \ovs{\sbul}{\res}_g \vec C)_i\\
    &=\sum_j\vet_{ji}\sbul \vec C_j\\
    &=\sum_{j,k} (\vec e_j\we \vec e_i )\sbul \big( \vec L_{kj}\we \vec e_k\big)\\
    &=\sum_{j,k} \Big(\big(\vec L_{kj}\cdot \vec e_j\big)\mko \vet_{ik}-\big(\vec L_{kj}\cdot \vec e_i\big)\mko \vet_{jk}+\de_{k}^j\mkt \vec L_{kj}\we \vec e_i-\de_{i}^k\mko \vec L_{kj}\we \vec e_j\Big)\\
    &=  (\vet\, \resg A)_i+\vec C_i+\sum_{j,k} \big(\vec L_{jk}\cdot \vec e_i\big)\mko \vet_{jk}.
    \end{aligned}
\end{align}
Moreover, from \eqref{repAB} it follows that
\begin{align}\label{comBreseta}\begin{aligned}
    \big(B\mko \resg \vet\big)_i
&= \frac 12 \sum_{j,k} B_{ijk} \,\vet_{jk}\\
&=\sum_{j,k}\Big( \big(\vec L_{ij}\cdot \vec e_k\big)\mko \vet_{jk}-\big(\vec L_{ik}\cdot \vec e_j\big)\mko \vet_{jk}
+ \big(\vec L_{jk}\cdot \vec e_i\big)\mko \vet_{jk} \Big)\\
&=\sum_{j,k}\Big( 2\mko \big(\vec L_{ij}\cdot \vec e_k\big)\mko \vet_{jk}
+ \big(\vec L_{jk}\cdot \vec e_i\big)\mko \vet_{jk} \Big).
\end{aligned}
\end{align}
Combining \eqref{comDbulreseta}--\eqref{comBreseta}, for each $i\in \{1,\dots, 4\}$ we have 
\begin{align*}
    \big(\vec\eta\, \ovs{\sbul}{\res}_g \vec C+ \vec D \,\ovs{\sbul}{\res}_g\vec\eta+\vec\eta\,\resg A- B\mko \resg\vet\big)_i=-3\mko \vec{C}_i.
\end{align*}
This completes the proof.
\end{proof}    
We define the codifferential $d^{*_g}$ as in~\eqref{defd*glap}. In particular, in dimension $4$ one has, for all differential forms,
\begin{equation}\label{defd*g4d}
    d^{*_g}\coloneqq-*_g d*_g.
\end{equation}
\begin{Lm}[cf.\ {\cite[Prop.~II.3]{Bernard25}}]\label{prop:III2ber}
Suppose $\al\in L^\nf\cap W^{1,4}\big(B^4, \bwe^2\R^4\big)$ and $\beta\in L^{4/3}\big(B^4,\bwe^2 \R^4\big)$. Then the following holds in $\mca D'(B^4,\R^4)$:\footnote{We write $d\beta \resg \al=(d\beta)\resg \al.$}
\begin{align}\label{eq:III2-scalar}
  \al\, \resg d^{*_{g}} \beta + d\beta\, \resg \al
  =d^{*_{g}}(\al \bulg \beta) + d\mko ( \al\,\res_g\beta) +\mca R_1[\al,\beta;g],
\end{align}
where the remainder satisfies
\begin{align}\label{eq:III2-bd}
    \big|\mca R_1 [\al,\beta;g]\big|\le C(\La)\big(|\g \al|\mko|\beta|+|\al| \mko|\beta|\mko|\nabla g|\big)\qquad \text{a.e. in }B^4.
\end{align}
Similarly, let $\vec\al\in L^\infty\cap W^{1,4}\big(B^4,\bwe^2\R^m\ot\bwe^2\R^4\big)$, $\beta\in L^{4/3}\big(B^4,\bwe^2 \R^4\big)$, and 
$\vec\beta\in L^{4/3}\big(B^4,\bwe^2\R^m\ot\bwe^2\R^4\big)$. Then we have\footnote{Our $\bulg$ coincides with $\bul$ in~\cite{Bernard25}, and our codifferential $d^{*_{g}}$ is the negative of $d^{\star}$ or $d^{\star_g}$ in the same reference.}
\begin{align}
  &\vec\al\,\resg d^{*_{g}}\beta + d\beta\,\resg \vec\al
  = d^{*_{g}}\big(\vec\al \bulg \beta\big) + d\big( \vec\al\,\resg \beta\big)
    +  \vec{\mca R}_1 [\vec\al,\beta;g],\label{eq:III2-vector1}\\
       &\vec\al\,\ovs{\sbul}{\res}_g d^{*_{g}}\vec\beta - d\vec\beta\,\ovs{\sbul}{\res}_g \vec\al
  = d^{*_{g}}\big(\vec\al\,\ovs{\sbul}{\sbul}_g \vec\beta\big) + d\big( \vec\al\,\ovs{\sbul}{\res}_g\vec\beta\big)
    +  \vec{\mca R}_2 \mko [\vec\al,\vec\beta;g],\label{eq:III2-vector2}
    \end{align}
with
\begin{equation}\label{eq:III2-bdvec}\begin{dcases}
\big|\mca {\vec R}_1 [\vec \al,\beta;g]\big|\le C(\La)\big(|\g \vec \al|\mko|\beta|+|\vec \al| \mko|\beta|\mko|\nabla g|\big),\\
 \big|\vec{\mca R}_2 \mko[\vec \al,\vec \beta; g]\big|
 \le C(\La)\big(|\nabla \vec\al|\mko |\vec\beta|+|\nabla g|\mko |\vec\al|\mko |\vec\beta|\big),
 \end{dcases}
 \qquad\text{a.e.\ in }B^4.
\end{equation}
\end{Lm}

\begin{proof}

We first prove~\eqref{eq:III2-scalar} with the pointwise bound \eqref{eq:III2-bd}. The argument is by approximation: once \eqref{eq:III2-scalar} and \eqref{eq:III2-bd} are proved for smooth $\al,\beta$, the general case then follows by taking smooth sequences $\al_k\to\al$ in $W^{1,4}$ and weak-$^*$ in $L^\infty$, and $\beta_k\to\beta$ in $L^{4/3}$. We now check all the estimates required for this approximation.
By~\eqref{*gL1+W-143}, we have
\begin{equation}\label{d*gL4/3bd}\begin{aligned}
    \|d^{*_g}\beta\|_{L^1+W^{-1,\frac 43}(B^4)}&=\|-*_g \,d*_g \beta\|_{L^1+W^{-1,\frac 43}(B^4)}\\
    &\le C(\La)\mko\|d*_g \beta\|_{W^{-1,\frac 43}(B^4)}\big(1+\|\g g\|_{L^4(B^4)}\big)\\
    &\le C(\La)\mko \|\beta\|_{L^{\frac 43}(B^4)}\mko\big (1+\|\g g\|_{L^4(B^4)}\big).
    \end{aligned}
\end{equation}
The inequality \eqref{prorL1W-143} then gives \begin{align*}
    &\|\al\, \resg d^{*_{g}} \beta\|_{L^1+W^{-1,\frac 43}(B^4)}\\
    &\le C(\La)\mko\|\al \|_{L^\nf\cap W^{1,4}(B^4)}\mko\|g\|_{L^\nf\cap W^{1,4}(B^4)}\mko\|d^{*_g}\beta\|_{L^1+W^{-1,\frac 43}(B^4)}\\
    &\le C(\La)\mko\|\al \|_{L^\nf\cap W^{1,4}(B^4)}\mko\|\beta\|_{L^{\frac 43}(B^4)}\big(1+\|\g g\|_{L^4(B^4)}^2 \big).
\end{align*} 
Similarly, $d\beta\,\resg \al \in L^1+W^{-1,4/3}(B^4,\R^4) $. In addition, we have \begin{align*}
    \|\al\bulg \beta\|_{L^{\frac 43}(B^4)}+ \|\al\resg \beta\|_{L^{\frac 43}(B^4)}\le C(\La)\mko \|\al\|_{L^\nf(B^4)}\mko\|\beta\|_{L^{\frac 43}(B^4)}.
\end{align*}
Then \eqref{*gL1+W-143} yields $d^{*_{g}}(\al \bulg \beta),\, d\mko ( \al\,\res_g\beta)\in L^1+W^{-1, 4/3}(B^4, \R^4)$. Moreover, the remainder $\mca R_1[\al,\beta;g]$ lies in $L^1(B^4,\R^4)$ provided \eqref{eq:III2-bd} holds. 

Thus by density, we may assume $\al,\,\beta$ are smooth. In this case, the coefficients $\al_{ij},\,\beta_{ij}$ lie in $L^\nf_{\loc} \cap W^{1,4}_{\loc}(B^4)$ by \eqref{forcomnot}. In addition, for $f_1,f_2\in L^\nf\cap W^{1,4}(U)$ with $U\Subset B^4$, we have
\[
    \|f_1 f_2\|_{L^\nf\cap W^{1,4}(U)}\le \|f_1\|_{L^\nf\cap W^{1,4}(U)}\mko \|f_2\|_{L^\nf\cap W^{1,4}(U)}.
\]
Hence $L^\nf_{\loc} \cap W^{1,4}_{\loc}(B^4)$ is closed under pointwise multiplication. It follows that $(\al \bulg \beta)_{ij}$ and $\al\,\res_g\beta$ lie in $L^\nf_{\loc}\cap W^{1,4}_{\loc}(B^4)$, and $\al\, \resg d^{*_{g}} \beta$, $d\beta\,\resg \al$, $d^{*_{g}}(\al \bulg \beta)$, $d\mko ( \al\,\res_g\beta)\in L^4_{\loc}$. It remains to show that \eqref{eq:III2-scalar} holds a.e. with the bound \eqref{eq:III2-bd}.

 The codifferential of $\beta$ is given by
\begin{align}\label{com-d*be}\begin{aligned}
    d^{*_g}\beta&= -*_g d*_g \beta\\
    &=-\frac 12 \sum_{i,j}*_g \,d\big ( \beta_{ij}*_g(e^i\we e^j)\big)\\
    &=-\frac 12 \sum_{i,j}*_g\,\big(d\beta_{ij}\we *_g(e^i\we e^j)\big)+\frac 12\sum_{i,j}\mkt \beta_{ij}\mkt d^{*_g} (e^i\we e^j).
    \end{aligned}
\end{align}
We define 
\begin{equation}\label{defmcaR}
    \mathcal R[\beta;g] \coloneq \frac 12\sum_{i,j}\mkt \beta_{ij}\mkt d^{*_g} (e^i\we e^j).
\end{equation} 
By \eqref{gra-sch-4d}, we have $|\g e^i|\le C(\La)\mko|\g g|$ a.e. in $B^4$. Hence, we have
\begin{align}\label{ptbdRbe}
    |\mca R[\beta;g]|\le C(\La)\mko|\beta|\mko |\g g|\qquad \text{a.e. in }B^4.
\end{align}
For a function $f$ on $B^4$, write $e_i(f)\coloneqq df(e_i)$. Applying the identity~\eqref{*_gcommwed} in \eqref{com-d*be}--\eqref{defmcaR}, we get
\begin{align}\label{d*gbecom}
    d^{*_g}\beta
    &=-\frac 12  \sum_{i,j} (e^i\we e^j)\,\res_g d\beta_{ij}+\mca R[\beta;g]
    = \sum_{i,j} e_j(\beta_{ij}) \mkt e^i+\mca R[\beta;g].
\end{align} 
Arguing as in~\eqref{com-d*be}, we also obtain
\begin{align}\label{dbetaexp}
        \big(d\beta-\mca R'[\beta;g]\big)_{ijk}= e_i(\beta_{jk})-e_j(\beta_{ik})+e_k(\beta_{ij}),
\end{align}
where the remainder satisfies
\begin{align}\label{ptbdR'be}
     |\mca R'[\beta;g]|\le C(\La)\mko|\beta|\mko |\g g|\qquad \text{a.e. in }B^4.
\end{align}
Now we compute that for each $1\le i\le 4$, 
\begin{align*}
    \big(d(\al\res_g\beta)\big)_i&=\frac 12\sum_{j,k} e_i\big(\al_{jk}\mkt \beta_{jk}\big)
    =\frac 12\sum_{j,k}\al_{jk}\mkt e_i(\beta_{jk})+\frac 12\sum_{j,k}\beta_{jk}\mkt e_i(\al_{jk}).
\end{align*}
Then by~\eqref{dbetaexp}, we have
\begin{align}
    \big(\mko(d\beta-\mca R'[\beta;g])\,\res_g\al\big)_i&=\frac 12 \sum_{j,k}\big(d\beta-\mca R'[\beta;g]\big)_{ijk}\, \al_{jk} \nonumber \\
    &=\frac12 \sum_{j,k}\al_{jk}\big(e_i(\beta_{jk})-e_j(\beta_{ik})+e_k(\beta_{ij})\big) \nonumber \\
    &=\big(d(\al\res_g\beta)\big)_i-\sum_{j,k}\al_{jk}\mkt e_j(\beta_{ik})-\frac 12\sum_{j,k}\beta_{jk}\mkt e_i(\al_{jk}). \label{dberesal}
\end{align}
Next, using~\eqref{d*gbecom} we compute
\begin{align}\begin{aligned}\label{alresd*be}
    (\al\,\res_g d^{*_g}\beta)_i
    &=\sum_{k} \al_{ki} (d^{*_g}\beta)_k
    =\sum_{j,k} \al_{ki} \,e_j(\beta_{kj})+\sum_{k} \al_{ki}\mko \mca R[\beta;g]_k.
    \end{aligned}
\end{align}
By~\eqref{bul2vecs}, we have 
\begin{equation}\label{indefbul22}
    (\al \bulg \beta)_{ij}= \sum_{k}(\al_{ik}\mkt \beta_{jk}-\al_{jk}\mkt\beta_{ik}).
\end{equation}
Applying~\eqref{ptbdRbe}--\eqref{d*gbecom} to $\al\bulg \beta$ yields
\begin{align}\label{d*albulbe1}
   \big(d^{*_g}(\al\bulg \beta)\big)_i&=\sum_{j}e_j\big((\al \bulg \beta)_{ij}\big)+\mca R[\al\bulg \beta;g]_i,
\end{align}
with
\begin{equation}\label{ptbdRalbeg}
     |\mca R[\al\bulg \beta;g]|\le C(\La)\mko|\al|\mko|\beta|\mko |\g g|\qquad \text{a.e. in }B^4.
\end{equation}
From \eqref{alresd*be}--\eqref{indefbul22}, we obtain  \begin{align}\label{d*albulbe2}\begin{aligned}
    &\sum_{j}e_j\big((\al \bulg \beta)_{ij}\big)\\
    &=\sum_{j,k}\big(\al_{ik}\mkt e_j(\beta_{jk})-\al_{jk}\mkt e_j(\beta_{ik})\big)+\sum_{j,k} \big(\beta_{jk}\mkt e_j(\al_{ik})-\beta_{ik}\mkt e_j(\al_{jk})\big)\\
    &= (\al\,\res_g d^{*_g}\beta)_i-\sum_{j,k} \al_{jk}\mkt e_j(\beta_{ik})+\sum_{j,k} \big(\beta_{jk}\mkt e_j(\al_{ik})-\beta_{ik}\mkt e_j(\al_{jk})\big)\\
    &\quad-\sum_{k} \al_{ki}\mko \mca R[\beta;g]_k.
    \end{aligned}
\end{align}
We now define $\mca R_1[\al, \beta;g]$ by  
\begin{align}\label{defR1abg}\begin{aligned}
    \mca R_1[\al, \beta;g]_i
    &=\big(\mca R'[\beta;g]\,\res_g\al\big)_i-\frac 12\sum_{j,k}\beta_{jk}\mkt e_i(\al_{jk})-\mca R[\al\bulg \beta;g]_i\\
    &\quad+\sum_k\al_{ki} \mko\mca R[\beta;g]_k -\sum_{j,k} \big(\beta_{jk}\mkt e_j(\al_{ik})-\beta_{ik}\mkt e_j(\al_{jk})\big).
    \end{aligned}
\end{align}
By the definition~\eqref{forcomnot} of $\al_{ik}$, we have \begin{align*}
    |e_j(\al_{jk})|\le C(\La)\big(|\g \al|+ |\al|\mko |\g g|\big)\qquad \text{a.e. in }B^4.
\end{align*} 
Using \eqref{ptbdRbe}, \eqref{ptbdR'be}, and \eqref{ptbdRalbeg}, we then obtain
\begin{align*}
     \big|\mca R_1 [\al,\beta;g]\big|\le C(\La)\big(|\g \al|\mko|\beta|+|\al| \mko|\beta|\mko|\nabla g|\big)\qquad \text{a.e. in }B^4.
\end{align*}
Combining~\eqref{d*albulbe1}, \eqref{d*albulbe2}, and \eqref{defR1abg}, we arrive at \begin{align}\label{befIII2sca}\begin{aligned}
    \big(d^{*_g}(\al\bulg \beta)\big)_i
    &=(\al\,\res_g d^{*_g}\beta)_i-\sum_{j,k} \al_{jk}\mkt e_j(\beta_{ik})-\frac 12\sum_{j,k}\beta_{jk}\mkt e_i(\al_{jk}) \\
    &\quad - \big(\mca R_1[\al, \beta;g]_i-\mca R'[\beta;g]\,\res_g\al\big)_i\mko .
    \end{aligned}
\end{align}
Subtracting \eqref{befIII2sca} from \eqref{dberesal} completes the proof of \eqref{eq:III2-scalar}. 

 Finally, to prove~\eqref{eq:III2-vector1}--\eqref{eq:III2-bdvec}, it suffices to take $\al\in L^\nf\cap W^{1,4}\big(B^4, \bwe^2\R^4\big)$ and $\beta\in L^{4/3}\big(B^4,\bwe^2 \R^4\big)$, and to consider the case $\vec \al= \vec v_1 \ot  \al $, $\vec \beta=\vec v_2 \ot \mathbf \beta$ for fixed $\vec v_1,\vec v_2\in \bwe^2 \R^m$. Then \eqref{eq:III2-scalar} implies that
 \begin{align*}
     \vec\al\,\resg d^{*_{g}}\beta + d\beta\,\resg \vec\al
     &=\vec v_1\ot \big(\al\, \resg d^{*_{g}} \beta + d\beta\, \resg \al\big)\\
     &=\vec v_1\ot\big(d^{*_{g}}(\al \bulg \beta) + d\mko ( \al\,\res_g\beta) +\mca R_1[\al,\beta;g]\big)\\
     &= d^{*_{g}}\big(\vec\al \bulg \beta\big) + d\big( \vec\al\,\resg \beta\big)+\vec v_1 \ot \mca R_1[\al,\beta;g].
 \end{align*}
Using that $\sbul\colon \bwe^2 \R^m \times \bwe^2 \R^m\to \bwe^2\R^m$ is antisymmetric, we also obtain
\begin{align*}
     \vec\al\,\ovs{\sbul}{\res}_g d^{*_{g}}\vec\beta - d\vec\beta\,\ovs{\sbul}{\res}_g \vec\al
     &=(\vec v_1 \sbul \vec v_2)\ot (\al\, \resg d^{*_{g}} \beta)-(\vec v_2 \sbul \vec v_1)\ot(d\beta\, \resg \al)\\
     &=(\vec v_1 \sbul \vec v_2)\ot\big(d^{*_{g}}(\al \bulg \beta) + d\mko ( \al\,\res_g\beta) +\mca R_1[\al,\beta;g]\big)\\
     &=d^{*_{g}}\big(\vec\al\,\ovs{\sbul}{\sbul}_g \vec\beta\big) + d\big( \vec\al\,\ovs{\sbul}{\res}_g\vec\beta\big)+(\vec v_1 \sbul \vec v_2)\ot\mca R_1[\al,\beta;g].
\end{align*}
The proof is complete.
\end{proof}
As in Notation~\eqref{not-met}, we denote by $\bn$ the Gauss map of $\bP$ and by $\pi_{\bn}$ the orthogonal projection onto the normal bundle of $\bP(B^4)\subset \R^m$. Using the operation $\,\res$ on $\bwe \R^m$, for all $\vec v\in \R^m$, we have 
 \begin{equation}\label{pinvres}
    \pi_{\bn}\mko \vec v=(-1)^{m-1}\mko\bn\,\res(\bn\,\res \vec v).
\end{equation}
Since $\bP\in \mathcal{I}_{1,2}(B^4,\R^m)$, we have $\bH\in W^{1,2}(B^4,\R^m)$ and $\bn\in L^\nf \cap W^{2,2}(B^4,\bwe^{m-4}\R^m)$, where $\bH$ is the mean curvature vector of $\bP$. Combining the proofs of~\eqref{prornd}--\eqref{prolorgend} with Lemma~\ref{lm-interpo}, for all $p\in (\frac 43,\nf)$, $q\in [1,\nf]$, $a\in L^\nf \cap W^{1,4}(B^4)$, and $T\in W^{-1,(p,q)}(B^4)$, we obtain 
\begin{align}\label{prolorgen}
     \|aT\|_{W^{-1,(p,q)}(B^4)}\le C(p)\mko\|a\|_{L^\nf \cap W^{1,4}(B^4)} \mko \|T\|_{W^{-1,(p,q)}(B^4)}.
\end{align}
In particular, this implies
\begin{align}\label{*estlor}
    \| *_{g}\|_{W^{-1,(p,q)}\lf(B^4, \bigwedge^\ell \R^4\rg)\to W^{-1,(p,q)}\lf(B^4, \bigwedge^{4-\ell} \R^4\rg)}\le C(\La,p)\big(1+ \|\g g\|_{L^4(B^4)}\big).
\end{align} 
Setting $p=q=2$, we obtain $\lap_g \bH=*_g\, d\,*_g\, d\bH\in W^{-1,2}(B^4)$. Then by~\eqref{pinvres} and~\eqref{prolorgen}, we have $d\bP\we \pi_{\bn}\mko\lap_g \bH\in W^{-1,2}\big(B^4,\bwe^2\R^m \ot \R^4\big)$.
  We now apply Lemmas~\ref{ber-prop:II.2} and \ref{prop:III2ber} to obtain the following result, which will be used in Section~\ref{sec:pfmainThm} to analyze the system~\eqref{sys-LRS}.
\begin{Prop}\label{prop:III3-1rev}
    Suppose $S\in L^{4/3}\big(B^4,\bwe^2 \R^4\big)$, $\vec R\in L^{4/3}\big(B^4,\bwe^2\R^m\ot \bwe^2 \R^4\big)$, $\vec L\in L^1+ W^{-1, 4/3}\big(B^4,\R^m\ot \bwe^2 \R^4\big)$, $\vt_{\dil}\in L^{4/3}(B^4,\R^2)$, and $\bvt_{\rot}\in L^{4/3}\big(B^4,\bwe^2 \R^m\ot \R^4\big)$ satisfy the relations 
\begin{align}\label{sys:d*dSR}
    \begin{dcases}
        d^{*_g} S = \vec{L} \,\dres  d\bP+\vt_{\dil}, \quad &d S = -2\mko \vec L  \dwe d\Phi,\\[0.5ex]
        d^{*_g} \vec R = \vec L \wres d\vec\Phi +\frac 12\mko d\bP\we \pi_{\bn}\mko\lap_g \bH+\bvt_{\rot}, \quad &
    d \vec R = -2\mko \vec L \ovs{\ovwe}\we d\vec\Phi.
    \end{dcases}
\end{align}
Then the following holds in $\mca D'\big(B^4,\bwe^2 \R^m\big)$:
\begin{align}
d^{*_{g}}\big(3\mko \vec R + \vet \,\ovs{\sbul}{\sbul}_g \vec R + \vet \,\bulg S\big)
+ d\big(\vet \,\ovs{\sbul}{\res}_g \vec R + \vet \,\resg S\big)
&=3\mko d\bP\we \pi_{\bn}\mko\lap_g \bH+\vec{\mca R}_0, \label{eq:III3-1}
\end{align}
where the remainder $\vec{\mca R}_0$ satisfies 
\begin{align}\label{R3ptbd}
     |\vec{\mca R}_0|\le C(\La) \big(|\vt_{\dil}|+|\bvt_{\rot}|+ |\g^2\bP|\mko(|\bR |+|S|) \mko\big)\qquad \text{a.e. in }B^4.
\end{align}
\end{Prop}

\begin{proof}
Set
\begin{equation*}
\begin{alignedat}{2}
    &A=d^{*_{g}}S-\vt_{\dil},&\qquad& B=-\mko dS,\qquad \\
    &\vec C= d^{*_{g}}\vec R-\frac 12\mko d\bP\we \pi_{\bn}\mko\lap_g \bH-\bvt_{\rot},&\qquad& \vec D=-\mko d\vec R.
\end{alignedat}
\end{equation*}
Then by~\eqref{prorL1W-143} and~\eqref{d*gL4/3bd}, we have $A,B,\vec C,\vec D\in L^1+W^{-1,4/3}(B^4)$. Lemma~\ref{ber-prop:II.2} implies that in $L^1+W^{-1,4/3}\big(B^4,\bwe^2\R^m\ot \R^4\big)$, there holds
\[
    -3\mko \vec C =\vec\eta \,\ovs{\sbul}{\res}_g \vec C+ \vec D\, \ovs{\sbul}{\res}_g\vec\eta+\vec\eta\,\resg A- B\mko \resg\vet.
\]
Equivalently, it holds that
\begin{align}\label{strideapp}\begin{aligned}
&-3\mko d^{*_{g}}\vec R+\frac 32\mko d\bP\we \pi_{\bn}\mko\lap_g \bH+\frac 12\mko\vet \,\ovs{\sbul}{\res}_g(d\bP\we \pi_{\bn}\mko\lap_g \bH)\\
&= \vet \,\ovs{\sbul}{\res}_g d^{*_{g}}\vec R
  - d\vec R\,\ovs{\sbul}{\res}_g \vet
  + \vet \,\resg d^{*_{g}}S
  + dS \,\res_g\vet
   -3\mko\bvt_{\rot}-\vet \,\ovs{\sbul}{\res}_g \bvt_{\rot}-\vet\,\resg \vt_{\dil}.
  \end{aligned}
\end{align}
Using the notation \eqref{forcomnot}, we compute 
\begin{align}\label{etabulresnwe}
\begin{aligned}
    \big(\vet \,\ovs{\sbul}{\res}_g (d\bP\we \pi_{\bn}\mko\lap_g \bH)\mko\big)_i
    &=\sum_{j=1}^4 \vet_{ji}\sbul ( \vec e_j\we \pi_{\bn}\mko\lap_g \bH)\\ &=\sum_{j=1}^4 ( \vec e_i\we \pi_{\bn}\mko\lap_g \bH-\de^j_i \mko \vec e_j\we  \pi_{\bn}\mko\lap_g \bH) \\
    &=3\mko \vec e_i\we  \pi_{\bn}\mko\lap_g \bH =3(d\bP\we  \pi_{\bn}\mko\lap_g \bH)_i.
    \end{aligned}
\end{align}
By~\eqref{eq:III2-vector1}--\eqref{eq:III2-bdvec}, we have
\begin{align}\label{appber25II3rev}\begin{dcases}
   \vet \,\ovs{\sbul}{\res}_g d^{*_{g}}\vec R
  - d\vec R\,\ovs{\sbul}{\res}_g \vet=  \mko d^{*_{g}}\big(\vet \,\ovs{\sbul}{\sbul}_g \vec R\big)+d\big(\vet\,\ovs{\sbul}{\res}_g \bR\big)+\mca {\vec R}_2 \mko[\vet,\bR;g],\\
  \vet \,\resg d^{*_{g}}S
  + dS \,\res_g\vet =d^{*_g}\big(\vet \,\bulg S\big)
+ d\big( \vet \,\resg S\big)
+ \vec {\mca R}_1[\vet,S;g],
\end{dcases}
\end{align}
with \[
   |\mca {\vec R}_2 \mko[\vet,\bR;g]|+ |\vec {\mca R}_1[\vet,S;g]|\le C(\La)\big(|\g \vec \eta|+|\g g|)\mko (|\bR|+|S|)\le C(\La)\mko |\g^2\bP| \mko(|\bR|+|S|).
\]
Let 
\begin{equation}\label{defvecR3}
    \vec{\mca R}_0\coloneqq 3\mko\bvt_{\rot}+\vet \,\ovs{\sbul}{\res}_g \bvt_{\rot}+\vet\,\resg \vt_{\dil}-\big(\mca {\vec R}_2 \mko[\vet,\bR;g] +\vec {\mca R}_1[\vet,S;g]\big).
\end{equation}
Then $\vec{\mca R}_0$ satisfies the bound~\eqref{R3ptbd}. Combining~\eqref{strideapp}--\eqref{defvecR3} yields~\eqref{eq:III3-1}, which completes the proof.
\end{proof}
\begin{Rm}\label{rm:g^2Pdecom}
    Following the convention~\eqref{con-formsob}, we write $|\bII|^2 \coloneqq\sum_{i,j=1}^4 \big|\pi_{\bn}\mko\p_{x^i}\p_{x^j} \bP\big|^2$.
We decompose the Hessian of $\bP$ using $g_{ij}=\p_{x^i} \bP\cdot\p_{x^j}\bP$ and $(g^{ij})=(g_{ij})^{-1}$ as follows
\begin{align}\label{d2phi=n+t}
   \p_{x^i}\p_{x^j}\bP=\bII_{ij} +\sum_{k=1}^4 g^{k\ell}(\p_{x^i}\p_{x^j} \bP \cdot \p_{x^k} \bP )\mko\p_{x^\ell} \bP.
\end{align}
As for the computation of Christoffel symbols, we have 
\begin{align}\label{chriscom}\begin{aligned}
   & \big| \p_{x^i}\p_{x^j} \bP \cdot \p_{x^k} \bP\big| \\
   &=\frac 12 \mko\big|\p_{x^i} (\p_{x^j}\bP\cdot\p_{x^k}\bP)+\p_{x^j}(\p_{x^i}\bP\cdot \p_{x^k}\bP)-\p_{x^k}(\p_{x^i}\bP\cdot \p_{x^j}\bP) \big|\\
    &\le |\g g|.
    \end{aligned}
\end{align}
It follows that
\begin{align}\label{get<II+gg}
    |\g^ 2\bP|\le C(\La)\mko(|\bII|+|\g g|)\le C(\La)\mko(|\g\bn|+|\g g|).
\end{align}
See for instance~\eqref{pn=II} for the relation between $d\bn$ and $\bII$. Analogously, by differentiating~\eqref{d2phi=n+t}, we obtain that
\begin{align}\label{ptbdg3p}
    |\g^3\bP|\le C(\La)\big( |\g g|^2+|\g \bn|^2+|\g^2 g|+|\g^2 \bn|\big).
\end{align}
\end{Rm}
To deal with the term $3\mko d\bP\we \pi_{\bn}\mko\lap_g \bH$ in~\eqref{eq:III3-1}, we need the following lemma.
\begin{Lm}\label{lm:d*XwedPhi}
    Let $\vec X\in L^{4/3}(B^4,\R^m\ot \R^4)$. Then the following holds in $\mca D'(B^4, \R^4)$:
    \begin{align}\label{d*XwedPhi}
        d^{*_g}\big(\vec X \ovs{\ovwe}{\we} d\bP \big)=(d^{*_g} \vec X) \we d\bP+d\big(\vec X \wres d\bP\big)+d\vec X \wres d\bP+\vec {\mca R}_1[\vec X;g],
    \end{align}
    where the remainder satisfies 
\begin{align}\label{vecR4ptbd}
     |\vec{\mca R}_1[\vec X;g]|\le C(\La)\mko|\g^2\bP|\mko |\vec X|\qquad \text{a.e. in }B^4.
\end{align}
\end{Lm}
\begin{proof}
    Arguing as in Lemma~\ref{prop:III2ber}, we obtain that each term in~\eqref{d*XwedPhi} lies in $W^{-1,4/3}+L^1(B^4,\R^4)$. Then by approximation, it suffices to prove~\eqref{d*XwedPhi}--\eqref{vecR4ptbd} for $\vec X\in C^\nf(B^4,\R^m \ot \R^4)$.
    Using the notation~\eqref{forcomnot} together with~\eqref{ptbdRbe}--\eqref{d*gbecom}, we obtain
    \begin{align}\label{d*XwedPfirex}
    \begin{aligned}
        d^{*_g}\big(\vec X \ovs{\ovwe}{\we} d\bP \big)
        &=\sum_{i,j} e_j\big(\vec X_i \we \vec e_j-\vec X_j \we \vec e_i\big) \, e^i+\vec{\mca R}\big[\vec X \ovs{\ovwe}{\we} d\bP;g\big]\\
        &=\sum_{i,j} \big(  e_j(\vec X_i) \we \vec e_j-e_j(\vec X_j)\we \vec e_i \big)\mko e^i+\vec{\mca R}\big[\vec X \ovs{\ovwe}{\we} d\bP;g\big]\\
        &\quad+\sum_{i,j}\big(\vec X_i\we e_j(\vec e_j)-\vec X_j\we e_j(\vec e_i)\mko\big)\mko e^i,
        \end{aligned}
    \end{align}
    where the remainder $\vec{\mca R}\big[\vec X \ovs{\ovwe}{\we} d\bP;g\big]$ satisfies
    \begin{align}\label{RXwedPh}
        \big|\vec{\mca R}\big[\vec X \ovs{\ovwe}{\we} d\bP;g\big] \big|\le C(\La) \mko |\vec X|\mko|\g g|\qquad \text{a.e. in }B^4.
    \end{align}
   On the other hand, the same proof of~\eqref{d*gbecom}--\eqref{dbetaexp} implies that 
   \begin{align}
      &d^{*_g}\vec X=- \sum_{j}e_j(\vec X_j)+\vec{\mca R}[\vec X;g],\label{expofd*X}\\
      &d\vec X=\frac 12\sum_{i,j} \big(e_i(\vec X_j)-e_j(\vec X_i)\big)\mko e^i\we e^j+\vec{\mca R}'[\vec X;g],\label{expofdX}
   \end{align}
   where the remainders satisfy
   \begin{align}\label{ptbdRR'X}
       |\vec{\mca R}[\vec X;g]|+|\vec{\mca R}'[\vec X;g]|\le C(\La)\mko |\g g|\mko|\vec X|\qquad \text{a.e. in }B^4.
   \end{align}
   Using~\eqref{expofdX} together with the identity $\sum_{j=1}^4 \vec X_j \we \vec e_j=\vec X\wres d\bP$, we obtain
\begin{align}\label{e_jXiweej}\begin{aligned}
       &\sum_{i,j} \big(  e_j(\vec X_i) \we \vec e_j\big)\mko e^i\\
       &=\sum_{i,j} \Big(  \big(e_j(\vec X_i)-e_i(\vec X_j)\mko \big) \we \vec e_j\Big) e^i+\sum_{i,j}\Big(e_i\big(\vec X_j \we \vec e_j \big)\mkt e^i-\big(\vec X_j\we e_i(\vec e_j)\mko\big)\mkt e^i \Big)\\
       &=\sum_{i,j} \Big(  \big(\vec{\mca R}'[\vec X;g]-d\vec X\big)_{ij} \we \vec e_j\Big) e^i+\sum_{i}e_i\big(\vec X\wres d\bP\big)\mkt e^i-\sum_{i,j}\big(\vec X_j\we e_i(\vec e_j)\mko\big)\mkt e^i\\
       &=\big(d\vec X-\vec{\mca R}'[\vec X;g]\big)\wres d\bP+d\big(\vec X\wres d\bP\big)-\sum_{i,j}\big(\vec X_j\we e_i(\vec e_j)\mko\big)\mkt e^i.
       \end{aligned}
   \end{align}
   By~\eqref{expofd*X}, we also have
   \begin{equation}\label{ejXje^id*X}
       -\sum_{i,j}\big(e_j(\vec X_j)\we \vec e_i \big)\mko e^i=\big(d^{*_g}\vec X- \vec{\mca R}[\vec X;g]\big)\we d\bP.
   \end{equation}
   Now we define
   \begin{equation}\label{defR1Xg}
   \begin{aligned}
       \vec {\mca R}_1[\vec X;g]
       &\coloneqq \vec{\mca R}\big[\vec X \ovs{\ovwe}{\we} d\bP;g\big]+\sum_{i,j}\big(\vec X_i\we e_j(\vec e_j)-\vec X_j\we e_j(\vec e_i)\mko\big)\mko e^i\\
       &\quad -\vec{\mca R}'[\vec X;g]\wres d\bP-\sum_{i,j}\big(\vec X_j\we e_i(\vec e_j)\mko\big)\mkt e^i-\vec{\mca R}[\vec X;g]\we d\bP.
       \end{aligned}
   \end{equation}
   Substituting~\eqref{e_jXiweej}--\eqref{defR1Xg} into~\eqref{d*XwedPfirex} then gives
   \begin{align*}
       d^{*_g}\big(\vec X \ovs{\ovwe}{\we} d\bP \big)=d\vec X \wres d\bP+d\big(\vec X \wres d\bP\big)+(d^{*_g} \vec X) \we d\bP+\vec {\mca R}_1[\vec X;g].
   \end{align*}
   The equation~\eqref{d*XwedPhi} is thus proved. Since $\vec e_i=d\bP(e_i)$, by~\eqref{gra-sch-4d} we have
   \begin{align*}
       |e_i(\vec e_j)|\le C(\La)\mko|\g^2\bP|\qquad \text{a.e. in }B^4.
   \end{align*}
  The pointwise bound~\eqref{vecR4ptbd} then follows from~\eqref{RXwedPh},~\eqref{ptbdRR'X}, and~\eqref{defR1Xg}.
\end{proof}
Before applying Proposition~\ref{prop:III3-1rev}, we provide here another lemma.
\begin{Lm}[cf.\ {\cite[Prop.~II.4]{Bernard25}}]\label{lm:nd*RS}
   Let $S\in L^{4/3}\big(B^4,\bwe^2 \R^4\big)$, and $\vec R\in L^{4/3}\big(B^4,\bwe^2\R^m\ot \bwe^2 \R^4\big)$. Then the following holds in $\mca D'(B^4)$: 
   \begin{align}\label{eq:norproRS}
      \pi_{\bn}\mkt d^{*_g}\Big( \big(\vet \,\ovs{\sbul}{\res}_g \vec R+ \vet \,\resg S\big)\mkt \res d\bP\Big)=-\pi_{\bn}  \big(\mko (d^{*_g} \bR\mko)\, \ovs{\res}\res_g d\bP \big)+\vec {\mca R}_2[\bR;g],
   \end{align}
    where the remainder $\vec {\mca R}_2[\bR;g]$ satisfies 
    \begin{align}\label{rmdbdnRS}
        |\vec {\mca R}_2[\bR;g]|\le C(\La)\mko |\bII|\mko |\bR| \qquad \text{a.e. in }B^4.
    \end{align}
\end{Lm}
\begin{proof}
    Since $\big(\vet \,\ovs{\sbul}{\res}_g \vec R+ \vet \,\resg S\big)\mkt \res d\bP\in L^{4/3}(B^4,\R^m\ot \R^4)$ and $\bn\in L^\nf \cap W^{1,4}(B^4,\R^m)$, the inequality~\eqref{prorL1W-143} and the proof of~\eqref{d*gL4/3bd} imply that
    \begin{align*}
         \pi_{\bn}\mkt d^{*_g}\Big( \big(\vet \,\ovs{\sbul}{\res}_g \vec R+ \vet \,\resg S\big)\mkt \res d\bP\Big)\in L^1+W^{-1,\frac 43}(B^4,\R^m).
    \end{align*}
    Similarly, it holds that
    \[
      \pi_{\bn} \big(\mko (d^{*_g} \bR\mko)\, \ovs{\res}\res_g d\bP \big)\in L^1+W^{-1,\frac 43}(B^4,\R^m).
   \]
    As in Lemma \ref{prop:III2ber}, by approximation it suffices to prove~\eqref{eq:norproRS}--\eqref{rmdbdnRS} for $\bR,S\in C^\nf$. Using the notation~\eqref{forcomnot}, we first compute
\begin{align*}
    \big(\mko(\vet\,\resg S)\mkt \res d\bP\big)_i&=\frac 12\sum_{j,k}S_{jk}\,\vet_{jk}\mkt \res \vec e_i\\
    &=\frac 12\sum_{j,k} S_{jk}(\de_i^j \mkt\vec e_k-\de_i^k\mkt\vec e_j)\\
    &= \sum_{k} S_{ik}\mkt \vec e_k\\
    &=-\big(S\,\resg d\bP\big)_i\,.
\end{align*}
By~\eqref{*_gcommwed} and~\eqref{*_gcomres}, it follows that
\begin{align}\label{vetresSnorvan}
\begin{aligned}
    \pi_{\bn}\mkt d^{*_g}\Big( \big( \vet \,\resg S\big)\mkt \res d\bP\Big)
     &=-\pi_{\bn}\mkt d^{*_g} \big(S\,\resg d\bP \big)\\
     &=\pi_{\bn} *_g d*_g  \big(S\,\resg d\bP \big)\\
     &=-\pi_{\bn} *_g \,d\big( \mko( *_g \,S)\we d\bP \big)\\
     &=\pi_{\bn} \big(\mko (d^{*_g} S) \,\resg d\bP\big)=0.
\end{aligned}
\end{align}We denote by $T\bP$ and $N\bP$ the tangent and normal bundles of $\bP(B^4)\subset\R^m$ respectively. Then $\bR$ admits a unique decomposition $\bR=\bR^{\perp\perp}+\bR^{\top\perp}+\bR^{\top\top}$ such that for a.e. $x\in B^4$, it holds that
    \begin{align}\label{defRperptop}
        \bR_{ij}^{\perp\perp}(x)\in \bwe^2N_x\bP,\qquad \bR^{\top\perp}_{ij}(x)\in T_x\bP\bwe N_x\bP,\qquad \bR_{ij}^{\top\top}(x)\in \bwe^2T_x\bP.
    \end{align}
For the right-hand side of~\eqref{eq:norproRS}, similar to~\eqref{vetresSnorvan}, we have
\begin{align}\label{rhsd*Rresdp}
\begin{aligned}
    \mkt(d^{*_g} \bR\mko)\,\ovs{\res}\res_g d\bP&=-(*_g \,d *_g \bR)\, \ovs{\res}{\res}_g d\bP\\
    &=-*_g \big(d (*_g \,\bR)  \ovs{\res}{\we} d\bP\big)\\
    &=-*_g d\big(\mkt (*_g \,\bR)  \ovs{\res}{\we} d\bP\big)\\
    &=*_g\, d*_g\big(\bR\,\ovs{\res}{\res}_g d\bP \big)\\
    &=-\mko d^{*_g} \big(\bR\, \ovs{\res}{\res}_g d\bP \big).
\end{aligned}
\end{align}
Since $\vet\,\ovs{\sbul}{\res}_g\bR^{\perp\perp}=0$ and $\bR^{\perp\perp}\,\ovs{\res}{\res}_g d\bP=0$, the identity~\eqref{rhsd*Rresdp} shows that it suffices to consider only the contributions of $\bR^{\top\perp}$ and $\bR^{\top\top}$ on both sides of~\eqref{eq:norproRS}. By~\eqref{defRperptop}, we have
\begin{align}\label{nres=0}
    \bn\,\res \Big(\big(\vet \,\ovs{\sbul}{\res}_g \bR^{\top\top}\big)\mkt \res d\bP  \Big)=0\qquad\text{and}\qquad \bn\,\res \big(\bR^{\top\top}\,\ovs{\res}{\res}_g d\bP \big)= 0.
\end{align}
Using~\eqref{pinvres} and~\eqref{nres=0}, we then estimate
\begin{align}\label{estd*Rresdp}
\begin{aligned}
    \Big|\pi_{\bn}\mkt d^{*_g}\big(\bR^{\top\top}\,\ovs{\res}{\res}_g d\bP\big)\Big|&=\Big|\bn\,\res d^{*_g}\big(\bR^{\top\top}\,\ovs{\res}{\res}_g d\bP \big) \Big|\\
    &=\Big|d^{*_g}\Big(\bn\,\res \big(\bR^{\top\top}\,\ovs{\res}{\res}_g d\bP \big) \Big)+*_g \Big(d\bn\ovs{\res}{\we} *_g\big(\bR^{\top\top}\,\ovs{\res}{\res}_g d\bP \big) \Big) \Big|\\
    &\le C(\La)\mko|\bR^{\top\top}|\mko|d\bn|\le C(\La)\mko|\bII|\mko |\bR|.
\end{aligned}
\end{align}
Similarly, we have
\begin{align}\label{d*RetaRdpest}
    \Big|\pi_{\bn}\mkt d^{*_g} \Big(\big(\vet \,\ovs{\sbul}{\res}_g \bR^{\top\top}\big)\mkt \res d\bP  \Big) \Big|\le C(\La)\mko |\bII|\mko |\bR|.
\end{align}
Now we define
\[
\vec {\mca R}_2[\bR;g]\coloneqq \pi_{\bn}\mkt d^{*_g} \Big(\big(\vet \,\ovs{\sbul}{\res}_g \bR^{\top\top}\big)\mkt \res d\bP  \Big)- \pi_{\bn}\mkt d^{*_g}\big(\bR^{\top\top}\,\ovs{\res}{\res}_g d\bP\big). 
\]
Then by~\eqref{vetresSnorvan} and~\eqref{rhsd*Rresdp}, we have
\begin{align}
\label{R2+d*Rperp}\begin{aligned}
     & \pi_{\bn} \mkt d^{*_g}\Big( \big(\vet \,\ovs{\sbul}{\res}_g \vec R+ \vet \,\resg S\big)\mkt \res d\bP\Big)+\pi_{\bn} \big(\mko (d^{*_g} \bR\mko)\,\ovs{\res}\res_g d\bP \big)\\
      &=\vec {\mca R}_2[\bR;g]+ \pi_{\bn}\mkt d^{*_g}\Big( \big(\vet \,\ovs{\sbul}{\res}_g \bR^{\top\perp}\big)\mkt \res d\bP\Big)-\pi_{\bn}\mkt d^{*_g}\big(\bR^{\top\perp}\,\ovs{\res}{\res}_g d\bP\big).
      \end{aligned}
\end{align}
In addition, the estimates~\eqref{estd*Rresdp}--\eqref{d*RetaRdpest} give
\[
     |\vec {\mca R}_2[\bR;g]|
     \le C(\La)\mko |\bII|\mko |\bR|\qquad \text{a.e. in }B^4.
\]
Finally, to compute the right-hand side of~\eqref{R2+d*Rperp}, let $x\in B^4$ and $\bv\in N_x\bP$. We have
\begin{align}\label{bsetaburesno}\begin{aligned}
    \Big(\vet \,\ovs{\sbul}{\res}_g\big( \mko(\bv \we \vec e_k)\ot ( e^i\we e^j)\big)\Big)\mkt \res d\bP&=\big(\vet_{ij}\sbul (\bv \we \vec e_k)\big)\mkt\res d\bP \\
    &=(\de_i^k\mkt  \bv \we \vec e_j- \de_j^k\mkt \bv\we \vec e_i)\mkt\res d\bP\\
    &= -\de_i^k\mkt  \bv\ot e^j +\de_j^k\mkt \bv \ot e^i\\
    &=\big( \mko(\bv \we \vec e_k)\ot(e^i\we e^j)\big)\mkt \ovs{\res}{\res}_g d\bP.
    \end{aligned}
\end{align}
Since for a.e. $x\in B^4$, $\bR^{\top\perp}(x)$ is spanned by $\big\{(\bv \we \vec e_k)\ot( e^i\we e^j)\colon 1\le i,j,k\le 4,\mkt\bv\in N_x\bP\big\}$, the equation~\eqref{bsetaburesno} implies
\begin{align}\label{etabuR=Rres}
     \big(\vet \,\ovs{\sbul}{\res}_g \bR^{\top\perp}\big)\mkt \res d\bP= \bR^{\top\perp} \mkt\ovs{\res}{\res}_g d\bP.
\end{align}
Combining~\eqref{R2+d*Rperp} with~\eqref{etabuR=Rres} yields~\eqref{eq:norproRS}. This completes the proof.
\end{proof}
Combining Proposition~\ref{prop:III3-1rev} and Lemmas~\ref{lm:d*XwedPhi}--\ref{lm:nd*RS}, we obtain the following equations for the system~\eqref{sys:d*dSR}.
\begin{Co}[cf.~{\cite[Thm.~I.6 and Cor.~I.1]{Bernard25}}]\label{co:sysvecu}
    Let $\bL$, $S$, $\bR$ be as in Proposition \ref{prop:III3-1rev}. Define 
    \begin{align}\label{defuetabul}
        \vec u\coloneqq\vet \,\ovs{\sbul}{\res}_g \vec R + \vet \,\resg S+3\mko d\bH \wres d\bP.
    \end{align}
    Then the following hold in $\mca D'\big(B^4, \bwe^2 \R^m\ot \bwe^4 \R^4\big)$ and $\mca D'(B^4,\R^m)$ respectively:
    \begin{align}
    &d*_gd\vec u=d*_g \vec{\mca R}_3,\label{eq:d*du=d*R3+}\\
        &\pi_{\bn}\mkt d^{*_g}( \vec u \,\res d\bP)=\lap_g \bH+ \vec{\mca R}_4,\label{eq:ndotd*uresdp}
    \end{align}
    where the remainders satisfy a.e. in $B^4$,
    \begin{align}
        |\vec{\mca R}_3|+ | \vec{\mca R}_4|\le C(\La) \big(|\vt_{\dil}|+|\bvt_{\rot}|+|\g^3\bP|\mko|\bH|+|\g^2\bP|\mko (|\bR |+|S|+|\g \bH|)\mko\big).\label{remtotptbd}
    \end{align}
   \end{Co}
   \begin{proof}
        By Proposition~\ref{prop:III3-1rev}, we obtain the following equation in $W^{-1,4/3}+L^1\big(B^4,\bwe^2 \R^m\ot \R^4\big)$:
   \begin{align}\label{d*3RdRS}
d^{*_{g}}\big(3\mko \vec R + \vet \,\ovs{\sbul}{\sbul}_g \vec R + \vet \,\bulg S\big)
+ d\big(\vet \,\ovs{\sbul}{\res}_g \vec R + \vet \,\resg S\big)
&= 3\mko d\bP\we \pi_{\bn}\mko\lap_g \bH+\vec{\mca R}_0,
\end{align}
with \begin{equation}\label{vecR0ptbdfinstr}
     | \vec{\mca R}_0|\le C(\La) \big(|\vt_{\dil}|+|\bvt_{\rot}|+ |\g^2\bP|\mko(|\bR |+|S|) \mko\big)\qquad \text{a.e. in }B^4.
\end{equation}
Set $\vec X=d\bH$. Since $d\vec X=0$ and $d^{*_g}\vec X=-\lap_g\bH$, by Lemma~\ref{lm:d*XwedPhi}, we have
\begin{align}\label{lapHndPdd*sys}\begin{aligned}
    d\bP\we \lap_g\bH&=(d^{*_g} \vec X) \we d\bP\\
    &= d^{*_g}\big(\vec X \ovs{\ovwe}{\we} d\bP \big)-d\big(\vec X \wres d\bP\big)-d\vec X \wres d\bP-\vec {\mca R}_1[\vec X;g]\\
    &= d^{*_g}\big(\vec X \ovs{\ovwe}{\we} d\bP \big)-d\big(\vec X \wres d\bP\big)-\vec {\mca R}_1[\vec X;g].
    \end{aligned}
\end{align}
In the above computation, the remainder term $\vec {\mca R}_1[\vec X;g]$ satisfies 
\begin{align}
    \label{vecR1ptbdfinstr}
     |\vec{\mca R}_1[\vec X;g]|\le C(\La)\mko|\g^2\bP|\mko |\g \bH|\qquad \text{a.e. in }B^4.
\end{align}
We denote by $\pi_T$ the orthogonal projection onto the tangent bundle $T\bP$ of $\bP(B^4)\subset \R^m$, and define 
\begin{equation}\label{defvecR2}
    \vec {\mca R}_3\coloneqq  \vec {\mca R}_0+3\lf(-d\bP\we\pi_{T}\lap_g\bH-\vec {\mca R}_1[\vec X;g] \rg).
\end{equation}
Combining~\eqref{d*3RdRS} and~\eqref{lapHndPdd*sys} with~\eqref{defvecR2} then yields
\begin{align}\label{d*dsysRSX}
    d^{*_{g}}\lf(3\mko \vec R + \vet \,\ovs{\sbul}{\sbul}_g \vec R + \vet \,\bulg S-3\mko d\bH \ovs{\ovwe}{\we} d\bP\rg)
+ d\vec u= \vec{\mca R}_3.
\end{align}
Now we estimate $\pi_T\lap_g\bH$. Since $\bH\cdot d\bP=0$, we have
\begin{align}\label{gpHdotpp}
    |\g(\p_j\bH\cdot \p_i \bP)|=|-\g(\bH\cdot \p_j\p_i\bP)|\le C(\La) \big( |\g^3\bP|\, |\bH|+|\g\bH|\, |\g^2\bP|\mko\big).
\end{align}
It follows that 
\begin{align}\label{estpiTlapH}
\begin{aligned}
    |\pi_T\mko \lap_g\bH|&=|g^{ij}(\lap_g\bH\cdot \p_i\bP) \p_j\bP|\\
    &=\big|-g^{ij}\big(\lan d\bH,d\mko\p_i\bP\ran_g-d^{*_g}(d\bH\cdot \p_i\bP)\mko\big)\p_j\bP\big|\\
    &\le C(\La) \big( |\g^3\bP|\mko |\bH|+|\g\bH|\mko|\g^2\bP|\mko\big).
    \end{aligned}
\end{align}
The desired estimate for $\vec{\mca R}_3$ in~\eqref{remtotptbd} then follows from~\eqref{vecR0ptbdfinstr}, \eqref{vecR1ptbdfinstr},~\eqref{defvecR2}, and~\eqref{estpiTlapH}.

By~\eqref{*gL1+W-143}, we can apply $d\,*_g$ to \eqref{d*3RdRS}, and since $d*_g d^{*_g}=-\,d*_g *_g \,d \,*_g=d^2 \,*_g=0$ on $2$-forms, we obtain in $W^{-2,4/3}+W^{-1,1}\big(B^4,\bwe^2 \R^m\ot \bwe^4 \R^4\big)$ that\footnote{Here we apply $d\,*_g$ instead of $d^{*_g}$ to~\eqref{d*3RdRS}, since $*_g$ is not well-defined on $W^{-2,4/3}+W^{-1,1}(B^4,\bwe \R^4)$. In fact, $*_g$ is even not well-defined on $W^{-2,(2,\nf)}\big(B^4,\bwe \R^4\big)$, see Remark~\ref{rm*W-22nf}.}
\begin{align}\label{d*dq=d*R}
\begin{aligned}
     d*_gd\vec u
    = d*_g \vec{\mca R}_3-d*_g d^{*_{g}}\lf(3\mko \vec R + \vet \,\ovs{\sbul}{\sbul}_g \vec R + \vet \,\bulg S-3\mko d\bH \ovs{\ovwe}{\we} d\bP\rg)=d*_g \vec{\mca R}_3.
\end{aligned}
\end{align}
The equation~\eqref{eq:d*du=d*R3+} is thus proved.
 Next, by Lemma~\ref{lm:nd*RS}, we obtain in $L^1+W^{-1,4/3}(B^4)$ that
\begin{align}\label{guNep+ra}
    \begin{aligned}
          \pi_{\bn}\mkt d^{*_g}\Big( \big(\vet \,\ovs{\sbul}{\res}_g \vec R+ \vet \,\resg S\big)\mkt \res d\bP\Big)=-\pi_{\bn}  \big(\mko (d^{*_g} \bR\mko)\, \ovs{\res}\res_g d\bP \big)+\vec {\mca R}_2[\bR;g],
    \end{aligned}
\end{align}
where the remainder satisfies
\begin{equation}\label{estR2}
        |\vec {\mca R}_2[\bR;g]|\le C(\La)\mko |\bII|\mko |\bR| \qquad \text{a.e. in }B^4.
 \end{equation}
The expression of $d^{*_g}\bR$ in~\eqref{sys:d*dSR} states that
\begin{align}\label{defd*R=-Lweres}
     d^{*_g} \vec R = \vec L \wres d\vec\Phi +\frac 12\mko d\bP\we \pi_{\bn}\mko\lap_g \bH+\bvt_{\rot}.
\end{align} 
Using the $g$-orthonormal frame $(e_i)$ and coframe $(e^i)$ as in~\eqref{bsetaburesno}, for $\vec v\in \R^m$ and $i\neq j$, we set $\vec \gamma\coloneqq\vec v \ot(e^i\we e^j)$ and compute 
\begin{align*}
   \pi_{\bn} \Big( \big(\vec \ga \wres d\vec\Phi\big)\ovs{\res}{\res}_g d\bP\Big)
    &=\pi_{\bn} \Big( \big(\mko(\vec v \we \vec e_i) \ot e^j -  (\vec v \we \vec e_j)\ot e^i\big) \ovs{\res}{\res}_g d\bP\Big)\\
    &=\pi_{\bn}  \big( \mko(\vec v \cdot \vec e_j)\mko \vec e_i- (\vec v \cdot \vec e_i) \mko\vec e_j \big)=0.
\end{align*}
Since $\bL= \sum_{i<j} \bL_{ij}\ot ( e^i\we e^j)$ with $\bL_{ij}\in W^{-1,4/3}+L^1(B^4,\R^m)$, it follows that in $W^{-1,4/3}+L^1(B^4)$,
\begin{align}\label{ndotLweresdpres}
   \pi_{\bn} \Big(\mko \big(\vec L \wres d\vec\Phi\mko\big) \mkt\ovs{\res}\res_g d\bP \Big)=0.
\end{align}
Hence by~\eqref{defd*R=-Lweres}--\eqref{ndotLweresdpres} and using $d\bP\,\ovs{\res}\res_g d\bP=4$, we have
\begin{align}\begin{aligned}
\label{ndotd*Rresdp}
     \pi_{\bn} \big(\mko (d^{*_g} \bR\mko)\, \ovs{\res}\res_g d\bP \big)
     &=2 \mko\pi_{\bn}\mko\lap_g \bH+\pi_{\bn}\big(\bvt_{\rot}\mkt\ovs{\res}\res_g d\bP \big).
     \end{aligned}
\end{align}
Moreover, we compute
\begin{align}\label{dHwdpresdp}
\begin{aligned}
    \big(d\bH \wres d\bP\big) \mkt\res d\bP
    &= (g^{ij}\p_i\bH \we \p_j\bP) \mkt\res d\bP\\
    &=g^{ij}(\p_i\bH\cdot d\bP)\mko \p_j\bP-g^{ij}(\p_j\bP\cdot d\bP)\mko\p_i\bH\\
    &=g^{ij}(\p_i\bH\cdot d\bP)\mko \p_j\bP-d\bH.
\end{aligned}
\end{align}
We define 
\begin{align}\label{defR4}
    \vec{\mca R}_4\coloneq -\pi_T\lap_g\bH-\pi_{\bn}\big(\bvt_{\rot}\mkt\ovs{\res}\res_g d\bP \big)+\vec {\mca R}_2[\bR;g]+3\mko\pi_{\bn}\mkt d^{*_g}\big(g^{ij}(\p_i\bH\cdot d\bP)\mko \p_j\bP \big).
\end{align}
Since $\bn\,\res \p_j\bP=0$, the same proof of~\eqref{estd*Rresdp} implies that
\begin{align}\label{estpind*dHdp}
    \big|\pi_{\bn}\mkt d^{*_g}\big(g^{ij}(\p_i\bH\cdot d\bP)\mko \p_j\bP \big)\big|\le C(\La)\mko |\bII|\mko |\g\bH|.
\end{align}
The pointwise bound for $\vec{\mca R}_4$ in~\eqref{remtotptbd} then follows from~\eqref{estpiTlapH},~\eqref{estR2}, and~\eqref{estpind*dHdp}.

Finally, combining~\eqref{guNep+ra} and~\eqref{ndotd*Rresdp}--\eqref{defR4} with the definition~\eqref{defuetabul} of $\vec u$ yields
\begin{align*}
    \pi_{\bn} \mkt d^{*_g}( \vec u \, \res d\bP)&=-\pi_{\bn}  \big(\mko (d^{*_g} \bR\mko)\, \ovs{\res}\res_g d\bP \big)+\vec {\mca R}_2[\bR;g]+3\mko \pi_{\bn}\mkt d^{*_g}\Big( \big(d\bH \wres d\bP\big) \mkt\res d\bP \Big)\\
    &=-2\mko\pi_{\bn}\mko \lap_g\bH-\pi_{\bn}\big(\bvt_{\rot}\mkt\ovs{\res}\res_g d\bP \big)+\vec {\mca R}_2[\bR;g]+3\mko\pi_{\bn}\mkt d^{*_g}\big(g^{ij}(\p_i\bH\cdot d\bP)\mko \p_j\bP \big)+3\mko \pi_{\bn}\mko\lap_g\bH\\
    &=\lap_g\bH+\vec {\mca R}_4.
\end{align*}
 This completes the proof.
\end{proof}
   
\section{Proof of Theorem~\ref{th-main}}\label{sec:pfmainThm}
In this section, we complete the proof of Theorem~\ref{th-main}. Let $\La\ge 1$ be a constant and $\vec{\Phi}\in \mathcal{I}_{1,2}(B^4, \R^m)$ satisfy
\begin{align}\label{weaimmconbP}
    \Lambda^{-1} \mko |v|^2_{\R^4} \le | d\vec{\Phi}_x(v)|^2_{\R^m}  \le \Lambda  \mko |v|^2_{\R^4},\qquad \text{for a.e. }x\in B^4 \text{ and all }v\in T_x B^4.
\end{align}
As in~\eqref{defE},  we fix $\vec c=(c_s)_{s=1}^8\in \R^8$, and define 
\[E_{\bc}(\bP)\coloneq \int_{B^4}\Big(|d\bH|_g^2+\sum_{s=1}^8 c_s P_s(g,\bII)\Big)\dvol_g.\]
We define the codifferential $d^{*_g}$ with respect to $g=g_{\bP}$ as in \eqref{defd*g4d}. Throughout Section~\ref{sec:pfmainThm}, we write $\p_i=\p_{x^i}$, and use the notations~\eqref{not-overset} and~\eqref{con-formsob}. To establish the regularity of $\bP$, we use the structural equations from~\cite{Bernard25} and carry out the analytic proof in full detail.   

\subsection{The Euler--Lagrange equation and estimate of the Noether current $\vec V$}\label{sec:EullagestV}
To prove Proposition~\ref{sysestLSR}, we use the pointwise invariance of the integrand by translations, dilations, and rotations in the ambient space to obtain the divergence-form Euler--Lagrange equation together with some conservation laws for weak critical points of $E$,  as in~\cite{Bernard25} and \cite{LaMaRi25}. For completeness, we present here a detailed proof. These conservation laws form the main ingredients for Proposition~\ref{sysestLSR}, proved at the end of the next subsection.

 \subsubsection*{The Noether current associated to translations.}

\begin{Lm}
\label{lm-critic-W^{1,1}} Let $\bP\in \mca I_{1,2}(B^4,\R^m)$ satisfy~\eqref{weaimmconbP}. Then there exist $\vec l_0\in W^{-1,\frac 43}+L^1(B^4,\R^m\ot \R^4)$ and $(\vec l_s)_{s=1}^8\subset L^1(B^4,\R^m\ot \R^4) $ depending on $\bP$ such that the following hold:
\begin{align}\label{ELequ}
    \begin{dcases}
    (i)\, \forall 1\le s\le 8, \quad |\vec l_s|\le C(\La)\big( |\g^2 \bn|\mko  |\bII|^2+ |\bII|^3(|\bII|+|\g g|)\mko \big)\quad \text{a.e.}\\
    (ii)\,\text{$\bP$ is a critical point of $E_{\bc}$ if and only if }d*_g\Big(\frac 12\mko d\lap_g\bH+\vec l_0+\sum_{s=1}^8c_s\vec l_s\Big)=0.
    \end{dcases}
\end{align}
\end{Lm}
\begin{proof}
 We first compute the variation $\frac d{dt}\, E_{\bc}(\bP+t\mko \bw)\big|_{t=0}$ for a smooth immersion $\bP\colon B^4\to \R^m$ and $\vec w\in C_c^\infty(B^4,\R^m)$. Then since $W^{2,2}(B^4)\hookrightarrow VMO(B^4)$, by the approximation result \cite[Thm.~IV.23]{LaMaRi25}, the computation remains valid for $\bP\in \mathcal{I}_{1,2}(B^4, \R^m)$.

 Let $\bP_t=\bP+t\mko \bw$, and we denote $g=g_{\bP}$, $g_t=g_{\vec{\Phi}_t}$, $\bH=\bH_{\bP}$, $\bH_t=\bH_{\bP_t}$, etc. By standard computations, see for instance \cite[Section V.1]{LaMaRi25}, we have
\begin{align}\label{varmet}\begin{dcases}
    \frac{d}{dt} \,g_t^{ij}\mko\Big|_{t=0}=-g^{ik}g^{\ell j}(\p_{k}\vec w \cdot\p_{\ell}\bP+\p_{\ell}\vec w \cdot\p_{k}\bP),\\
    \frac{d}{dt} \,\dvol_{g_t}\mko\Big|_{t=0}=d\bw \dwe *_g \,d\bP.
    \end{dcases}
\end{align}
We obtain the pointwise variation
\begin{align}\label{ptvar4d}
\begin{aligned}
    &\frac d{dt}\big(|d\bH_t|_{g_t}^2\,\dvol_{g_t}\big)\Big|_{t=0}\\
    &= \lf(\frac{d}{dt}\, g_t^{ij}\Big|_{t=0}\p_i \bH\cdot\p_j \bH+ 2\,\bigg\langle d\bH,d\Big(\frac{d}{dt}\mkt \bH_t\Big|_{t=0}\Big) \bigg\rangle_{\!g} \rg)\dvol_g+ |d\bH|_g^2\mkt \frac{d}{dt}\,\dvol_{g_t}\Big|_{t=0}\\
    &=-\mko 2\mko g^{ik}g^{\ell j}(\p_{k}\vec w \cdot\p_{\ell}\bP )\mkt(\p_i \bH\cdot\p_j \bH)\,\dvol_g-2\mkt (*_g \,d\bH)\dwe d\Big(\frac{d}{dt}\mkt \bH_t\Big|_{t=0}\Big)\\
    &\quad +|d\bH|_g^2\, d\bw \dwe *_g \,d\bP\\
    &=-\mko 2\mko g^{ik}g^{\ell j}(\p_{k}\vec w \cdot\p_{\ell}\bP )\mkt(\p_i \bH\cdot\p_j \bH)\,\dvol_g+2\mko d*_g\Big(\frac{d}{dt}\mkt \bH_t\Big|_{t=0}\cdot d\bH \Big)\\
    &\quad -2\mko \frac{d}{dt}\mkt \bH_t\Big|_{t=0}\cdot \lap_g \bH\,\dvol_g+|d\bH|_g^2\, d\bw \dwe *_g \,d\bP
    \end{aligned}
\end{align}
Next, we compute the pointwise variations of $\bn_t$ and $\bH_t$. As in Lemma~\ref{lm:Hdnstrequ}, let $(\bn_{\al,t})_{\al=1}^{m-4}\subset C^\nf\big(B^4\times (-1,1),\R^m\big)$ be orthonormal frames with $\bn_t=\bn_{1,t}\we\cdots \we \bn_{m-4,t}$, and denote by $\pi_T$ the orthogonal projection onto the tangent bundle of $\bP(B^n)\subset \R^m$. Then we compute
\begin{align*}
    \pi_T\mko \frac{d}{dt}\mko \bn_{\al,t}\mko\Big|_{t=0}
    =g^{ij}\Big(\p_{i}\bP \cdot \frac{d}{dt}\mko \bn_{\al,t}\mko\Big|_{t=0}\Big)\mko \p_{j}\bP
    =-g^{ij}\Big(\bn_\al \cdot \frac{d}{dt}\mko \p_{i}\bP_t\mko\Big|_{t=0}\Big)\mko\p_{j}\bP
    =-\mko d\bP\, \resg (\bn_\al\cdot d\bw).
\end{align*} 
 Since $\frac{d}{dt} \mko\bn_{\al,t}\mko\big|_{t=0}\cdot \bn_\al  =0$, it follows that
 \begin{equation}\label{vargau}
 \begin{aligned}
     \frac d{dt}\mko \bn_t\mko\Big|_{t=0}
     &=\sum_{\al=1}^{m-4} \Big(\pi_T \frac d{dt}\mko \bn_{\al,t}\mko\Big|_{t=0}\Big)\we (\bn\,\res \bn_\al)\\
     &=-\sum_{\al=1}^{m-4}\big(d\bP\, \resg (\bn_\al\cdot d\bw)\mko\big)\we (\bn\,\res \bn_\al)\\
     &=-\mko d\bP \wres (\bn\, \res d\bw).
     \end{aligned}
 \end{equation}
 By a similar argument, we obtain that
 \begin{align}\label{pn=II}
     \p_i\bn= -g^{jk} \p_k\bP \we (\bn \, \res \bII_{ij}).
 \end{align}
 Now for all $\vec v\in \R^m$, we have 
 \begin{equation}\label{resreppin}
    \pi_{\bn}\mko \vec v=(-1)^{m-1}\mko\bn\,\res(\bn\,\res \vec v).
 \end{equation}
 Hence, for all $\vec v_1,\vec v_2\in \R^m$, it holds that
 \begin{align}\label{norproequi}
     \vec v_1\cdot \pi_{\bn} \mko \vec v_2=(\bn\,\res \vec v_1)\cdot (\bn\,\res \vec v_2).
 \end{align}
 Write $\bII_{ij,t}=\pi_{\bn_t} \p_i\p_j\bP_t$. Since $\bn_t\,\res d\bP_t=0$, by~\eqref{vargau} we have
 \begin{align}\label{varnresII}
 \begin{aligned}
     \frac{d}{dt}\mko(\bn_t\,\res \bII_{ij,t})\mko\Big|_{t=0}
     &= \frac{d}{dt}\mko(\bn_t\,\res \p_i\p_j\bP_t)\mko\Big|_{t=0}\\
     &= -\frac{d}{dt}\mko(\p_i\mko\bn_t\,\res \p_j\bP_t)\mko\Big|_{t=0}\\
     &=-\p_i \bn \,\res \p_j \bw+\p_i\big(d\bP \wres (\bn\,\res d\bw)\mko \big)\,\res \p_j\bP.
     \end{aligned}
 \end{align}
 It follows that
\begin{align}\begin{aligned}\label{denresH}
      4\frac{d}{dt}\mko(\bn_t\,\res \bH_t)\mko\Big|_{t=0}
      &=g^{ij}  \frac{d}{dt}\mko(\bn_t\,\res \bII_{ij,t})\mko\Big|_{t=0}+(\bn \,\res \bII_{ij})\mko\frac{d}{dt} \,g_t^{ij}\mko\Big|_{t=0}\\
      &=-d\bn\, \rres d\bw+d\big(d\bP \wres (\bn\,\res d\bw) \mko\big)\rres d\bP-2\mko \bn \,\res \bII(d\bw,d\bP),
      \end{aligned}
 \end{align}
 where we write $\bII(d\bw,d\bP)\coloneq g^{ik}g^{\ell j}(\p_{\ell}\vec w \cdot\p_{k}\bP)\mkt \bII_{ij}$.
 Hence we have
 \begin{align}\label{varnresHdvol}\begin{aligned}
      4 \frac{d}{dt}\mko(\bn_t\,\res \bH_t)\mko\Big|_{t=0}\mko \dvol_g
       &=(*_g\,d\bn) \ovs{\res}{\we} d\bw-2\mko \bn \,\res \bII(d\bw,d\bP)\,\dvol_g+d\big(d\bP \wres (\bn\,\res d\bw) \mko\big) \ovs{\res}{\we}*_g\,d\bP.
       \end{aligned}
 \end{align}
 In addition, since $\bH$ is normal, using~\eqref{vargau}, we obtain
 \begin{align}\label{debnresnresH}
      \frac d{dt} \mko \bn_t \Big|_{t=0}\,\res (\bn\, \res \bH)=-\big( d\bP \wres (\bn\, \res d\bw)\mko\big)\,\res (\bn\,\res \bH)=(-1)^m \mko d\bP\, \resg (d\bw \cdot \bH).
 \end{align}
By~\eqref{resreppin}, we have
\begin{align}\label{varHt}
    (-1)^{m-1}\frac d{dt}\mko \bH_t\Big|_{t=0}&=\frac d{dt} \big(\bn_t\res (\bn_t\,\res \bH_t)\mko\big)\Big|_{t=0}=
    \frac d{dt} \mko \bn_t \Big|_{t=0}\,\res (\bn\, \res \bH)+\bn\,\res  \frac{d}{dt}\mko(\bn_t\,\res \bH_t)\mko\Big|_{t=0}.
\end{align}
Combining~\eqref{norproequi}--\eqref{varHt} with~\eqref{eq:defres}, it follows that 
\begin{align}\label{varHdotlapH}
\begin{aligned}
    &\frac{d}{dt}\mkt \bH_t\Big|_{t=0}\cdot \lap_g \bH\,\dvol_g\\
    &=-(d\bw \cdot \bH)\we (*_g\, d\bP \cdot \lap_g \bH)+  \frac{d}{dt}\mko(\bn_t\,\res \bH_t)\mko\Big|_{t=0}\mko \dvol_g\cdot (\bn\,\res \lap_g \bH)\\
    &=-(d\bw \cdot \bH)\we (*_g\, d\bP \cdot \lap_g \bH)+\frac 14\mko (*_g\,d\bn) \dwe \big(d\bw\we (\bn\,\res \lap_g\bH)\mko\big)-\frac 12\mko \bII(d\bw,d\bP)\cdot \lap_g \bH \,\dvol_g\\
    &\quad +\frac 14\mko d\big(d\bP \wres (\bn\,\res d\bw)\mko \big) \dwe\big(*_gd\bP \we (\bn\,\res \lap_g \bH)\mko\big)\\
    &=d\bw \dwe *_g \Big( -(d\bP \cdot \lap_g \bH) \bH +\frac 14(-1)^m\mko  d\bn\, \res (\bn\,\res \lap_g\bH)-\frac 12\mko (\bII^k_j \cdot \lap_g\bH)\p_k\bP\,  dx^j\Big)\\
    &\quad +\frac 14 d*_g\Big( \big(d\bP \wres (\bn\,\res d\bw)\mko\big) \cdot \big(d\bP \we (\bn\,\res \lap_g \bH)\mko\big) \Big)-\frac 14\big(d\bP \wres (\bn\,\res d\bw)\mko\big) \cdot d*_g\big(d\bP \we (\bn\,\res \lap_g \bH)\mko\big).
    \end{aligned}
\end{align}
Since $d*_g d\bP$ and $\bn\,\res d\bw$ are normal, by decomposition into tangential and normal components, we compute
\begin{align}\label{resdwddlapH}
\begin{aligned}
    &\big(d\bP \wres (\bn\,\res d\bw)\mko\big) \cdot d*_g\big(d\bP \we (\bn\,\res \lap_g \bH)\mko\big)\\
    &=- \big(d\bP \wres (\bn\,\res d\bw)\mko\big)\cdot \big(\mko (*_g\, d\bP) \ovs{\ovwe}{\we} d(\bn\,\res \lap_g\bH)\mko\big)\\
    &= \big\lan\bn\,\res d\bw, d(\bn\,\res \lap_g\bH)\big\ran_g\, \dvol_g\\
    &=(\bn\,\res d\bw) \dwe *_g \,d(\bn\,\res \lap_g\bH)\\
    &=(-1)^{m-1}d\bw \dwe *_g \big(\bn\, \res d(\bn\,\res \lap_g\bH)\mko\big)\\
    &=d\bw \dwe *_g \mkt  d\big( \pi_{\bn}\mko \lap_g \bH\big)+(-1)^{m}d\bw \dwe *_g \big(d\bn\, \res (\bn\,\res \lap_g\bH)\mko\big).
    \end{aligned}
\end{align}
We also have
\begin{align}\label{comdwdotlapH}
 \big(d\bP \wres (\bn\,\res d\bw)\mko\big) \cdot \big(d\bP \we (\bn\,\res \lap_g \bH)\mko\big)=(\bn\,\res d\bw)\cdot (\bn\,\res \lap_g\bH)=d\bw \cdot \pi_{\bn}\mko  \lap_g\bH.
\end{align}
Set 
\begin{align}
\begin{aligned}\label{defva}
    \vec l_0 &\coloneq 2\mko (d\bP \cdot \lap_g \bH) \bH -\frac 12(-1)^m d\bn\, \res (\bn\,\res \lap_g\bH)+(\bII^k_j \cdot \lap_g\bH)\p_k\bP\,  dx^j\\
    &\quad-\frac 12 \mko d\big(\pi_T\mko \lap_g\bH\big)+\frac 12 (-1)^m d\bn\, \res (\bn\,\res \lap_g\bH) -2\mko g^{\ell j} (\p_j\bH\cdot d\bH)\mko \p_\ell \bP+|d\bH|^2d\bP.
    \end{aligned}
\end{align}
Combining~\eqref{varHdotlapH}--\eqref{defva} and writing $\pi_{\bn}\lap_g\bH=\lap_g\bH-\pi_T\lap_g\bH$, we obtain
\begin{align}\label{varHdlapH}\begin{aligned}
    &-2\mko \frac{d}{dt}\mkt \bH_t\Big|_{t=0}\cdot \lap_g \bH\,\dvol_g\\
    &=d\bw\dwe *_g \Big(\frac 12 \mko d\mko \lap_g\bH+\vec l_0 +2\mko g^{\ell j} (\p_j\bH\cdot d\bH)\mko \p_\ell \bP-|d\bH|^2d\bP\Big)-\frac 12 d*_g(d\bw \cdot \pi_{\bn} \mko \lap_g\bH).
    \end{aligned}
\end{align}
By~\eqref{ptvar4d} and~\eqref{varHdlapH}, we then have
\begin{align}\label{var|dH|2vol}
\begin{aligned}
    \frac d{dt}\big(|d\bH_t|_{g_t}^2\,\dvol_{g_t}\big)\Big|_{t=0}
    &=d*_g\Big(2\mko \frac{d}{dt}\mkt \bH_t\Big|_{t=0}\cdot d\bH -\frac 12\mko d\bw \cdot \pi_{\bn} \mko \lap_g\bH \Big)+d\bw \dwe*_g \Big(\frac 12\mko d\lap_g\bH+\vec l_0\Big)\\
    &=d*_g\bigg(2\mko \frac{d}{dt}\mkt \bH_t\Big|_{t=0}\cdot d\bH -\frac 12\mko d\bw \cdot \pi_{\bn} \mko \lap_g\bH+\bw\cdot \Big(\frac 12\mko d\lap_g\bH+\vec l_0\Big) \bigg)\\
    &\quad-\bw\cdot d*_g\Big(\frac 12\mko d\lap_g\bH+\vec l_0\Big).
    \end{aligned}
\end{align}
For smooth immersions, this gives the first variation of $|d\bH|_{g}^2\,\dvol_{g}$. We now show that~\eqref{var|dH|2vol} remains valid for $\bP\in \mca I_{1,2}(B^4,\R^m)$. 
For such $\bP$, by the first equality of~\eqref{ptvar4d} and~\eqref{denresH}, the left-hand side of~\eqref{var|dH|2vol} is well-defined in $L^1\big(B^4,\bwe^4\R^4\big)$. Then by~\cite[Thm.~IV.23]{LaMaRi25}, it suffices to show that the right-hand side is well-defined in distribution. By~\eqref{*estlor} we have $\lap_g\bH\in W^{-1,2}(B^4,\R^m)$, hence $d\lap_g\bH\in W^{-2,2}(B^4,\R^m\ot\R^4)$. In addition, the pointwise estimate~\eqref{estpiTlapH} implies that $\pi_T\lap_g\bH\in L^{4/3}(B^4)$. To estimate $*_g\,d\lap_g\bH$, we take $a\in L^\nf\cap W^{2,2}(B^4)$ and $f\in L^{2}(B^4)$, and write
\begin{align*}
         a \, \p_i\p_j f=\p_i\p_j(af)-\p_i(f\,\p_j\mko a  ) -\p_j(f\,\p_i\mko a )+f\,\p_i\p_j\mko a.
\end{align*}
We have the following estimates:
\begin{align}\label{apijfW-22est}\begin{dcases}
   \big \|\p_i\p_j (af)\big\|_{W^{-2,2}(B^4)}\le \|af \|_{L^{2}(B^4)}\le \|a\|_{L^\nf(B^4)}\mko\|f\|_{L^{2}(B^4)},\\
    \big\|\p_i(f\,\p_j\mko a  )\big\|_{W^{-1,\frac 43}(B^4)}\le \|f\,\p_j\mko a  \|_{L^{\frac 43}(B^4)}\le  \|\g a\|_{L^4(B^4)}\mko \|f\|_{L^{2}(B^4)},\\
    \|f\,\p_i\p_j\mko a\|_{L^1(B^4)}\le \|f\|_{L^2(B^4)}\|\g^2 a\|_{L^2(B^4)}.
    \end{dcases}
\end{align}
Hence there exists a universal constant $C>0$ such that for all $a\in L^\nf\cap W^{2,2}(B^4)$ and $T\in W^{-2,2}+L^1(B^4)$, it holds that
\begin{equation}\label{W-22L1estaT}
    \|aT\|_{W^{-2,2}+L^1(B^4)}\le C\mko\|a\|_{L^\nf\cap W^{2,2}(B^4)}\|T\|_{W^{-2,2}+L^1(B^4)}.
\end{equation}
In particular, the operator $*_g\colon W^{-2,2}+L^1\big(B^4,\bwe^\ell \R^4\big)\to W^{-2,2}+L^1\big(B^4,\bwe^{4-\ell} \R^4\big)$ is bounded, and hence $*_g\,d\lap_g\bH\in W^{-2,2}+L^1\big(B^4,\R^m\ot \bwe^3\R^4 \big)$. Similar to~\eqref{W-22L1estaT}, we have the following inequality:
\begin{align}\label{ineL1+W-143}
    \|fT\|_{L^1+W^{-1,\frac 43}(B^4)}\le C\|f\|_{W^{1,2}(B^4)}\|T\|_{W^{-1,2}(B^4)}.
\end{align}
Since $d\bn\in W^{1,2}$, by~\eqref{prorL1W-143},~\eqref{estpiTlapH}, and~\eqref{ineL1+W-143}, we then obtain $*_g\,\vec l_0\in L^1+W^{-1,4/3}$. Using~\eqref{varnresHdvol}--\eqref{varHt}, we also have 
$$
\frac{d}{dt}\mkt \bH_t\Big|_{t=0}\in L^4(B^4,\R^m).
$$ 
It follows that~\eqref{var|dH|2vol} remains valid for $\bP\in \mca I_{1,2}(B^4,\R^m)$ satisfying~\eqref{weaimmconbP}.\\
Concerning the lower-order terms $P_s(g,\bII)$, by~\eqref{varmet} and~\eqref{varnresII}, for each $1\le s\le 8$ we show that there exist $\varrho_s\in L^{4/3}\big(B^4,TB^4\ot T^*B^4\ot (\R^m)^*\big)$ and $\vec l_s\in  L^1(B^4,\R^m\ot T^*B^4)$ such that
\begin{align}\label{condvrsls}
\begin{dcases}
(i)\, |\varrho_s|\le C(\La)\mko |\bII|^3\quad\text{a.e.}\\
(ii)\, |\vec l_s|\le C(\La)\big( |\g^2 \bn|\mko  |\bII|^2+ |\bII|^3(|\bII|+|\g g|)\mko \big)\quad \text{a.e.}\\
(iii)\,\frac d{dt}\big( P_s(g_t,\bII_t) \,\dvol_{g_t}\big)\Big|_{t=0}= d*_g \big(\varrho_s(d\bw)+\bw\cdot \vec l_s \big)-\vec w\cdot d*_g \vec l_s\quad\text{in }\mca D'\big(B^4,\bwe^4T^*B^4\big).   
\end{dcases}
\end{align}
We only present the explicit computation for $\frac d{dt} \big( P_4(g_t,\bII_t) \,\dvol_{g_t}\big)\big|_{t=0}$, and the variations of the other polynomials follow exactly in the same way. By~\eqref{eq:defres}, ~\eqref{varmet} and~\eqref{varnresII}, we compute
\begin{align}\label{varIIsqu}\begin{aligned}
    \frac d{dt}(|\bII_t|_{g_t}^2)\Big|_{t=0}&=2\mko \frac d{dt}\mko g^{ij}\Big|_{t=0} g^{k\ell}\mkt \bII_{ik}\cdot \bII_{j\ell} +2\mko g^{ij}g^{k\ell} \frac d{dt} (\bn_t\,\res \bII_{ik,t} )\Big|_{t=0}\cdot (\bn\,\res \bII_{j\ell})\\
    &=\big\lan d\bw, -4\mko g^{kj}(\bII_i^\ell\cdot \bII_{j\ell}) \p_k\bP\mkt dx^i\big\ran_g-2\mko g^{k\ell} \p_i\bn\cdot \big(\p_k\bw \we (\bn\,\res \bII_{\ell}^i)\mko\big)\\
    &\qquad+2\mko g^{ij}\lf(\p_i\big(d\bP \wres (\bn\,\res d\bw)\mko \big)\res \p_k\bP\rg) \cdot (\bn\,\res \bII_j^k)\\
    &=\big\lan  d\bw, -4\mko g^{kj}(\bII_i^\ell\cdot \bII_{j\ell}) \p_k\bP\mkt dx^i+2\mko (-1)^m \p_i\bn \,\res (\bn\,\res\bII_\ell^i)\mko dx^\ell\big\ran_g\\
    &\qquad +2\mko \Big \lan d\big( d\bP \wres (\bn\,\res d\bw)\mko\big), \p_k\bP\we (\bn\,\res \bII^k_j)\mko dx^j \Big\ran_g.
    \end{aligned}
\end{align}
Similar to~\eqref{resdwddlapH} and~\eqref{comdwdotlapH}, we have 
\begin{align}\label{compd()II}\begin{aligned}
     &\Big \lan d\big( d\bP \wres (\bn\,\res d\bw)\mko\big), \p_k\bP\we (\bn\,\res \bII^k_j)\mko dx^j \Big\ran_g |\bII|_g^2\,\dvol_g\\
     &=d*_g\Big(\big( d\bP \wres (\bn\,\res d\bw)\mko\big)\cdot \big(\p_k\bP\we (\bn\,\res \bII^k_j)\mko |\bII|_g^2\mko dx^j\big) \Big)-\big( d\bP \wres (\bn\,\res d\bw)\mko\big) \cdot d*_g\big(\p_k\bP\we (\bn\,\res \bII^k_j)\mko |\bII|_g^2 \mko dx^j\big)\\
     &=d*_g \big(\mko(\p_k\bw \cdot \bII^k_j) \mko |\bII|_g^2 \mkt dx^j\big)-(\bn \,\res d\bw) \dwe *_g\Big( \textup{div}_g\big(\p_k\bP\we (\bn\,\res \bII^k_j)\mko |\bII|_g^2 \mko dx^j\big)\,\res d\bP\Big)\\
     &=d*_g \big(\mko(\p_k\bw \cdot \bII^k_j) \mko |\bII|_g^2 \mkt dx^j\big)+(-1)^{m-1}d\bw \dwe *_g\bigg(\bn\,\res \Big( \textup{div}_g\big(\p_k\bP\we (\bn\,\res \bII^k_j)\mko |\bII|_g^2 \mko dx^j\big)\,\res d\bP\Big) \bigg).
     \end{aligned}
\end{align}
Now for any $\vec v\in \R^m$ and $1\le k\le 4$, we define $\varrho_4 (\vec v\mkt dx^k)\coloneq 2\mko (\vec v\cdot \bII_j^k)\mko  |\bII|_g^2\mkt dx^j$. We also set
\begin{align}\label{defl4}\begin{aligned}
\vec l_4&\coloneq |\bII|_g^4\, d\bP-8 \mko|\bII|_g^2\mkt g^{kj}(\bII_i^\ell\cdot \bII_{j\ell}) \p_k\bP\mkt dx^i+4\mko (-1)^m |\bII|_g^2\mkt \p_i\bn \,\res (\bn\,\res\bII_\ell^i)\mko dx^\ell\\
&\qquad +2(-1)^{m-1}\bn\,\res \Big( \textup{div}_g\big(\p_k\bP\we (\bn\,\res \bII^k_j)\mko |\bII|_g^2 \mko dx^j\big)\,\res d\bP\Big).
\end{aligned}
\end{align}
Combining~\eqref{varIIsqu}--\eqref{defl4}, we obtain
\begin{align*}
    &\frac d{dt}\big( P_4(g_t,\bII_t) \,\dvol_{g_t}\big)\Big|_{t=0}\\
    &=|\bII|_g^4 \,d\bw \dwe *_g \,d\bP+2\frac d{dt}(|\bII_t|_{g_t}^2)\Big|_{t=0}|\bII|_g^2\, \dvol_g\\
    &=d*_g\big(\varrho_4(d\bw)\mko\big)+d\bw \dwe *_g\,\vec l_4\\
    &=d*_g \big(\varrho_4(d\bw)+\bw\cdot \vec l_4 \big)-\vec w\cdot d*_g \vec l_4.
\end{align*}
As in \eqref{defV1st}, we set 
\begin{equation}\label{defVmase}
\vec V\coloneq\frac 12\mko d\lap_g\bH+\vec l_0+\sum_{s=1}^8c_s\vec l_s.
\end{equation}
Combining~\eqref{var|dH|2vol} with~\eqref{condvrsls} and using integration by parts then yields
 \begin{equation}\label{varEfor}
  \frac d{dt}\, E_{\bc}(\bP_t)\Big|_{t=0}=-\mko\big\lan d*_g \vec V,\bw \big\ran,
\end{equation}
 where $\big\lan\cdot,\cdot\big\ran$ denotes the canonical pairing between $\mca D'\big(B^4,\R^m\ot \bwe^4 \R^4\big)$ and $C_c^\nf(B^4,\R^m)$. This completes the proof.
\end{proof}

\subsubsection*{Rescaling.}
Let $\bP\in \mathcal{I}_{1,2}(B^4, \R^m)$ be a weak critical point of $E$ satisfying~\eqref{weaimmconbP}. Then by the Sobolev--Lorentz embedding, we have $\bP\in W^{2,(4,2)}(B^4,\R^m)$. For $\rho\in (0,1)$, define $\bP_\rho\in \mathcal{I}_{1,2}(B^4, \R^m)$ by $\bP_\rho(x)\coloneq \rho^{-1}\bP(\rho\mko x)$.
Then $\bP_\rho$ is also a critical point of $E$ since $E$ is invariant under rescaling. In addition, the condition~\eqref{weaimmconbP} remains valid if $\bP$ is replaced by $\bP_\rho$, and for a.e. $x\in B^4$, we have
\begin{align}\label{dilgnH}
    g_{ij,{\bP_\rho}}(x)=g_{ij,{\bP}}(\rho\mko x),\qquad \bn_{\bP_\rho}(x)=\bn_{\bP} (\rho\mko x),\qquad \bH_{\bP_\rho} (x)=\rho \mko \bH_{\bP} (\rho\mko x).
\end{align}
Consequently, by change of variables,
\begin{align}\label{gggndilinv}\begin{aligned}
   & \|\g g_{\bP_\rho}\|_{L^{4,2}(B^4)}=\|\g g_{\bP}\|_{L^{4,2}(B_\rho(0))},\qquad \|\g^2 g_{\bP_\rho}\|_{L^2(B^4)}=\|\g^2 g_{\bP}\|_{L^2(B_\rho(0))},\\
    &\|\g \bn_{\bP_\rho}\|_{L^{4,2}(B^4)}=  \|\g\bn_{\bP}\|_{L^{4,2}(B_\rho(0))},\qquad \|\g^2\bn_{\bP_\rho}\|_{L^2(B^4)}=\|\g^2 \bn_{\bP}\|_{L^2(B_\rho(0))}.
    \end{aligned}
\end{align}
Thus by dilation and restriction to small enough balls, we may without loss of generality assume 
\begin{align}\label{defep_0}
    \big\||\g g|+|\g \bn|\big\|_{L^{4,2}(B^4)}+\big\||\g^2 g|+|\g^2 \bn|\big\|_{L^2(B^4)}\le 1.
\end{align} 
Then it follows from Remark~\ref{rm:g^2Pdecom} that
\begin{align}\label{gHL2bd}
    \|\g^2\bP\|_{L^{4,2}(B^4)}+\|\g^3 \bP\|_{L^2(B^4)}\le C(\La).
\end{align}
\subsubsection*{Estimates of $\vec V$.}
By~\eqref{defep_0} and~\eqref{gHL2bd}, we have 
\begin{equation}\label{dpdHdHL1}
    \big\|g^{\ell j} (\p_j\bH\cdot d\bH)\mko \p_\ell \bP\big\|_{L^1(B^4)}+\big\|\mkt|d\bH|^2d\bP\big\|_{L^1(B^4)}
    \le C(\La)\|\g \bH\|_{L^2(B^4)}^2
    \le C(\La)\|\g \bH\|_{L^2(B^4)}.
\end{equation}
Applying the inequality~\eqref{*estlor}, we also obtain
\begin{align}\label{estlapH}
    \|\lap_g\bH\|_{W^{-1,2}(B^4)}\le C(\La) \|\g \bH\|_{L^2(B^4)}.
\end{align}
Moreover, combining the pointwise estimates~\eqref{gpHdotpp}--\eqref{estpiTlapH} with~\eqref{gHL2bd} implies that 
\begin{align*} 
    \|\pi_T\lap_g\bH\|_{L^{\frac 43}(B^4)}&\le C(\La)\Big( \|\g^3\bP\|_{L^2(B^4)}\|\bH\|_{L^4(B^4)}+\|\g\bH\|_{L^2(B^4)}\big\||\g g|+|\bII|\big\|_{L^4(B^4)}\Big)\\
    &\le C(\La)\big(\|\g\bH\|_{L^2(B^4)}+\|\bII\|_{L^4(B^4)}\big).
\end{align*}
By the inequality~\eqref{ineL1+W-143} and the definition~\eqref{defva} of $\vec l_0$, it follows that
\begin{align}\label{estl0}
    \|\vec l_0\|_{L^1+W^{-1,\frac 43}(B^4)}\le C(\La) \big(\|\g\bH\|_{L^2(B^4)}+\|\bII\|_{L^4(B^4)}\big).
\end{align}
The embedding $W^{-1,\frac 43}(B^4)\hookrightarrow W^{-2,2}(B^4)$ (see Lemma~\ref{lm:LpembW-1}) then gives 
\begin{align*}
    \|\vec l_0\|_{L^1+W^{-2,2}(B^4)}\le C(\La) \big(\|\g\bH\|_{L^2(B^4)}+\|\bII\|_{L^4(B^4)}\big).
\end{align*}
  In addition, by~\eqref{condvrsls} and~\eqref{defep_0}, for each $1\le s\le 8$, we have 
  \begin{align}\label{estls}
      \|\vec l_s\|_{L^1(B^4)}\le C(\La)\|\bII\|_{L^4(B^4)}.
  \end{align}
Combining~\eqref{estlapH}--\eqref{estls} with the definition~\eqref{defVmase} of $\vec V$, we obtain that
\begin{gather}\label{est-V}
\|\vec V\|_ {W^{-2,2}+L^1(B^4)}\le C(\La,\bc)\big(\|\g \bH\|_{L^2(B^4)}+\|\bII\|_{L^4(B^4)}\big).\end{gather}
Then it follows from~\eqref{W-22L1estaT}, \eqref{defep_0}, and~\eqref{est-V} that
\begin{align}\label{*VW-22+L1est}
    \|*_g \vec V\|_{W^{-2,2}+L^1(B^4)}&\le C(\La)\mko\|g\|_{W^{2,2}(B^4)} \|\vec V\|_{W^{-2,2}+L^1(B^4)}
    \le C(\La,\bc)\big(\|\g \bH\|_{L^2(B^4)}+\|\bII\|_{L^4(B^4)}\big).
\end{align}
Using the embeddings $L^1(B^4)\hookrightarrow W^{-1,(4/3,\nf)}(B^4)\hookrightarrow W^{-2,(2,\nf)}(B^4)$ (see Lemma~\ref{lm:LpembW-1}), we obtain
\begin{align*}
     \|*_g \vec V\|_{W^{-2,(2,\nf)}(B^4)}\le C(\La,\bc)\big(\|\g \bH\|_{L^2(B^4)}+\|\bII\|_{L^4(B^4)}\big).
\end{align*}
\subsection{The Noether currents associated to dilations and rotations}\label{sec:conlaws}
By Lemma~\ref{lm-critic-W^{1,1}}, for a weak critical point $\bP$ of $E$, we have $d *_g\vec V=0$. Hence by Corollary~\ref{co-Hod-decw-1p} and~\eqref{*estlor}, there exists $\bL\in W^{-1,(2,\nf)}\big(B^4,\R^m\ot \bwe^2 \R^4\big)$ such that
\begin{subnumcases}{}
    d*_g\bL=*_g\,\vec V, \label{d*gL=*V}\\[0.3ex]
    d\bL=0.\label{dL0=0}
\end{subnumcases}
Moreover, we have the estimate
\begin{align}
\begin{aligned}
    \label{est-L_0}
    \|\bL\|_{W^{-1,(2,\infty)}(B^4)}&\le C(\La)\|*_g\bL\|_{W^{-1,(2,\infty)}(B^4)}\big(1+ \|\g g\|_{L^4(B^4)}\big)\\
    &\le C(\La) \|*_g\vec V\|_{W^{-2,(2,\infty)}(B^4)}\\
    &\le C(\La,\bc)\big(\|\g \bH\|_{L^2(B^4)}+\|\bII\|_{L^4(B^4)}\big).
\end{aligned}
\end{align}
 \begin{Rm}\label{rm*W-22nf}
     Under our assumption $g\in L^\nf\cap W^{2,2}(B^4,\R^{4\times 4}_{\sym})$ is uniformly elliptic, in general we cannot solve the equation $d^{*_g} \bL=\vec V$ (equivalent to \eqref{d*gL=*V}) with the inequality
     \begin{align*}
     \|\bL\|_{W^{-1,(2,\infty)}(B^4)}\le C\| \vec V \|_{W^{-2,(2,\infty)}(B^4)}.
     \end{align*}
     This is because for $a\in  L^\nf\cap W^{2,2}(B^4)$ and $T\in W^{-2,(2,\infty)}(B^4)$, in general we do not have $a\mko T\in W^{-2,(2,\nf)}(B^4)$, and hence $*_g$ may not be bounded on $W^{-2,(2,\nf)}(B^4,\bwe\R^4)$. For instance, we take $a=\sin\big(\log\log\frac {e}{|x|}\big)$ and $f(x)=|x|^{-2}$. Then we have $a\in L^\nf\cap W^{2,2}(B^4) $ and $f\in L^{2,\nf}(B^4)$. We write formally \begin{equation}\label{sym*W-22nf}
         a \, \p_i\p_j f=\p_i\p_j(af)-\p_i(f\,\p_j\mko a  ) -\p_j(f\,\p_i\mko a )+f\,\p_i\p_j\mko a.
     \end{equation}
And we have \begin{align*}\begin{dcases}
   \big \|\p_i\p_j (af)\big\|_{W^{-2,(2,\nf)}(B^4)}\le \|af \|_{L^{2,\nf}(B^4)}\le \|a\|_{L^\nf(B^4)}\mko\|f\|_{L^{2,\nf}(B^4)},\\
    \big\|\p_i(f\,\p_j\mko a  )\big\|_{W^{-1,(\frac 43,\nf)}(B^4)}\le \|f\,\p_j\mko a  \|_{L^{\frac 43,\nf}(B^4)}\le C \|\g a\|_{L^4(B^4)}\mko \|f\|_{L^{2,\nf}(B^4)}.
    \end{dcases}
\end{align*}
However, for all $1\le i,j\le 4$, the last term $f\,\p_i\p_j\mko a$ in \eqref{sym*W-22nf} is not locally integrable in $B^4$, and hence $  a \, \p_i\p_j f$ is not well-defined in $W^{-2,(2,\nf)}(B^4)$.
If we assume in addition that $a\in W^{2,(2,1)}(B^4)$, then since $L^{2,\nf}(B^4)=\big(L^{2,1}(B^4)\big)^*$, for $f\in L^{2,\nf}(B^4)$ we have
\begin{align*}
     \|f\,\p_i\p_j\mko a\|_{W^{-2,(2,\nf)}(B^4)}\le C\mko  \|f\,\p_{i}\p_{j}\mko a\|_{L^1(B^4)}\le C\mko\|\g^2 a\|_{L^{2,1}(B^4)}\mko \|f\|_{L^{2,\nf}(B^4)}.
\end{align*}
It follows that $*_g$ is a bounded operator on $W^{-2,(2,\nf)}(B^4,\bwe \R^4)$ if we assume $g\in W^{2,(2,1)}(B^4,\R^{4\times 4}_{\sym})$ in addition to the previous assumptions.
\end{Rm}

Since $\Big(|d\bH|_g^2+\sum_{s=1}^8 c_s P_s(g,\bII)\Big)\dvol_g$ is pointwise invariant under dilations and rotations, Noether's theorem yields the corresponding conservation laws, which we now establish.
\begin{Lm}
\label{lm-dilations}
Let $\bP$ satisfying~\eqref{weaimmconbP} be a weak critical point of $E$, and define $\bL$ as in~\eqref{d*gL=*V}--\eqref{est-L_0}. Then there exist $\vt_{\dil}\in L^{4/3}(B^4,\R^4)$ and $\bvt_{\rot}\in L^{4/3}\big(B^4,\bwe^2\R^m\ot \R^4 \big)$ such that
\begin{subnumcases}{}
d*_g\big( \vec{L} \,\dres  d\bP+\vt_{\dil}\mko\big)=0,\label{cons-law-dilation}\\[1ex]
d*_g \Big(\vec L \wres d\vec\Phi +\frac 12\mko d\bP\we \pi_{\bn}\mko\lap_g \bH+\bvt_{\rot} \Big)=0.\label{cons-law-rotation}
\end{subnumcases}
In addition, the following pointwise bound holds a.e. in $B^4$:
\begin{align}\label{ptbdvts}
    |\vt_{\dil}|+|\bvt_{\rot}|\le C(\La,\bc)\big(|\bII|^3+(|\bII|+|\g g|) |\g \bH| \big).
\end{align}
\end{Lm}
\begin{proof}
To prove~\eqref{cons-law-dilation}, we first consider the variation $\bP_t=(1+t)\bP$ with 
$$
    \bw=\frac d{dt}\mko \bP_t\Big|_{t=0}=\bP.
$$
Denote $\bH_t$, $\bII_t$ and $g_t$ as in the proof of Lemma~\ref{lm-critic-W^{1,1}}. Since $\big(|d\bH_t|_g^2+\sum_{s=1}^8 c_s P_s(g_t,\bII_t)\mko\big)\dvol_{g_t}$ is independent of $t$, using~\eqref{var|dH|2vol} and~\eqref{condvrsls}, we obtain that
\begin{align}\label{dilconcomp1}
\begin{aligned}
    0&=   \frac d{dt}\bigg(\Big(|d\bH_t|_g^2+\sum_{s=1}^8 c_s P_s(g_t,\bII_t)\Big)\dvol_{g_t}\bigg)\bigg|_{t=0}\\
    &=d*_g\bigg(2\mko \frac{d}{dt}\mkt \bH_t\Big|_{t=0}\cdot d\bH -\frac 12\mko d\bw \cdot \pi_{\bn} \mko \lap_g\bH+\sum_{s=1}^8c_s\mko\varrho_s(d\bw) +\bw\cdot \vec V\bigg).
    \end{aligned}
\end{align}
Since $d\bP\cdot \pi_{\bn} \mko \lap_g\bH=0$, taking $\bw=\bP$ in~\eqref{dilconcomp1} gives
\begin{align}\label{conlawdil1vers}
    d*_g\bigg(2\mko \frac{d}{dt}\mkt \bH_t\Big|_{t=0}\cdot d\bH+\sum_{s=1}^8c_s\mko\varrho_s(d\bP) +\bP\cdot \vec V\bigg)=0.
\end{align}By~\eqref{d*gL=*V} and~\eqref{*_gcomres}, we have
\begin{align}\label{d*(PV)}
    d*_g(\bP\cdot \vec V)=d(\bP\cdot d*_g \bL)=d\bP\dwe d*_g \bL=-d\big( \mko(*_g\,\bL)\dwe d\bP\big)=d*_g(\vec{L} \,\dres  d\bP).
\end{align}
Set 
\begin{equation}\label{defvtdil}
    \vt_{\dil}\coloneq 2\mko \frac{d}{dt}\mkt \bH_t\Big|_{t=0}\cdot d\bH+\sum_{s=1}^8c_s\mko\varrho_s(d\bP).
\end{equation}
The conservation law~\eqref{cons-law-dilation} then follows by combining~\eqref{conlawdil1vers}--\eqref{defvtdil}. To estimate $\vt_{\dil}$, taking $\bw=\bP$ in~\eqref{denresH} and~\eqref{varHt} yields the pointwise bound
\begin{align}\label{ptbddtH}
    \Big|\frac{d}{dt}\mkt \bH_t\Big|_{t=0}\Big|\le C(\La)|\bII| \qquad \text{a.e. in }B^4.
\end{align}
Combining~\eqref{defvtdil}--\eqref{ptbddtH} with~\eqref{condvrsls}, we obtain that
\begin{align}\label{estvtdil}
    |\vt_{\dil}|\le C(\La,\bc)\big(|\bII|^3+|\bII|\mko |\g \bH| \big) \qquad \text{a.e. in }B^4.
\end{align}
Next, we prove~\eqref{cons-law-rotation}.  For $\vec a\in \bwe^2\R^m$ we define $\bP_t$ by
\begin{align}\label{varrot}
  \begin{dcases} 
  \frac d{dt}\mko \bP_t=\vec a\,\res \bP_t, 
  \qquad t\in \R,\\[0.3ex]
  \bP_0=\bP.
  \end{dcases}
\end{align} 
For this variation, since $ \frac d{dt}\mko \bP_t\cdot \bP_t=\vec a \cdot (\bP_t\we \bP_t)=0$, we have $|\bP_t|^2=|\bP|^2$ for all $t\in \R$. In addition, $\bP_t$ depends linearly on $\bP$.
Hence there exists $Q_t\in \text{SO}(m)$, depending only on $\vec a$ and $t$, such that $\bP_t=Q_t\circ \bP$ on $B^4$. Consequently, the expression $\big(|d\bH_t|_g^2+\sum_{s=1}^8 c_s P_s(g_t,\bII_t)\mko\big)\dvol_{g_t}$ is independent of $t$, and the equation~\eqref{dilconcomp1} remains valid if we take $\bw=\vec a\,\res \bP$. Identifying $(\R^m)^*=\R^m$ and using~\eqref{eq:defres}, we obtain that 
\begin{align}\label{rotvps}
    \varrho_s\big(d(\vec a\,\res \bP)\mko\big)=\varrho_s(\vec a\,\res d\bP)=(\vec a\,\res \p_i\bP)\cdot \varrho_s(dx^i)=\vec a\cdot \big(\p_i\bP\we \varrho_s(dx^i)\mko\big).
\end{align}
Similarly, we have
\begin{align}\label{dwdotpnlapH}
    d(\vec a\,\res \bP) \cdot \pi_{\bn} \mko \lap_g\bH=\vec a\cdot (d\bP \we \pi_{\bn} \mko \lap_g\bH).
\end{align}
By applying~\eqref{eq:defres} again to the equations~\eqref{var|dH|2vol} and~\eqref{condvrsls}, for the variation~\eqref{varrot} we can write
\begin{align}
    2\mko \frac{d}{dt}\mkt \bH_t\Big|_{t=0}\cdot d\bH +\sum_{s=1}^8c_s\mko\varrho_s\big(d(\vec a\,\res \bP)\mko\big)=-\mko \vec a\cdot \bvt_{\rot},
\end{align}
where $\bvt_{\rot}$ is independent of $\vec a$, and by~\eqref{get<II+gg} we have
\begin{align}\label{estvtrot}
     |\bvt_{\rot}|\le C(\La,\bc)\big(|\bII|^3+(|\bII|+|\g g|) |\g \bH| \big) \qquad \text{a.e. in }B^4.
\end{align}
Finally, we argue as in~\eqref{d*(PV)}:
\begin{align}\label{rotd*aV}
    d*_g\big(\mko(\vec a\,\res \bP)\cdot \vec V\big)=\vec a\cdot d*_g(\bP\we \vec V)=\vec a\cdot d(\bP\we d*_g\bL)=\vec a\cdot d\big(\mko(*_g\,\bL) \ovs{\ovwe}{\we} d\bP\big)=-\mko\vec a\cdot d*_g\big(\bL \wres d\bP\big).
\end{align}
The conservation law~\eqref{cons-law-rotation} then follows by combining~\eqref{dwdotpnlapH}--\eqref{rotd*aV} and using that $\vec a$ is arbitrary. This completes the proof.
\end{proof}
Define $\bL$ as in~\eqref{d*gL=*V}--\eqref{est-L_0}. Then by~\eqref{dL0=0}, we have
\begin{align}
\begin{dcases}\label{dLwedbp}
    d(\vec L\,\dot \we\, d\bP)=d\vec L \dwe d\bP=0,\\[0.5ex]
    d\big(\bL \ovs{\ovwe}{\we} d\bP\big)=d\bL \ovs{\ovwe}{\we} d\bP=0.
    \end{dcases}
    \end{align}
Invoking \eqref{prolorgen} and~\eqref{defep_0}, we estimate 
\begin{align}\label{estLwedplapH}
    \begin{dcases}
       \big\|\vec{L} \,\dres  d\bP\big\|_{W^{-1,(2,\nf)}(B^4)}+ \big\|\vec L \wres d\vec\Phi\big\|_{W^{-1,(2,\nf)}(B^4)}\le C(\La) \|\bL \|_{W^{-1,(2,\nf)}(B^4)},\\
   \big\|\bL\dwe d\bP\big\|_{W^{-1,(2,\nf)}(B^4)}+ \big\|\vec L \ovs{\ovwe}{\we} d\vec\Phi\big\|_{W^{-1,(2,\nf)}(B^4)}\le C(\La) \|\bL \|_{W^{-1,(2,\nf)}(B^4)},\\
        \big\|d\bP\we \pi_{\bn}\mko\lap_g \bH\big\|_{W^{-1,2}(B^4)}\le C(\La) \big\|\lap_g \bH\big\|_{W^{-1,2}(B^4)}\le C(\La)\|\g \bH\|_{L^2(B^4)}.
    \end{dcases}
\end{align}
By~\eqref{ptbdvts} and~\eqref{defep_0}, together with the embedding $L^{4/3}(B^4)\hookrightarrow W^{-1,2}(B^4)$, we also have
\begin{align}\label{estvts}
    \|\vt_{\dil}\|_{W^{-1,2}(B^4)}+\|\bvt_{\rot}\|_{W^{-1,2}(B^4)}\le C\big\| |\vt_{\dil}|+|\bvt_{\rot}|\big\|_{L^{\frac 43}(B^4)}\le C(\La,\bc) \big(\|\g \bH\|_{L^2(B^4)}+\|\bII\|_{L^4(B^4)}\big).
\end{align}
Combining the equations~\eqref{cons-law-dilation},~\eqref{cons-law-rotation}, and~\eqref{dLwedbp} with Corollary \ref{Co:Hod-declp}, we obtain $S\in L^{2,\infty}_{\loc}\big(B^4,\bwe ^2 \R^4\big)$ and $\bR\in L^{2,\infty}_{\loc}\big(B^4,\bwe^2\R^m\ot\bwe ^2 \R^4\big)$ such that, in $\mca D'(B^4)$ we have
    \begin{align*}
    \begin{dcases}
        d^{*_g}S=\vec{L} \,\dres  d\bP+\vt_{\dil},\quad &dS=-2\mko \bL \dwe d\bP,\\[0.5ex]
        d^{*_g}\bR=\vec L \wres d\vec\Phi +\frac 12\mko d\bP\we \pi_{\bn}\mko\lap_g \bH+\bvt_{\rot},\quad &    d\bR= -2\mko \bL\ovs{\ovwe}{\we} d\bP.
    \end{dcases}
    \end{align*}
    Moreover, applying the estimates~\eqref{estsiLpq},~\eqref{est-L_0}, and~\eqref{estLwedplapH}--\eqref{estvts} yields that
    \begin{gather*}
        \|S\|_{L^{2,\infty}(B_{1/2}(0))}+ \|\bR\|_{L^{2,\infty}(B_{1/2}(0))}\le C(\La,\bc)\big(\|\g \bH\|_{L^2(B^4)}+\|\bII\|_{L^4(B^4)}\big).
    \end{gather*}

We summarize our results of Sections~\ref{sec:EullagestV}--\ref{sec:conlaws} in the following theorem.
\begin{Prop}\label{sysestLSR}
Assume $\bP\in \mca I_{1,2}(B^4,\R^m)$ is a weak critical point of the functional 
$$
    E_{\bc}(\vec{\Phi})=\int_{B^4}\bigg(|d\bH|_g^2+\sum_{s=1}^8 c_s P_s(g,\bII)\bigg)\dvol_g. 
$$
Then there exist $\bL\in W^{-1,(2,\nf)}\big(B^4,\R^m\ot \bwe^2 \R^4\big)$, $S\in L^{2,\nf}_{\loc}\big(B^4,\bwe^2\R^4\big)$, $\vec R\in L^{2,\nf}_{\loc}\big(B^4,\bwe^2 \R^m\ot \bwe^2\R^4\big)$,  $\vec l_0\in W^{-1,\frac 43}+L^1(B^4,\R^m\ot \R^4)$, $\vec l_s\in L^1(B^4,\R^m\ot \R^4)$ ($1\le s\le 8$), $\vt_{\dil}\in L^{4/3}(B^4,\R^4)$, and $\bvt_{\rot}\in L^{4/3}\big(B^4,\bwe^2 \R^m\ot \R^4\big)$ such that the following system holds in $\mca D'(B^4)$:
\begin{equation}\label{sys-LRS}
    \begin{dcases}
        d *_g \bL=*_g\bigg(\frac 12\mko d\lap_g\bH+\vec l_0+\sum_{s=1}^8c_s\vec l_s\bigg),\\
        d^{*_g} S=\vec{L} \,\dres  d\bP+\vt_{\dil},\\
           d^{*_g} \vec R = \vec L \wres d\vec\Phi +\frac 12\mko d\bP\we \pi_{\bn}\mko\lap_g \bH+\bvt_{\rot},\\[0.5ex]
           d\bL=0,\\[0.5ex]
       dS=-2\mko \vec L\,\dot \we\, d\bP,\\[0.5ex]
        d \vec R = -2\mko \vec L \ovs{\ovwe}\we d\vec\Phi.
\end{dcases}
\end{equation}
Moreover, assuming that~\eqref{weaimmconbP} and~\eqref{defep_0} hold, then $\vt_{\dil}$, $\vec\vt_{\rot}$, $\vec L$, $S$, and $\bR$ can be chosen such that the following estimates hold:
\begin{align}\label{estLSR}
    & \|\bL\|_{W^{-1,(2,\infty)}(B^4)}+ \|S\|_{L^{2,\infty}(B_{1/2}(0))}+ \|\bR\|_{L^{2,\infty}(B_{1/2}(0))}\le C(\La,\bc)\left(\|\g \bH\|_{L^2(B^4)}+\|\bII\|_{L^4(B^4)}\right),\\[2mm]
     & |\vt_{\dil}|+|\bvt_{\rot}|\le C(\La,\bc)\left(|\bII|^3+\big( |\bII|+|\g g| \big)\, |\g \bH| \right)\quad \text{a.e. in }B^4.\label{ptbdvts2}
\end{align}
\end{Prop}
\subsection{Regularity}\label{sec:4dreg}
\begin{Th}\label{th:lapHL432}
    Suppose $\vec{\Phi}\in \mathcal{I}_{1,2}(B^4, \R^m)$ is a weak critical point of $E$ and satisfies~\eqref{weaimmconbP}. Then there exists $\vae_0=\vae_0(\La)>0$ such that for any $\al\in (0,3)$ and all $r\in(0,\frac 14]$, the following holds: Assume that
    \begin{align}\label{smaIIggvae}
      \big\||\g g|+|\g \bn|\big\|_{L^{4,2}(B^4)}+\big\||\g^2 g|+|\g^2 \bn|\big\|_{L^2(B^4)}\le \vae_0.
    \end{align}
    Then we have $\lap_g \bH\in L^{4/3,2}_{\loc}(B^4)$ with the estimate
\begin{align}\label{est:lapHL432}\begin{aligned}
        &\|\lap_g \bH\|_{L^{\frac 43,2}(B_r(0))}\\
        &\le C(\La,\bc,\al)\big( \big\||\g g|+|\g \bn|\big\|_{L^{4,2}(B^4)}+\big\||\g^2 g|+|\g^2 \bn|\big\|_{L^2(B^4)}+r^{\al}\big)\big(\|\g \bH\|_{L^2(B^4)}+\|\bII\|_{L^4(B^4)}\big).
        \end{aligned}
    \end{align}
\end{Th}
\begin{proof}
  Assume \eqref{smaIIggvae} holds with $\vae_0\in(0,1)$ to be determined below. As in~\eqref{defuetabul}, set \begin{align}\label{defuetabul2}
        \vec u\coloneq\vet \,\ovs{\sbul}{\res}_g \vec R + \vet \,\resg S+3\mko d\bH \wres d\bP.
    \end{align}
By Corollary~\ref{co:sysvecu} and~\eqref{sys-LRS}, in $\mca D'\big(B^4,\bwe^2 \R^m \ot \bwe^4 \R^4\big)$ and $\mca D'(B^4,\R^m)$ respectively, we have
\begin{align}
    &d*_gd\vec u=d*_g \vec{\mca R}_3,\label{eq:d*du=d*R3+2}\\[0.2ex]
        &\pi_{\bn}\mkt d^{*_g}( \vec u \,\res d\bP)=\lap_g \bH+ \vec{\mca R}_4.\label{eq:ndotd*uresdp2}
    \end{align}
Moreover, combining~\eqref{remtotptbd},~\eqref{gHL2bd},~\eqref{estLSR}, and~\eqref{ptbdvts2} with Lemma~\ref{lm:lorHold}, we obtain
\begin{align}\label{estRL432}\begin{aligned}
   &\big\|\mko |\vec{\mca R}_3|+|\vec{\mca R}_4|\big\|_{L^{\frac 43,2}(B_{1/2}(0))}\\
   &\le C(\La,\bc)\big(\|\g^2\bP\|_{L^{4,2}(B^4)}+\|\g^3\bP\|_{L^2(B^4)} \big) \big( \big\||\bR|+|S|+|\g \bH|\big\|_{L^{2,\nf}(B_{1/2}(0))}+\|\bH\|_{L^4(B^4)}\big)\\
   &\le C(\La,\bc)\big(\|\g^2\bP\|_{L^{4,2}(B^4)}+\|\g^3\bP\|_{L^2(B^4)}\big)\big(\|\g \bH\|_{L^2(B^4)}+\|\bII\|_{L^4(B^4)}\big).
   \end{aligned}
\end{align}
In particular, this implies that
\begin{align}
\label{d*R3+pRbd}
\begin{aligned}
    \big\|d*_g \vec{\mca R}_3\big\|_{W^{-1,(\frac 43,2)}(B_{1/2}(0))}
    &\le C(\La)  \big\| \vec{\mca R}_3\big\|_{L^{\frac 43,2}(B_{1/2}(0))}\\
    &\le C(\La,\bc)\big(\|\g^2\bP\|_{L^{4,2}(B^4)}+\|\g^3\bP\|_{L^2(B^4)}\big)\big(\|\g \bH\|_{L^2(B^4)}+\|\bII\|_{L^4(B^4)}\big).
    \end{aligned}
\end{align}
In addition, by~\eqref{estLSR} and \eqref{defuetabul2}, we have 
\begin{align}\label{u2nfbd}
    \|\vec u\|_{L^{2,\nf}(B_{1/2}(0))}&\le  C(\La,\bc)\big(\|\g \bH\|_{L^2(B^4)}+\|\bII\|_{L^4(B^4)}\big).
\end{align}
Also, the equation~\eqref{eq:d*du=d*R3+2} gives in $W^{-2,(2,\nf)}_{\loc}\big(B^4,\bwe^2 \R^m\ot \bwe^4 \R^4\big)$ that
\begin{align*}
    \p_{j}\big(g^{ij}(\det g)^{\frac 12} \p_{i}\vec u \big)\,dx^1\we \cdots \we dx^4=d*_g \vec{\mca R}_3.
\end{align*}
Then it follows from Lemma~\ref{lm:ellcacciolp} that $\vec u\in W^{1,(4/3,2)}_{\loc}\big(B^4,\bwe^2\R^m\big)$. Moreover, the elliptic estimate~\eqref{eq:ellcacciop1} together with Remark~\ref{rm:smW1ncoe} implies existence of $\vae=\vae(\La)>0$ such that, if $ \|\g g\|_{L^{4,2}(B^4)}\le \vae$, then for any $\al\in (0,3)$ and all $r\in (0,\frac 14]$, it holds that
\begin{align}\label{guL432bd}
    \|\g \vec u\|_{L^{\frac 43,2}(B_r(0))}
    \le C(\La,\al) \Big(\big\|d*_g \vec{\mca R}_3\big\|_{W^{-1,(\frac 43,2)}(B_{1/2}(0))}+r^{\al}\|\vec u\|_{L^{2,\nf}(B_{1/2}(0))}\Big).
\end{align}
Now we set $\vae_0\coloneq\vae$ and assume~\eqref{smaIIggvae} holds. The estimates~\eqref{d*R3+pRbd}--\eqref{guL432bd} imply that for all $r\in (0,\frac 14]$,
\begin{align}\label{guL432est}
\begin{aligned}
    \|\g \vec u\|_{L^{\frac 43,2}(B_r(0))}
    &\le   C(\La,\bc,\al)\big( \|\g^2\bP\|_{L^{4,2}(B^4)}+\|\g^3\bP\|_{L^2(B^4)}+r^{\al}\big)\big(\|\g \bH\|_{L^2(B^4)}+\|\bII\|_{L^4(B^4)}\big).
\end{aligned}
\end{align}
Then by~\eqref{u2nfbd},~\eqref{guL432est}, and Lemma~\ref{lm:lorHold}, for all $r\in (0,\frac 14]$, we have
\begin{align}
 \label{estndd*ures}\begin{aligned}
     &\big\|\pi_{\bn}\mkt d^{*_g}( \vec u \,\res d\bP)\big\|_{L^{\frac 43,2}(B_r(0))}\\
    &\le   C(\La) \Big( \|\g \vec u\|_{L^{\frac 43,2}(B_r(0))} +\|\g^2 \bP\|_{L^{4,2}(B_r(0))} \|\vec u\|_{L^{2,\nf}(B_r(0))}\Big)\\
    &\le C(\La,\bc,\al)\big( \|\g^2\bP\|_{L^{4,2}(B^4)}+\|\g^3\bP\|_{L^2(B^4)}+r^{\al}\big)\big(\|\g \bH\|_{L^2(B^4)}+\|\bII\|_{L^4(B^4)}\big).
\end{aligned}
 \end{align}    
 Finally, combining~\eqref{eq:ndotd*uresdp2}--\eqref{estRL432} and~\eqref{estndd*ures} gives $\lap_g \bH\in L^{4/3,2}_{\loc}(B^4,\R^m)$. Moreover, applying Remark~\ref{rm:g^2Pdecom}, for all $r\in (0,\frac 14]$, we have
 \begin{align*}
     &\|\lap_g \bH\|_{L^{\frac 43,2}(B_r(0))} \\
     &\le  \big\|\pi_{\bn}\mkt d^{*_g}( \vec u \,\res d\bP)\big\|_{L^{\frac 43,2}(B_r(0))}+\|\vec{\mca R}_4\|_{L^{\frac 43,2}(B^4)} \\
     &\le C(\La,\bc,\al)\big( \|\g^2\bP\|_{L^{4,2}(B^4)}+\|\g^3\bP\|_{L^2(B^4)}+r^{\al}\big)\big(\|\g \bH\|_{L^2(B^4)}+\|\bII\|_{L^4(B^4)}\big)\\
     &\le C(\La,\bc,\al)\big( \big\||\g g|+|\g \bn|\big\|_{L^{4,2}(B^4)}+\big\||\g^2 g|+|\g^2 \bn|\big\|_{L^2(B^4)}+r^{\al}\big)\big(\|\g \bH\|_{L^2(B^4)}+\|\bII\|_{L^4(B^4)}\big).\qedhere
 \end{align*}
\end{proof}
Lemma~\ref{lm:ellcacciolp} together with Theorem~\ref{th:lapHL432} gives the following corollary.
\begin{Co}\label{morestHgn}
     Suppose $\vec{\Phi}\in \mathcal{I}_{1,2}(B^4, \R^m)$ is a weak critical point of $E$ and satisfies~\eqref{weaimmconbP}. Then for any $\beta\in (0,1)$, there exists $\vae_1=\vae_1(\La,\bc,\beta)\in (0,\vae_0)$ such that the following holds for all $r\in (0,1]$: Assume that 
     \begin{equation}\label{smaIIggvae1}
      \big\||\g g|+|\g \bn|\big\|_{L^{4,2}(B^4)}+\big\||\g^2 g|+|\g^2 \bn|\big\|_{L^2(B^4)}\le \vae_1.
    \end{equation}
     Then we have
     \begin{equation}\label{HgHmornor}
         \|\bII\|_{L^4(B_r(0))}+ \|\g \bH\|_{L^{2}(B_r(0))}\le C(\La,\bc,\beta) \mko r^\beta \big(\|\bII\|_{L^4(B^4)}+ \|\g \bH\|_{L^{2}(B^4)}\big).
     \end{equation}
\end{Co}
\begin{proof}
 For $0<r\le 1$, we set \begin{align}\label{defN1r}\begin{aligned}
          &\mca N_1(r)\coloneq  \|\bH\|_{L^4(B_{r}(0))}+ \|\g \bH\|_{L^{2}(B_{r}(0))},\\[0.3ex]
          & \mca N_2(r)\coloneq  \|\bII\|_{L^4(B_{r}(0))}+ \|\g \bH\|_{L^{2}(B_{r}(0))}.
          \end{aligned}
      \end{align}
 Assume \eqref{smaIIggvae1} holds with $\vae_1\in (0,\vae_0)$ to be determined below. Then the estimate~\eqref{est:lapHL432} (with $\al=1$) implies that for all $0\le  r\le \frac 14$, there holds
  \begin{align}\label{bdlapHN2}
      \|\lap_g \bH\|_{L^{\frac 43,2}(B_r(0))}
      \le C(\La,\bc)\mko( \vae_1+r)\mko\mca N_2(1).
  \end{align}
  By Lemma~\ref{lm:LpembW-1} and~\eqref{*estlor}, we have
    \begin{align}
    \label{lapHW-12bd}\begin{aligned}
        \big\|\p_i\big(g^{ij}(\det g)^{\frac 12}\p_j \bH\big)\big\|_{W^{-1,2}(B^4)}
        &= \|*_g \lap_g \bH\|_{W^{-1,2}(B^4)}\\
         &\le C(\La) \|\lap_g \bH\|_{W^{-1,2}(B^4)}\\
        &\le C(\La) \|\lap_g \bH\|_{L^{4/3,2}(B^4)}.
        \end{aligned}
    \end{align}
    Then by Lemma~\ref{lm:ellcacciolp} (with $\alpha=1$), \eqref{lapHW-12bd} and Remark~\ref{rm:smW1ncoe}, there exists $\vae=\vae(\La)>0$ such that, if $
               \|\g g\|_{L^{4,2}(B^4)}\le \vae$, then for all $r\in (0,\frac 12]$, it holds that
    \begin{equation}\label{HgHlapH}
          \mca N_1(r)\le C(\La)\big(  \|\lap_g \bH\|_{L^{4/3,2}(B^4)}+r\mko \mca N_1(1)\mko \big).
           \end{equation}
We choose $\vae_1\in \big(0,\min(\vae,\vae_0)\mko\big)$ so that~\eqref{HgHlapH} applies under the assumption~\eqref{smaIIggvae1}. Next, we estimate $\mca N_2(r)$. By the approximation result for weak immersions in~\cite{LaMaRi25}, Lemma~\ref{lm:Hdnstrequ} remains valid for the weak immersion $\bP$. Hence, we have
     \begin{align*}
          d\bn+4\mko d\bP\we (\bn\,\res \bH)= \frac 12 *_g\star\mko \Big( \big(d\bn\sbul^{(m-5)} \bn\big) \ovs{\ovwe}{\we}d\bP\ovs{\ovwe}{\we}d\bP\Big).
     \end{align*}
     Applying $d\,*_g$ to the above identity and defining $\ovs{\sbul}{\we}^{(m-5)}$ as in~\eqref{not-overset} yields
     \begin{align}\label{d*str4dnH}
      d*_g d\mko \bn+4\mko d*_g\big(d\bP\we (\bn\,\res \bH)\mko\big)=\frac{1}{2}\star \lf( \Big(d\bn\ovs{\sbul}{\we}^{(m-5)} d\bn\Big)\ovs{\ovwe}{\we}d\bP\ovs{\ovwe}{\we}d\bP\rg).
     \end{align}
     Hence by~\eqref{smaIIggvae1} and the embedding $L^2(B^4)\hookrightarrow W^{-1,4}(B^4)$, we obtain
     \begin{align}\label{lapnW-14}
     \begin{aligned}
         \|d*_gd\mko \bn\|_{W^{-1,4}(B^4)}
         &\le 4\mko \Big\|d*_g\big(d\bP\we (\bn\,\res \bH)\mko\big)\Big\|_{W^{-1,4}(B^4)}+C\mko\Big\| \Big(d\bn\ovs{\sbul}{\we}^{(m-5)} d\bn\Big)\ovs{\ovwe}{\we}d\bP\ovs{\ovwe}{\we}d\bP \Big\|_{L^2(B^4)}\\
         &\le C(\La)\big(\|\bH\|_{L^4(B^4)}+\vae_1\|\bII\|_{L^4(B^4)} \big).
      \end{aligned}
       \end{align}
     Fix $\beta\in (0,1)$ and $\beta_0\coloneq \frac{\beta+1}2$. As in~\eqref{HgHlapH}, combining Lemma~\ref{lm:ellcacciolp} (with $\al=\beta_0$) and~\eqref{lapnW-14} with Remark~\ref{rm:smW1ncoe}, after possibly decreasing $\vae_1=\vae_1(\La,\bc,\beta)>0$, we have for all $r\in (0,1]$,
     \begin{align}\label{gnL4est}
         \begin{aligned}
         \|\bII\|_{L^4(B_r(0))}
         &\le C(\La)\mko\|\g\bn\|_{L^4(B_r(0))}\\
         &\le C(\La,\beta) \big(   \|d*_gd\mko \bn\|_{W^{-1,4}(B^4)}+r^{\beta_0}\mko \|\bII\|_{L^4(B^4)} \big)\\[0.2ex]
             &\le C(\La,\beta)\big(\|\bH\|_{L^4(B^4)}+(\vae_1+r^{\beta_0})\|\bII\|_{L^4(B^4)} \big).
         \end{aligned}
     \end{align} 
     It follows that 
     \begin{equation}\label{bdN2N1}
         \mca N_2(r)\le C(\La,\beta)\big(\mca N_1(1)+(\vae_1+r^{\beta_0}) \mca N_2(1)\mko\big).
     \end{equation}
    Applying the dilation~\eqref{dilgnH}--\eqref{gggndilinv} to the estimates~\eqref{bdlapHN2},~\eqref{HgHlapH}, and~\eqref{bdN2N1}, for all $0<4r\le \rho\le1$, we have
    \begin{subnumcases}{}
        \|\lap_g \bH\|_{L^{ 4/3,2}(B_r(0))}
      \le C(\La,\bc)\mko( \vae_1+r/\rho)\mko\mca N_2(\rho),\label{lapHN2}\\[0.3ex]
       \mca N_1(r)\le C(\La)\big(  \|\lap_g \bH\|_{L^{4/3,2}(B_\rho(0))}+(r/\rho) \mko \mca N_1(\rho)\mko \big),\label{N1lapHN1}\\[0.3ex]
       \mca N_2(r)\le C(\La,\beta)\Big(\mca N_1(\rho)+\big(\vae_1+(r/\rho)^{\beta_0}\big) \mca N_2(\rho)\Big).\label{N2N1+N2}
    \end{subnumcases}
      Fix a positive integer $s$ such that $\beta<\frac{(s-2)\beta_0}{s}$. By~\eqref{N1lapHN1} and~\eqref{N2N1+N2}, we obtain that for all $r\in (0,\frac 14]$, 
\begin{align}\label{HgHleer}\begin{aligned}
          &\mca N_2(r^s)+\mca N_1(r^{s-1})\\[0.3ex]
          &\le C(\La,\beta) \big(\mca N_1(r^{s-1})+(\vae_1+r^{\beta_0})\mca  N_2(r^{s-1})\mko\big)\\[0.3ex]
          &\le  C(\La,\beta)\lf(  \|\lap_g \bH\|_{L^{4/3,2}\lf(B_{r^{s-2}}(0)\rg)}+r\mkt \mca N_1(r^{s-2})+(\vae_1+r^{\beta_0})\mca  N_2(r^{s-1})\rg).
          \end{aligned}
    \end{align} 
    Since $s$ depends only on $\beta$, by~\eqref{lapHN2} and induction, we obtain a constant $C_0=C_0(\La,\bc,\beta)>0$ such that 
    \begin{align} 
          \mca N_2(r^s) +\mca N_1(r^{s-1})&\le C_0\big(\vae_1+r^{(s-2)\beta_0} \big) \mca N_2(1).
     \end{align}
     Since $\beta<\frac{(s-2)\beta_0}{s}$, we can choose $r_0=r_0(\La,\bc,\beta)\in (0,\frac 14]$ such that 
     \begin{align}\label{C0r0albeta}
         C_0 \mko r_0^{(s-2)\beta_0}\le \frac 12\mko r_0^{s\beta}. 
        \end{align}
     Now we set $r_1\coloneq r_0^s$ and $\vae_1=\vae_1(\La,\bc,\beta)>0$ such that $\vae_1\le\min \big\{\vae,\vae_0, \frac{r_1^\beta}{2C_0}\big\}$. Then the estimate~\eqref{HgHleer} implies that under the assumption~\eqref{smaIIggvae1}, we have
     \begin{align}\label{HgHlerbe}
          \mca N_2(r_1)\le r_1^\beta \mca N_2(1).
     \end{align}
    Finally, using the dilation~\eqref{dilgnH}--\eqref{gggndilinv}, we may apply the estimate~\eqref{HgHlerbe} to $B_{r_1}(0)$, $B_{r_1^2}(0)$, $B_{r_1^3}(0),\dots$ For arbitrary $r\in (0,1]$ with $r_1^{k}< r \le r_1^{k-1} $, $k\in \N^+$, we obtain from \eqref{HgHlerbe} that
     \begin{align*}
          \mca N_2(r)&\le  r_1^{\beta(k-1)}\mca N_2(1)
           \le r_1^{-\beta}\mko r^{\beta}\mca N_2(1).
     \end{align*}
     This completes the proof.
\end{proof}

Combining the Morrey estimates for $\bH,\,\g \bH,$ and $\lap_g \bH$ with the Euler--Lagrange equation, we are now ready to prove the main theorem.
\begin{hproof5}
Let $\bP_0\in \mathcal{I}_{1,2}(\Si^4,\R^m)$ be a critical point of $E_{\bc}$. By~\cite[Thm.~4.1]{MarRiv2}, for any $p\in \Si^4$, there exist $g_{\bP_0}$-harmonic coordinates on a neighborhood of $p$ such that, under these coordinates, it holds that $g^{ij}\in W^{2,(2,1)}\hookrightarrow C^0$. For this reason, it suffices to consider $\bP\in \mathcal{I}_{1,2}(B^4,\R^m)$ being a critical point of $E_{\bc}$ and satisfying~\eqref{weaimmconbP} and~\eqref{smaIIggvae}, and the coordinate functions $\{x^j\}$ are $g_{\bP}$-harmonic. Then we have
\begin{equation}\label{harcoreq}
    \forall\, 1\le j\le 4, \qquad \p_i(g^{ij}\big(\det g)^{1/2} \big)=0.
\end{equation}
By Corollary~\ref{morestHgn}, for any $\beta\in (0,1)$, we have
\begin{align}\label{morHgHapp}
       \sup_{x\in B_{1/2}(0), r\le \frac 12} r^{-\beta} \big(\|\bII\|_{L^4(B_{r}(x))}+\|\g \bH\|_{L^2(B_r(x))}\big)<\nf.
\end{align}
The standard Riesz potential estimate (see \cite[Thm.~3.2]{Adams75}) then gives $\bH\in L^{\frac {4-2\beta}{1-\beta}}(B_{1/4}(0))$. Since the preceding argument is local and $\beta\in (0,1)$ is arbitrary, a covering implies $\bH\in L_{\loc}^p(B^4)$ for any $p<\nf$. Now by~\eqref{harcoreq}, we have
\begin{align}\label{ellgraf=H}
    4\mko \bH=\lap_g\bP=g^{ij} \p_i\p_j\bP.
\end{align}
Since $g^{ij}\in C^0(B^4)$, by~\eqref{ellgraf=H} and \cite[Thm.~4.1]{Chiarenza93}, we obtain that
\begin{align}\label{phiW2p}
    \forall\, p<\nf,\qquad \bP\in W^{2,p}_{\loc}(B^4,\R^m).
\end{align} 
Combining Theorem~\ref{th:lapHL432} with~\eqref{morHgHapp}, we have
\begin{align*}
       \sup_{x\in B_{1/2}(0), r\le \frac 1{16}} r^{-\beta} \|\lap_g \bH\|_{L^{\frac 43,2}(B_r(x))}\le C(\La,\bc)  \sup_{x\in B_{1/2}(0), r\le \frac 18}r^{-\beta} \big(\|\bII\|_{L^4(B_{4r}(x))}+\|\g \bH\|_{L^2(B_{4r}(x))}\big)<\nf.
\end{align*}
The Riesz potential estimates \cite[Thm.~3.1]{Adams75} and~\cite[Eqs.~(1.3)--(1.5)]{Mingi11} then imply $*_g\,\lap_g \bH\in W^{-1,q_0}(B_{1/2}(0))$ for some $q_0\in (2,4)$. Since the preceding argument is local, a covering implies $*_g\,\lap_g\bH\in W^{-1,q_0}_{\loc}(B^4)$, that is,
\begin{align}\label{lapHW-1s0}
    \p_i\big(g^{ij}(\det g)^{\frac 12}\p_j \bH\big)\in W^{-1,q_0}_{\loc}(B^4,\R^m).
\end{align}
Since $g^{ij}\in  C^0(B^4)$, by~\cite[Thm.~3.1]{Byun05} we have $\bH\in W^{1,q_0}_{\loc}(B^4)$. Combining~\cite[Thm.~2.4.4]{Maugeri00} with~\eqref{phiW2p} then gives
\begin{align}\label{phiW3s0}
    \bP\in W^{3,q_0}_{\loc}(B^4,\R^m).
\end{align}
Now we return to Proposition~\ref{sysestLSR} and repeat the argument using the improved regularities established above. 
Let $\vec V$ be as in~\eqref{defVmase}. Arguing as in~\eqref{dpdHdHL1}--\eqref{*VW-22+L1est} yields
    $*_g\,\vec V\in W^{-2,q_0}_{\loc}(B^4)$.  Retaining the definitions of $\vec L$, $S$, $\bR$ from~\eqref{sys-LRS}, and applying Corollaries~\ref{co-Hod-decw-1p}--\ref{Co:Hod-declp}, we obtain that $\bL\in W^{-1,q_0}_{\loc}(B^4)$ and $S,\bR\in L^{q_0}_{\loc}(B^4)$. 
    By~\eqref{remtotptbd}, ~\eqref{ptbdvts2} (as used in the proof of Theorem~\ref{th:lapHL432}) and~\eqref{phiW2p}, we deduce $ \mca {\vec R}_3, \vec{\mca R}_4\in L^q_{\loc}(B^4)$  for all $q\in (2,q_0)$. \\
    
 The equations~\eqref{eq:d*du=d*R3+2}--\eqref{eq:ndotd*uresdp2} then imply that for any $q\in (2,q_0)$, we have $\vec u\in W^{1,q}_{\loc}(B^4)$ and \begin{equation}\label{lapHLs}
     \lap_g\mko \bH\in L_{\loc}^q(B^4,\R^m).
\end{equation}
In particular, the Sobolev embedding gives $\lap_g\bH\in W^{-1,q}_{\loc}(B^4)$ for any $q\in (2,4)$. By repeating the argument from~\eqref{lapHW-1s0} to~\eqref{lapHLs}, we obtain $\lap_g\bH\in L^q_{\loc}(B^4)$ for any $q\in (2,4)$. 
Since~\eqref{phiW2p} implies $g\in W^{1,p}(B^4)$ for any $p<\nf$ and $\g \bH\in L^{q}(B^4)$, by the proof of \cite[Thm.~8.8]{Gilbarg01}, we obtain $\bH\in W^{2,2}_{\loc}(B^4)$. 
The harmonic coordinate condition~\eqref{harcoreq} implies that for any $q\in (2,4)$,
$$
g^{ij}\mko\p_i\p_j \bH= \lap_g\bH\in L^q_{\loc}(B^4,\R^m).$$
Hence by \cite[Thm.~4.1]{Chiarenza93}, we have $\bH\in W^{2,q}_{\loc}(B^4)$ for any $q\in (2,4)$. 
It follows that $\g \bH \in L^p_{\loc}(B^4)$ for any $p<\nf$. Arguing as in~\eqref{phiW3s0}, we obtain 
\begin{align}\label{phiW3p}
    \forall\, p<\nf,\qquad \bP\in W^{3,p}_{\loc}(B^4,\R^m).
\end{align}  
Since $\bH\in W^{2,q}_{\loc}(B^4)$ for any $q\in (2,4)$, differentiating~\eqref{ellgraf=H} and combining~\cite[Thm.~2.4.4]{Maugeri00} with~\eqref{phiW3p} further yields \begin{equation}\label{phiW4s}
    \forall\, 2<q<4,\qquad \bP\in W^{4,q}_{\loc}(B^4,\R^m).
\end{equation}\\
Fix $q_1\in (2,4)$. Now we prove by induction on $k\ge 4$ that $\bP\in W^{k,q_1}_{\loc} (B^4)$ for all $k\in \N$. Assume $k\ge 4$ and $\bP\in W^{k,q_1}_{\loc}(B^4)$. We use the Euler--Lagrange equation in~\eqref{ELequ}, which is equivalent to
\begin{align}\label{equicondia}
    d*_g d\big(\lap_g \bH\big)=-2\mko d*_g \bigg(\vec l_0+\sum_{s=1}^8c_s\vec l_s \bigg)\qquad \text{in }\mca D'\big(B^4,\R^m\ot \bwe^4\R^4\big).
\end{align}
The induction hypothesis gives $\bII_{ij}\in W^{k-2,q_1}_{\loc}(B^4)\hookrightarrow W^{k-3,\frac{4q_1}{4-q_1}}_{\loc}(B^4)$, hence by~\eqref{defva},~\eqref{condvrsls} and the polynomial structure as in~\eqref{defl4}, we have
\begin{equation*}
    \vec l_0+\sum_{s=1}^8c_s\vec l_s\in W^{k-4,q_1}_{\loc}(B^4,\R^m\ot \R^4).
\end{equation*}
Since $g^{ij}\in W^{k-1,q_1}_{\loc}(B^4)\hookrightarrow C^{k-3}_{\loc}(B^4)$ by induction hypothesis, we obtain that the right-hand side of~\eqref{equicondia} lies in $W^{k-5,q_1}_{\loc}(B^4)$. Let $\ga_1$ be a multi-index with $|\ga_1|=k-4$. Since $\lap_g \bH \in W^{k-4,q_1}_{\loc}(B^4)$ and $g^{ij}\in W^{k-1,q_1}_{\loc}(B^4)$, applying $\p^{\ga_1}$ to~\eqref{equicondia} yields 
$$
    d*_g d\big(\p^{\ga_1}(\lap_g \bH)\mko\big)\in W^{-1,q_1}_{\loc}\big(B^4, \R^m\ot \bwe^4\R^4\big).
$$ 
Combining~\cite[Thm.~4.1]{laMan20} with~\cite[Thm.~1.5]{Byun05} then implies $\p^{\ga_1}( \lap_g \bH)\in W^{1,q_1}_{\loc}(B^4)$. Hence, we have 
\begin{equation} \label{lapHWk-3q}
    g^{ij}\mko\p_i\p_j\bH= \lap_g \bH\in W^{k-3,q_1}_{\loc}(B^4,\R^m).
\end{equation}
Since $\bH\in W^{k-2,q_1}_{\loc}(B^4)$ and $g^{ij}\in C^{k-3}_{\loc}(B^4)$, differentiating~\eqref{lapHWk-3q} $(k-4)$ times and applying~\cite[Thm.~2.4.4]{Maugeri00} yields $\bH\in W^{k-1,q_1}_{\loc}(B^4)$. Finally, by differentiating~\eqref{ellgraf=H} $(k-2)$ times and invoking~\cite[Thm.~2.4.4]{Maugeri00} again, we obtain 
$$
    \bP\in W^{k+1,q_1}_{\loc}(B^4,\R^m).
$$ 
This completes the induction and thus we have $\bP\in C^\nf(B^4)$.

Since~\eqref{ellgraf=H} together with~\eqref{equicondia} (where $g_{ij}=\p_i\bP\cdot \p_j\bP$) forms an analytic elliptic system for $\bP$ and $\bH$ of the type considered in~\cite{Morrey58}, it follows that $\bP$ is real-analytic. This completes the proof.
\end{hproof5}

\end{document}